\documentclass[11pt]{amsart}
\usepackage[utf8]{inputenc}

\usepackage{overpic, amssymb, times, graphicx, tikz-cd}
\usepackage[all]{xy}
\usepackage[backref=page]{hyperref}
\usepackage{verbatim}
\hypersetup{
	colorlinks=true,
	linkcolor=blue,
	filecolor=blue,      
	urlcolor=blue,
	citecolor=blue,
}
\DeclareMathOperator{\Flow}{Flow}
\DeclareMathOperator{\Diff}{Diff}
\DeclareMathOperator{\Symp}{Symp}
\DeclareMathOperator{\Int}{Int}
\DeclareMathOperator{\Span}{Span}

\DeclareMathOperator{\Id}{Id}
\DeclareMathOperator{\tb}{tb}
\newcommand{\Rthree}{(\mathbb{R}^{3},\xi_{std})}
\newcommand{\Sthree}{(S^{3},\xi_{std})}

\newcommand{\Mxi}{(M,\xi)}
\newcommand{\AOB}{(\Sigma,\Phi)}

\newcommand{\Lie}{\mathcal{L}}

\newcommand{\Sstd}{(S^{2n+1},\xi_{std})}
\newcommand{\Ldom}{(W,\lambda)}
\newcommand{\Sdom}{(\Sigma,\beta)}

\newcommand{\disk}{\mathbb{D}}
\newcommand{\MOB}{M_{(\Sigma,\Phi)}}

\newcommand{\MCL}{\mathcal{L}}
\newcommand{\SympGroup}{\Symp((\Sigma,d\beta),\partial\Sigma)}
\newcommand{\Cinfty}{C^{\infty}}
\newcommand{\Mxink}[2]{(P_{#1, #2}, \zeta_{#1, #2})}
\newcommand{\half}{\frac{1}{2}}

\newcommand{\be}{\begin{enumerate}}
\newcommand{\ee}{\end{enumerate}}

\numberwithin{equation}{subsection}

\newtheorem{thm}{Theorem}[section]

\newtheorem{prop}[thm]{Proposition}

\newtheorem{defn}[thm]{Definition}

\newtheorem{lemma}[thm]{Lemma}
\newtheorem{cor}[thm]{Corollary}
\newtheorem{conj}[thm]{Conjecture}
\newtheorem{q}[thm]{Question}
\newtheorem{rmk}[thm]{Remark}

\newtheorem{ex}[thm]{Example}

\begin{document}

\title{Liouville hypersurfaces and connect sum cobordisms}

\author{Russell Avdek}
\email{russell.avdek@gmail.com}
\date{\today}

\begin{abstract}
The purpose of this paper is to introduce \emph{Liouville hypersurfaces} in contact manifolds, which generalize ribbons of Legendrian graphs and pages of supporting open books.  Liouville hypersurfaces are used to define a gluing operation for contact manifolds called the \emph{Liouville connect sum}.  Performing this operation on a contact manifold $\Mxi$ gives an exact -- and in many cases, Weinstein -- cobordism whose concave boundary is $\Mxi$ and whose convex boundary is the surgered manifold.  These cobordisms are used to establish the existence of ``fillability'' and ``non-vanishing contact homology'' monoids in symplectomorphism groups of Liouville domains, study the symplectic fillability of a family of contact manifolds which fiber over the circle, associate cobordisms to certain branched coverings of contact manifolds, and construct exact symplectic cobordisms that do not admit Weinstein structures. 

The Liouville connect sum generalizes the Weinstein handle attachment and is used to extend the definition of contact $(1/k)$-surgery along Legendrian knots in contact 3-manifolds to contact $(1/k)$-surgery along Legendrian spheres in contact manifolds of arbitrary dimension.  We use contact surgery to construct exotic contact structures on $5$- and $13$-dimensional spheres after establishing that $S^{2}$ and $S^{6}$ are the only spheres along which generalized Dehn twists smoothly square to the identity mapping.  The exoticity of these contact structures implies that Dehn twists along $S^{2}$ and $S^{6}$ do not symplectically square to the identity, generalizing a theorem of Seidel. A similar argument shows that the $(2n+1)$-dimensional contact manifold determined by an open book whose page is $(T^{*}S^{n}, -\lambda_{can})$ and whose monodromy is any negative power of a symplectic Dehn twist is not exactly fillable.
\end{abstract}

\maketitle
\setcounter{tocdepth}{1}
\tableofcontents
\newpage

\section{Introduction}

\subsection{Preliminaries}\label{Sec:Preliminaries}

A \emph{contact manifold} is a pair $\Mxi$ where $M$ is an oriented $(2n+1)$-dimensional manifold and $\xi$ is a globally cooriented $(2n)$-plane field on $M$ such that there is a 1-form $\alpha\in\Omega^{1}(M)$ satisfying
\begin{equation*}
\ker(\alpha)=\xi\quad\text{and}\quad \alpha\wedge(d\alpha)^{n}> 0
\end{equation*}
with respect to the orientation on $M$.  We also say that $\xi$ is a \emph{contact structure} on $M$.  A 1-form $\alpha$ satisfying the above equation is a \emph{contact form for $\Mxi$}.

An oriented, codimension-1 submanifold $M$ of a symplectic manifold $(W,\omega)$ is a \emph{contact hypersurface} \cite{Weinstein:Conjecture} if there is a neighborhood $N(M)$ of $M$ such that $\omega=d\lambda$ for some $\lambda\in\Omega^{1}(N(M))$ and the vector field $X$ determined by $\omega(X,\ast)=\lambda$ is positively transverse to $M$.  This implies that $\lambda|_{TM}$ is a contact form on $M$.  In this paper, we will be primarily concerned with the following class of symplectic manifolds whose boundaries are contact hypersurfaces:

\begin{defn}\label{Def:Ldom}
A \emph{Liouville domain} is a pair $\Sdom$ where
\be
\item $\Sigma$ is a smooth, compact manifold with boundary,
\item $\beta\in\Omega^{1}(\Sigma)$ is such that $d\beta$ is a symplectic form on $\Sigma$, and
\item the unique vector field $X_{\beta}$ satisfying $d\beta(X_{\beta},\ast) = \beta$ points out of $\partial \Sigma$ transversely.
\ee
The vector field $X_{\beta}$ on $\Sigma$ described above is called the \emph{Liouville vector field} for $\Sdom$.
\end{defn}

We say that two Liouville 1-forms $\beta$ and $\beta'$ are \emph{homotopic} if there is a smooth family $\beta_{t}$, $t\in[0,1]$, of Liouville 1-forms on $\Sigma$ with $\beta_{0}=\beta$ and $\beta_{1}=\beta'$, such that $\beta_{t} = \beta$ for all $t$ on some neighborhood of $\partial \Sigma$.

\begin{ex}\label{Ex:Sstd}
Denote by $\disk^{2n+2}$ the unit disk in $\mathbb{R}^{2n+2}$.  The \emph{standard 1-form} on $\disk^{2n+2}$ is
\begin{equation*}
\lambda_{std} = \frac{1}{2}\sum_{1}^{n+1} (x_{j}dy_{j} - y_{j}dx_{j})
\end{equation*}
in terms of coordinates $(x_{1},\dots,x_{n+1},y_{1},\dots,y_{n+1})$.  The \emph{standard contact sphere}, denoted $\Sstd$, is the boundary of $\disk^{2n+2}$ with $\xi_{std}=\ker(\lambda_{std}|_{TS^{3}})$.
\end{ex}

\subsection{Liouville submanifolds of contact manifolds}\label{Sec:Hypersurface}

We would like to define a class of codimension $1$ submanifolds of contact manifolds formally analogous to contact hypersurfaces in symplectic manifolds.  One natural candidate definition would be that of a \emph{convex hypersurface} introduced by Giroux in \cite[\S 1.3]{Giroux:Convex} and reviewed in Section \ref{Sec:Convex}.  In this paper we study the following more restricted class of hypersurfaces in contact manifolds:

\begin{defn}\label{Def:Hypersurface}
Let $\Mxi$ be a $(2n+1)$-dimensional contact manifold and let $\Sdom$ be a $2k$-dimensional Liouville domain. A \emph{Liouville embedding} $i:\Sdom\rightarrow \Mxi$ is an embedding $i:\Sigma\rightarrow M$ such that there exists a contact form $\alpha$ for $\Mxi$ for which $i^{*}\alpha=\beta$. The image of a Liouville embedding will be called a \emph{Liouville submanifold} and will be denoted by $\Sdom\subset\Mxi$.  When $k=n$, we say that $\Sdom\subset\Mxi$ is a \emph{Liouville hypersurface} in $\Mxi$.
\end{defn}

\begin{rmk}
In Section \ref{Sec:HighCodimension} it is shown that every Liouville submanifold in a contact manifold $\Mxi$ can be realized as the zero section of a symplectic disk bundle whose total space is a Liouville hypersurface in $\Mxi$.
\end{rmk}

Definition \ref{Def:Hypersurface} implies that the boundary $\partial\Sigma$ of a Liouville hypersurface $\Sdom\subset\Mxi$ is a codimension 2 contact submanifold of $\Mxi$ when oriented as the boundary of $\Sigma$.  For example, when $\Mxi$ is 3-dimensional, the boundary of a Liouville hypersurface is a (positive) transverse link.  Intuitively, we think of  Liouville hypersurfaces in a contact manifold $\Mxi$ as positive regions of convex hypersurfaces in $\Mxi$.  This will be made more precise in Proposition \ref{Prop:PositiveRegion}.

\begin{figure}[h]
	\begin{overpic}[scale=.55, angle=-90]{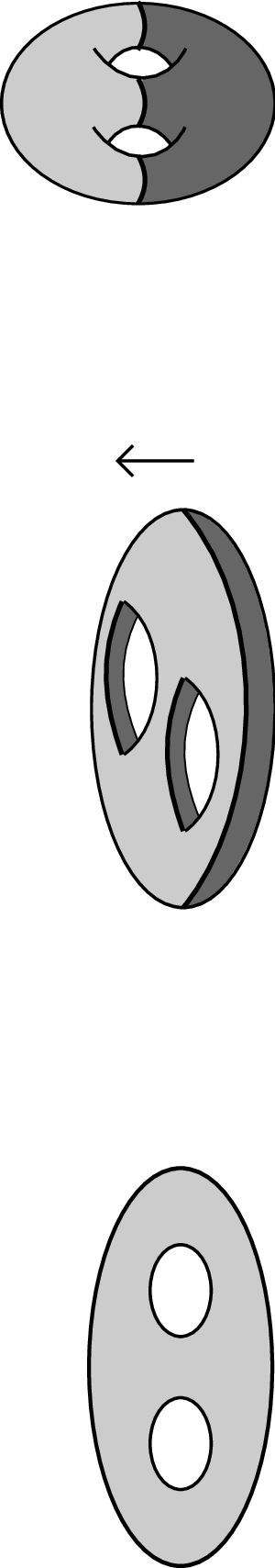}
        \put(12,-3.5){$\Sigma$}
        \put(11.5,-6.5){(1)}
        \put(53,-6.5){(2)}
        \put(91.5,-6.5){(3)}
        \put(52,-3.5){$\mathcal{N}(\Sigma)$}
        \put(90,-3.5){$\partial \mathcal{N}(\Sigma)$}
        \put(70,2){$\partial_{z}$}
    \end{overpic}
    \vspace{8.5mm}
	\caption{Moving from left to right we have (1) a Liouville hypersurface $\Sdom$ represented by a pair of pants, (2) a neighborhood $\mathcal{N}(\Sigma)$ of $\Sigma$ represented by a handlebody, and (3) $\partial \mathcal{N}(\Sigma)$ depicted as an abstract surface.  In schematic figures Liouville domains and hypersurfaces will be represented by pairs-of-pants unless otherwise stated.  Whenever we draw convex hypersurfaces, we lightly shade the positive regions and heavily shade the negative regions.  See Section \ref{Sec:Convex} for further explanation.}
    \label{Fig:StandardNbhd}
\end{figure}
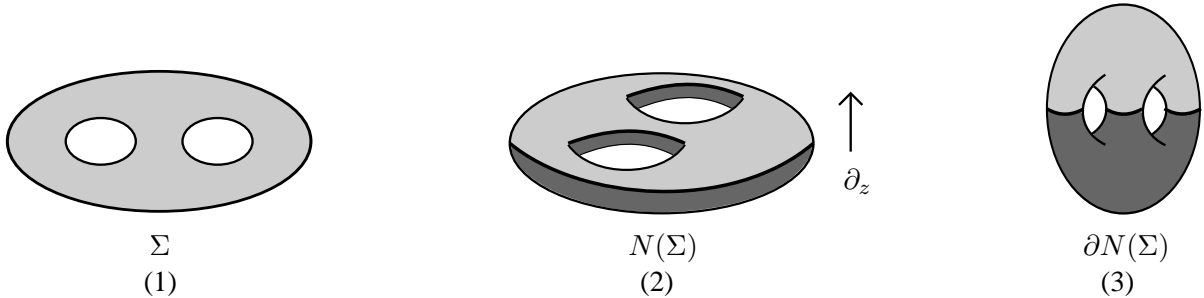

Every Liouville hypersurface $\Sdom\subset \Mxi$ admits a neighborhood of the form
\begin{equation*}
N(\Sigma)=[-\epsilon,\epsilon]\times\Sigma \quad\text{on which}\quad \alpha = dz + \beta
\end{equation*}
where $z$ is a coordinate on $[-\epsilon,\epsilon]$.  After rounding the edges $(\partial[-\epsilon,\epsilon])\times\partial\Sigma$ of $[-\epsilon,\epsilon]\times\Sigma$, we obtain a neighborhood $\mathcal{N}(\Sigma)$ of $\Sigma$ for which $\partial \mathcal{N}(\Sigma)$ is a smooth convex surface in $\Mxi$ with contact vector field 
\begin{equation}\label{Eq:Vbeta}
V_{\beta} = z\partial_{z} + X_{\beta}
\end{equation}
and dividing set $\lbrace 0 \rbrace \times \partial\Sigma$. See Figure \ref{Fig:StandardNbhd}. Details of the edge rounding described appear in Section \ref{Sec:StandardNeighborhood}.

\begin{ex}[Neighborhoods of isotropic submanifolds as Liouville hypersurfaces]\label{Ex:IsotropicNeighborhood}
If 
\begin{equation*}
\Sdom=(\disk^{2n},\lambda_{std})\subset\Mxi
\end{equation*}
 is a Liouville submanifold of a $(2n+1)$-dimensional contact manifold, then the interior of $\mathcal{N}(\Sigma)$ is a Darboux ball.  If $L\subset \Mxi$ is an isotropic sphere with trivial normal bundle and $\alpha$ is a contact form for $\Mxi$, then we can find a compact hypersurface $\Sigma$ with non-empty boundary in $\Mxi$ which deformation retracts onto $L$ and is diffeomorphic to a tubular neighborhood of the zero section of the bundle $\mathbb{R}^{2m}\oplus T^{*}L\rightarrow L$ for which $\alpha|_{T\Sigma}=\lambda_{std}-\lambda_{can}$.  Here $m+ \dim (L)=n$ and $\lambda_{can}$ is the canonical 1-form on $T^{*}L$ described in Example \ref{Ex:CotangentBundle}, below.

Similar statements hold without the assumptions that the normal bundle of $L$ is trivial or that $L$ is a sphere, by the results in Section \ref{Sec:HighCodimension}.  The case $\Sdom=(\disk^{2n},\lambda_{std})$ corresponds to the case where $L$ is a single point.  See Examples \ref{Ex:Sstd} and \ref{Ex:CotangentBundle}.
\end{ex}

\subsection{The Liouville connect sum and associated cobordisms}\label{Sec:ConnectSum}

Convex hypersurfaces provide a simple method of constructing contact manifolds by cut-and-paste.  However, examples are hard to find in high ($>3$) dimensional contact manifolds and it is notoriously difficult to determine how geometric properties of contact structures -- such as symplectic fillability, or tightness -- behave under convex gluing.  See, for example, \cite{Honda:Gluing, Wand:Tightness}.

Using Liouville hypersurfaces, we introduce a special type of convex gluing for contact manifolds called the \emph{Liouville connect sum}.  Theorem \ref{Thm:Cobordism} shows that this gluing operation determines an exact symplectic cobordism whose negative boundary is $\Mxi$ and whose positive boundary is the surgered manifold $\#_{\Sdom}\Mxi$, allowing us to relate symplectic filling properties of $\#_{\Sdom}\Mxi$ to those of $\Mxi$.

\subsection{Outline of the main construction}\label{Sec:SurgeryOutline}

In this section we define the Liouville connect sum and state Theorem \ref{Thm:Cobordism} from which most of our other results will be derived.

Fix a $(2n)$-dimensional Liouville domain $\Sdom$ and a (possibly disconnected) $(2n+1)$-dimensional contact manifold $\Mxi$.  Let $i_{1}$ and $i_{2}$ be Liouville embeddings of $\Sdom$ into $\Mxi$ whose images, which we will denote by $\Sigma_{1}$ and $\Sigma_{2}$, are disjoint.  Let $\alpha$ be a contact form for $\Mxi$ satisfying 
\begin{equation*}
i_{1}^{\ast}\alpha=i_{2}^{\ast}\alpha=\beta.
\end{equation*}

Consider neighborhoods $\mathcal{N}(\Sigma_{1}),\mathcal{N}(\Sigma_{2})\subset M$ as described in Section \ref{Sec:Hypersurface}.  Taking coordinates $(z,x)$ on each such neighborhood with $x\in \Sigma$, we may consider the mapping
\begin{equation*}\label{Eq:ConnectSum}
\Upsilon:\partial \mathcal{N}(\Sigma_{1})\rightarrow \partial \mathcal{N}(\Sigma_{2}),\quad \Upsilon(z,x)=(-z,x).
\end{equation*}
The map $\Upsilon$ sends
\be
\item the positive region of $\partial \mathcal{N}(\Sigma_{2})$ to the negative region of $\partial \mathcal{N}(\Sigma_{1})$,
\item the negative region of $\partial \mathcal{N}(\Sigma_{1})$ to the positive region of $\partial \mathcal{N}(\Sigma_{2})$, and
\item the dividing set of $\partial \mathcal{N}(\Sigma_{1})$ to the dividing set of $\partial \mathcal{N}(\Sigma_{2})$
\ee
in such a way that we may perform a convex gluing.  In other words, the map $\Upsilon$ naturally determines a contact structure $\#_{(\Sdom,(i_{1},i_{2}))}\xi$ on the manifold
\begin{equation*}
\#_{(\Sigma,(i_{1},i_{2}))} M:= \Bigl( M \setminus \Int\bigl(\mathcal{N}(\Sigma_{1})\cup \mathcal{N}(\Sigma_{2})\bigr)\Bigr) /\sim
\end{equation*}
where $p\sim \Upsilon(p)$ for $p\in \partial \mathcal{N}(\Sigma_{1})$. A careful construction of the neighborhood $\mathcal{N}(\Sigma)$ as well as the normalizations of the contact forms required to perform the convex gluing used to define the Liouville connect sum will be described in Section \ref{Sec:StandardNeighborhood}.

\begin{defn}\label{Def:LiouvilleSum}
In the above notation, we say that the contact manifold
\begin{equation*}
\#_{(\Sdom,(i_{1},i_{2}))} \Mxi := (\#_{(\Sigma,(i_{1},i_{2}))} M, \#_{(\Sdom,(i_{1},i_{2}))}\xi)
\end{equation*}
is the \emph{Liouville connect sum} of $\Mxi$ associated to the Liouville embeddings $i_{1}, i_{2}$.
\end{defn}

When the embeddings $i_{1}$ and $i_{2}$ of Definition \ref{Def:LiouvilleSum} are understood, we will use the short-hand notation $\#_{\Sdom}\Mxi$ for $\#_{(\Sdom,(i_{1},i_{2}))} \Mxi$.  It should be noted that the Liouville connect sum in general depends on the embeddings $i_{1}$ and $i_{2}$, not just the images of $\Sigma$ under these mappings.

\begin{ex}[Weinstein surgery as a Liouville connect sum]\label{Ex:WeinsteinSum}
Consider the disjoint union $\Mxi\sqcup (S^{2n+1},\xi_{std})$ of some arbitrary $(2n+1)$-dimensional contact manifold and the standard $(2n+1)$-sphere.  Let $L$ be an isotropic $k$-sphere in $\Mxi$ with trivial normal bundle.  By considering $S^{2n+1}\subset\mathbb{R}^{2n+2}$ as in Example \ref{Ex:Sstd}, we may define $L'\subset (S^{2n+1},\xi_{std})=\partial (\disk^{2n+2},\lambda_{std})$ to be the isotropic $k$-sphere $S^{2n+1}\cap \Span(x_{1},\dots,x_{k+1})$.  Then we can find Liouville hypersurfaces $(\Sigma_{1},\lambda_{std}-\lambda_{can})\subset\Mxi$ and $(\Sigma_{2},\lambda_{std}-\lambda_{can})\subset (S^{2n+1},\xi_{std})$ which deformation retract onto $L$ and $L'$, respectively, as described in Example \ref{Ex:IsotropicNeighborhood}.  Now let $\Sdom$ be an additional copy of a neighborhood of the zero section of the bundle $\mathbb{R}^{2m}\oplus T^{*}S^{k}\rightarrow S^{k}$ with $\beta=\lambda_{std}-\lambda_{can}$ and define Liouville embeddings $i_{1}:\Sdom\rightarrow \Mxi$ and $i_{2}:\Sdom\rightarrow (S^{2n+1},\xi_{std})$ for which $i_{j}(\Sdom)=(\Sigma_{j},\lambda_{std}-\lambda_{can})$, $j=1,2$.

Applying a Liouville connect sum, we have that $\#_{\Sdom}\bigr{(}\Mxi\sqcup (S^{2n+1},\xi_{std})\bigr{)}$ is the same contact manifold as is described by a Weinstein handle attachment along $L\subset\Mxi$ (with respect to some framing of the symplectic normal bundle of $L$).  See Section \ref{Sec:WHandle}.
\end{ex}

The main result of this paper is the following theorem, whose proof appears in Section \ref{Sec:Cobordism}.

\begin{thm}\label{Thm:Cobordism}
Let $\Mxi$ be a closed, possibly disconnected, $(2n+1)$-dimensional contact manifold.  Suppose that there are two Liouville embeddings $i_{1},i_{2}:\Sdom\rightarrow\Mxi$ with disjoint images.  Then there is an exact symplectic cobordism $\Ldom$ whose negative boundary is $\Mxi$ and whose positive boundary is $\#_{\Sdom}\Mxi$.  Moreover, if $\Sdom$ admits a Weinstein structure, then so does the cobordism $\Ldom$.
\end{thm}

The proof of Theorem \ref{Thm:Cobordism} consists of attaching a symplectic handle $(\mathcal{H}_{\Sigma},\omega_{\beta})$ to the symplectization of $\Mxi$.  Note that the cobordism $(W,\lambda)$ described above is always Weinstein when dim$(M)=3$ as every 2-dimensional Liouville domain admits a Weinstein structure.  The proof of Theorem \ref{Thm:Cobordism} provides an explicit Weinstein handle decomposition of the cobordism $(W,\lambda)$ in the event that $\Sdom$ admits the structure of a Weinstein domain. The Weinstein decomposition of the handle $\mathcal{H}_{\Sigma}$ is roughly summarized by saying that a Weinstein $k$ handle in $\Sdom$ determines a Weinstein $k+1$ handle in the associated symplectic cobordism.

When the components of $\Mxi$ appear as convex boundary components of a weak symplectic cobordism $(W,\omega)$ (see Definition \ref{Def:Weak}) and $i_{j}:\Sdom\rightarrow \Mxi$, $j=1,2$, are Liouville embeddings, then it is possible to attach a slightly modified version of $(H_{\Sigma},\omega_{\beta})$ to $\partial W$ as above, provided the vanishing of a cohomological obstruction.  This obstruction always vanishes when $\dim(M)=3$.  See Section \ref{Sec:WeakHandle}.

\begin{ex}[Weinstein handle attachment as a connect sum cobordism]\label{Ex:LiouvilleWeinstein}
Consider the Liouville connect sum $\#_{\Sdom}\bigr{(}\Mxi\sqcup (S^{2n+1},\xi_{std})\bigr{)}$ from Example \ref{Ex:WeinsteinSum}.  In this case, the cobordism $\Ldom$ from $\bigr{(}\Mxi\sqcup (S^{2n+1},\xi_{std})\bigr{)}$ to $\#_{\Sdom}\bigr{(}\Mxi\sqcup (S^{2n+1},\xi_{std})\bigr{)}$ described in Theorem \ref{Thm:Cobordism} is the same as the usual cobordism (from $\Mxi$ to $\#_{\Sdom}\bigr{(}\Mxi\sqcup (S^{2n+1},\xi_{std})\bigr{)}$) provided by a Weinstein handle attachment with a standard ball $(\disk^{2n+2},\lambda_{std})$ removed from its interior.  See Sections \ref{Sec:WHandle} and \ref{Sec:Cobordism}.
\end{ex}

\subsection{Applications}\label{Sec:TopologicalApplications}

Now we state some consequences of Theorem \ref{Thm:Cobordism} whose proofs will appear later in the text. Throughout, we freely make use of the definitions and notation of Section \ref{Sec:Symplectic}.

\subsubsection{Open books and fillability monoids}

Our first application of Theorem \ref{Thm:Cobordism} is to the study of contact manifolds determined by open books. Proofs of the following results appear in Sections \ref{Sec:Monoids} and \ref{Sec:HCMonoids}.

\begin{defn}
Let $\Sigma$ be a compact, oriented manifold with non-empty boundary.  Let $\Diff^{+}(\Sigma,\partial \Sigma)$ be the group of orientation preserving diffeomorphisms of $\Sigma$ which restrict to the identity on some collar neighborhood of $\partial \Sigma$.  When $\Sigma$ admits a symplectic form $\omega$, the \emph{symplectomorphism group} of $(\Sigma,\omega)$ will refer to the subgroup 
\begin{equation*}
\Symp((\Sigma,\omega),\partial\Sigma) \subset \Diff^{+}(\Sigma,\partial \Sigma)
\end{equation*}
whose elements preserve $\omega$.
\end{defn}

For each pair $\AOB$ with $\Phi\in \Diff^{+}(\Sigma,\partial \Sigma)$ we can build a smooth manifold $\MOB$ defined by
\begin{equation*}
\begin{gathered}
\MOB = (\Sigma\times[0,1])/\sim\quad \text{where}\\
(\Phi(x),1)\sim(x,0)\quad \forall\ x\in \Sigma,\\
\quad (x,t)\sim(x,t')\quad\forall\ (x,t),(x,t')\in(\partial \Sigma)\times[0,1].
\end{gathered}
\end{equation*}
The manifold $\MOB$ is called the \emph{open book} associated to the pair $\AOB$.  The diffeomorphism class of $\MOB$ depends only on $\Phi$ up to conjugation and isotopy in $\Diff^{+}(\Sigma,\partial\Sigma)$.  Each $\Sigma\times\lbrace t \rbrace\subset \MOB$ is called a \emph{page} of the open book.  The codimension two submanifold $\partial \Sigma$ of $\MOB$ is called the \emph{binding} and is naturally oriented as the boundary of a page.  The diffeomorphism $\Phi$ is called the \emph{monodromy}.

\begin{defn}\label{Def:SupportingBook}
Let $\Mxi$ be a $(2n+1)$-dimensional contact manifold, let $\Sdom$ be a $2n$-dimensional Liouville domain, and let $\Phi\in\SympGroup$.  We say $\Mxi$ is \emph{supported by} the pair $(\Sdom,\Phi)$ if $M=\MOB$ and there is a contact form $\alpha$ for $\Mxi$ such that
\be
\item the restriction of $\alpha$  to the binding $B$ is a contact form,
\item the restriction of $d\alpha$ to each page is symplectic, and
\item the Liouville vector field $X_t$ for $d\alpha$ on each page $\Sigma_t$ points outward along a collar of $\partial\Sigma_t$.
\ee
\end{defn}

\begin{thm}[\cite{Giroux:ContactOB,TW:OB}]\label{Thm:GirCor}
Let $\Sdom$ be a Liouville domain and let $\Phi\in \SympGroup$.  Then
\be
\item $\MOB$ naturally carries a contact structure $\xi_{(\Sdom,\Phi)}$ supported by the pair $(\Sdom,\Phi)$.
\item $(M_{(\Sigma,\Phi)},\xi_{(\Sdom,\Phi)})$ depends only on $\Sdom$ and $\Phi$ up to conjugacy and isotopy in $\SympGroup$.
\item Every 3-dimensional contact manifold is supported by an open book, and
\item Two contact $3$-manifolds $(M_{(\Sigma,\Phi)},\xi_{(\Sdom,\Phi)})$ and $(M_{(\Sigma',\Psi)},\xi_{((\Sigma',\beta'),\Psi)})$ are contact-diffeomorphic if and only if the pairs $\AOB$ and $(\Sigma',\Psi)$ are related by a sequence of positive stabilizations.
\ee
\end{thm}

We recommend the exposition \cite{Etnyre:OBIntro} for further details on the above theorem. For simplicity, we will denote the contact manifold $(M_{(\Sigma,\Phi)},\xi_{(\Sdom,\Phi)})$ described in Theorem \ref{Thm:GirCor}(1) by $\Mxi_{(\Sdom,\Phi)}$.

\begin{rmk}
Giroux \cite{Giroux:ContactOB} has also outlined a program for characterizing high dimensional contact manifolds in terms of open books with Weinstein pages.  When $\dim(\Sigma)=2$, a symplectic form on $\Sigma$ is simply a volume form and every such $\Sigma$ admits the structure of a Liouville domain $\Sdom$.  In this case, after having specified such a 1-form $\beta$ on $\Sigma$, every $\Phi\in\Diff^{+}(\Sigma,\partial\Sigma)$ is isotopic to an element of $\SympGroup$.  Furthermore, any two Liouville 1-forms on a compact oriented surface with boundary are homotopic in the sense of Section \ref{Sec:Symplectic}.  Therefore the study of monodromies of open books determining contact 3-manifolds reduces to the study of mapping class groups of compact, oriented surfaces with non-empty boundary.
\end{rmk}

\begin{defn}
Let $\Sdom$ be a 2n-dimensional Liouville domain.  A property $\mathcal{P}$ of contact (2n+1)-manifolds is a \emph{monoid property} for $\SympGroup$ if the collection of $\Phi\in\SympGroup$ for which $\Mxi_{(\Sdom,\Phi)}$ satisfies $\mathcal{P}$ is a monoid in $\SympGroup$.
\end{defn}

For the following results, we abbreviate $\Symp = \SympGroup$ when $\Sdom$ is understood.

\begin{thm}\label{Thm:Monoids}
Let $\Sdom$ be a Liouville domain.
\be
\item ``Symplectically fillable'' and ``exactly symplectically fillable'' are monoid properties for $\Symp$.
\item If $\dim(\Sigma)=2$, then ``weakly fillable'' and ``Weinstein fillable'' are monoid properties for $\Symp$.
\item Moreover, if $\Sdom$ is of any even dimension and admits the structure of Weinstein domain, then ``Weinstein fillable'' is a monoid property for $\Symp$.
\ee
\end{thm}

Theorem \ref{Thm:Monoids} was motivated by and generalizes results of Baker, Etnyre, and van Horn-Morris \cite[\S 1.2]{BEV:Monoids} and Baldwin \cite[Theorems 1.1-1.3]{Baldwin:Monoids}.

The question of whether or not ``weakly fillable'' is a monoid property for Liouville domains of dimension greater than two appears to be more subtle.

\begin{thm}\label{Thm:WeakMonoids}
Let $\Sdom$ be a Liouville domain for which $\dim(\Sigma) > 2$, and let $\Phi,\Psi\in\Symp$.  Suppose that $\Mxi_{(\Sdom,\Phi)}$ and $\Mxi_{(\Sdom,\Psi)}$ admit weak symplectic fillings $(W_{1},\omega_{1})$ and $(W_{2},\omega_{2})$, respectively.  Then, if $i_{\Phi}^{*}\omega_{1}=i_{\Psi}^{*}\omega_{2} \in H^{2}(\Sigma ;\mathbb{R})$, the contact manifold $\Mxi_{(\Sdom,\Phi\circ\Psi)}$ is weakly symplectically fillable.  

In particular, if $H^{2}(\Sigma ;\mathbb{R})=0$, then ``weakly symplectically fillable'' is a monoid property for $\Symp$.
\end{thm}

If $\Sdom$ is a Liouville domain, then $\Id_{\Sigma}\in\Symp$ is an element of the ``exactly symplectically fillable'' monoid in $\SympGroup$.  This is a consequence of the fact that the contact manifold $\Mxi_{(\Sdom,\Id_{\Sigma})}$ can be realized as the boundary of the Liouville domain obtained by rounding the corners of $(\Sigma\times\disk^{2},\beta+\lambda_{std})$.  Similarly, if $\Sdom$ admits the structure of a Weinstein domain, then $\Id_{\Sigma}$ is an element of the ``Weinstein fillable'' monoid in $\Symp$.  This is a consequence of the fact that $(\Sigma\times\disk^{2},\beta+\lambda_{std})$ admits the structure of a Weinstein domain after rounding the corners of the product.  For more information on Weinstein domains of this type, see \cite{Cieliebak}.

Given a Liouville domain $\Sdom$ and some $\Phi\in\Symp$, there is a natural Liouville embedding $i_{\Phi}$ of $\Sdom$ into the contact manifold $\Mxi_{(\Sdom,\Phi)}$ whose image is the page of the associated open book.  The naturality of this embedding follows from the fact that the manifold $M_{(\Sigma,\Phi)}$ is defined constructively. The proof of Theorem \ref{Thm:Monoids} consists of constructing $\Mxi_{(\Sdom,\Phi\circ\Psi)}$ from $\Mxi_{(\Sdom,\Phi)}\sqcup\Mxi_{(\Sdom,\Psi)}$ by a Liouville connect sum and then appealing to the existence of the symplectic cobordism $(W,\lambda)$ described in Theorem \ref{Thm:Cobordism}.  The fact that this cobordism is exact provides us with the following easy corollary:

\begin{cor}\label{Cor:HCMonoids}
Let $\Sdom$ be a Liouville domain.  Then ``non-vanishing contact homology with rational coefficients'' is a monoid property for $\Symp$.
\end{cor}

This result is analogous to a theorem first proved in \cite[Theorem 1.2]{Baldwin:Comultiplication} (see also \cite{BEV:Monoids,Baldwin:Monoids})  regarding the non-vanishing of contact classes in the Heegaard Floer homologies of contact $3$-manifolds.

\subsubsection{Fillability of fibered contact manifolds}

Here we define a family $\Mxi_{(\Sdom,\Phi,\Psi)}$ of $(2n+1)$-dimensional contact manifolds which fiber over the circle $S^{1}$, each determined by a $2n$-dimensional Liouville domain $\Sdom$ and a pair of symplectomorphisms $\Phi,\Psi\in\SympGroup$. When $\dim(\Sigma)=2$, this family of contact manifolds forms a subset of the collection of universally tight surface bundles over $S^{1}$.  The tightness and fillability of contact structures on surface bundles over the circle have been studied extensively.  See, for example, \cite{DG:Surgery, Eliashberg:Torus, Ghiggini:HFFillability, Giroux:T3, Honda:TightClassII,HKM:Fibered,VHM:Thesis,Wendl:Cobordism}.

Consider a tubular neighborhood $[-1,1]\times \partial \mathcal{N}(\Sigma)$ where $\mathcal{N}(\Sigma)$ is the model neighborhood described in Section \ref{Sec:Hypersurface} and $\theta$ is a coordinate on $[-1,1]$.  The manifold $[-1,1]\times \mathcal{N}(\Sigma)$ inherits a $\theta$-invariant contact structure from the contact form $d\theta+\beta$ on $N(\Sigma)$.  By gluing
\be
\item the positive region of $\{1\}\times\partial \mathcal{N}(\Sigma)$ to the negative region of $\{-1\}\times \partial \mathcal{N}(\Sigma)$ using the map $\Phi$ and
\item the positive region of $\{-1\}\times\partial \mathcal{N}(\Sigma)$ to the negative region of $\{1\}\times \partial \mathcal{N}(\Sigma)$ using the map $\Psi$
\ee
we obtain $\Mxi_{(\Sdom,\Phi,\Psi)}$.  See Figure \ref{Fig:SurfaceBundle}. In the simplest case, with $\Phi=\Psi=\Id_{\Sigma}$, $\Mxi_{(\Sdom,\Phi,\Psi)}$ is the boundary of the Liouville domain obtained by rounding the corners of $(\Sigma\times \disk^{*}S^{1},\beta-\lambda_{can})$.

\begin{figure}[h]
	\begin{overpic}[scale=.7]{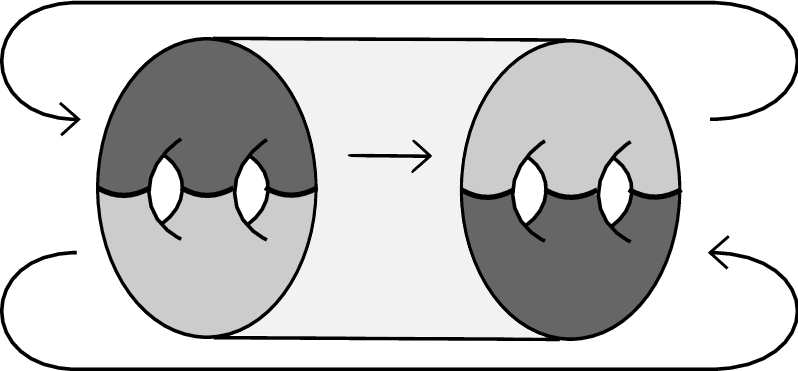}
        \put(-5,37){$\Phi$}
        \put(102,7){$\Psi$}
        \put(48,20){$\partial_{\theta}$}
    \end{overpic}
    \vspace{1.5mm}
	\caption{The contact manifold $\Mxi_{(\Sdom,\Phi,\Psi})$.  It is determined by the convex gluing instructions shown on the boundary of the contact manifold $\bigr{(}[-1,1]\times \partial \mathcal{N}(\Sigma),\xi_{\Sdom}\bigr{)}$.  See Section \ref{Sec:Fibration} for further explanation.}
    \label{Fig:SurfaceBundle}
\end{figure}

\begin{rmk}\label{Rmk:OBConvention}
Our convention that $\Phi$ and $\Psi$ point in opposite direction in Figures \ref{Fig:SurfaceBundle}, \ref{Fig:GluingInstructions}, and \ref{Fig:Heegaard2} as well as in Section \ref{Sec:GluingInstructions} has been chosen so as to fit into the following scenario: If $(M_{1}, \xi_{1})$ and $(M_{2}, \xi_{2})$ are standard neighborhoods of some $\Sdom$, then the contact manifold obtained by the gluing described -- with $\Phi$ sending the positive region of $\partial M_{1}$ to the negative region of $\partial M_{2}$ and $\Psi$ sending the positive region of $\partial M_{2}$ to the negative region of $\partial M_{1}$ -- will be supported by an open book decomposition with page $\Sdom$ and whose monodromy is $\Phi \circ \Psi \in \Symp(\Sdom, \partial \Sigma)$.
\end{rmk}

A slight variation of the proof of Theorem \ref{Thm:Monoids} -- appearing in Section \ref{Sec:Fibration} -- yields the following result which relates the fillability of these contact manifolds to fillability of open book decompositions:

\begin{thm}\label{Thm:Fibration}
Let $\Sdom$ be a $2n$-dimensional Liouville domain.
\be
\item If $\Mxi_{(\Sdom,\Phi\circ\Psi)}$ is symplectically (exactly) fillable, then $\Mxi_{(\Sdom,\Phi,\Psi)}$ is also symplectically (exactly) fillable.
\item If $\Sdom$ admits a Weinstein structure, and $\Mxi_{(\Sdom,\Phi\circ\Psi)}$ is Weinstein fillable then so is $\Mxi_{(\Sdom,\Phi,\Psi)}$.
\item If $\dim(M)=3$ and $\Mxi_{(\Sdom,\Phi\circ\Psi)}$ is weakly fillable, then so is $\Mxi_{(\Sdom,\Phi,\Psi)}$.
\ee
Furthermore, if $\dim (M)>3$, $\Mxi_{(\Sdom,\Phi\circ\Psi)}$ admits a weak filling $(W,\omega)$, and $i_{\Phi\circ\Psi}^{*}\circ(\Id_{\Sigma}-\Phi^{*})\omega\in \Omega^{2}(\Sigma)$ is exact, then $\Mxi_{(\Sdom,\Phi,\Psi)}$ is also weakly fillable.
\end{thm}

\subsubsection{Fillability of branched covers}

Our next application of Theorem \ref{Thm:Cobordism} concerns branched covers of contact manifolds. Let $\Mxi$ be a contact manifold and suppose that we have a compact, codimension $1$ submanifold $\Sigma \subset M$ whose boundary is a compact codimension $2$ contact submanifold $(\partial \Sigma, \zeta)$ of $\Mxi$. Then for each $k > 0$, the hypersurface $\Sigma$ determines a $k$-fold contact-branched covering $\Mxi_{\Sigma, k}$ of $\Mxi$, with branch locus $\partial \Sigma$.

\begin{thm}\label{Thm:Branched}
Suppose that the hypersurface $\Sigma$ described above is a Liouville hypersurface $\Sdom\subset\Mxi$.  Then for $k\geq 2$ there is an exact symplectic cobordism $\Ldom$ whose concave boundary is $\sqcup^{k}\Mxi$ and whose convex boundary is $\Mxi_{\Sigma,k}$.  If $\Sdom$ admits a Weinstein structure, then so does the cobordism $\Ldom$.  Moreover, if $\Mxi$ is weakly fillable, then so is $\Mxi_{\Sigma,k}$.
\end{thm}

Theorem \ref{Thm:Branched} is similar in flavor to results of Baldwin \cite{Baldwin:Monoids} and of Harvey, Kawamuro, and Plamenevskaya \cite{HKP:Branched} regarding cyclic branched coverings of contact 3-manifolds. These results will summarized in Section \ref{Sec:Branched} along side the proof of Theorem \ref{Thm:Branched}.

\subsubsection{Liouville domains without Weinstein structures}

In Section \ref{Sec:NotWeinstein} we discuss how Theorem \ref{Thm:Cobordism} can be used to construct Liouville domains and exact symplectic cobordisms which do not admit Weinstein structures.  The examples we provide have connected boundary, although their construction relies on the existence of Liouville domains with disconnected boundary -- c.f. \cite{Geiges:Disconnected,MNW12,McDuff}.  The examples appearing in Section \ref{Sec:NotWeinstein} show that the cobordisms described in Theorem \ref{Thm:Cobordism} are not always Weinstein.

\subsubsection{Contact $(1/k)$-surgery and generalized Dehn twists}

In \cite{DG:Surgery}, Ding-Geiges define contact $(1/k)$-surgery along Legendrian knots in contact 3-manifolds, generalizing Weinstein's Legendrian surgery \cite{Weinstein:Handles} in the 3-dimensional case.  Using the Liouville connect sum and generalized Dehn twists \cite{Arnold, Seidel1}, we provide a definition of contact $(1/k)$-surgery along Legendrian $n$-spheres in contact $(2n+1)$-manifolds for arbitrary $n\geq 1$ which coincides with the usual definition for $n=1$.  For $k=-1$, our definition coincides with that Legendrian surgery.  For $n=1$, our definition coincides with the usual notion of contact $(1/k)$-surgery.  

As in those known cases, contact $(1/k)$ surgery may be summarized as follows: We take a neighborhood $[-\epsilon, \epsilon] \times \disk^{*}S^{n}$ of a Legendrian sphere in a contact $(2n+1)$-manifold, remove it, and then glue it back, applying $-k$ symplectic Dehn twists applied to the ``top'' $\{\epsilon \} \times \disk^{*}S^{n}$ of our neighborhood.

In the case $n=1$, our Legendrian sphere is a knot $K$ and the surgery described is smoothly a Dehn surgery with coefficient $(1/k)$ with respect to the framing on $K$ determined by the contact structure. See for example \cite[Section 11]{OzbSt:SteinSurgery}. Likewise, for $n=1$ a generalized Dehn twist coincides with the usual notion of a positive Dehn twist on the annulus.

This construction is described in detail in Section \ref{Sec:ContactSurgery}. There we observe that many properties of contact $(1/k)$-surgery known to hold in the 3-dimensional case easily carry over to contact manifolds of arbitrary dimension.

Denote by $\tau_{n} \in \Symp((\disk^{*}S^{n},-\lambda_{std}), \partial \disk^{*}S^{n})$ a positive symplectic Dehn twist. Our first application of contact surgery concerns contact manifolds constructively described as open books.

\begin{thm}\label{Thm:NegativeMonodromy}
Let $k$ and $n$ be positive integers. Then the contact manifold determined by the open book whose page is $(\disk^{*}S^{n},-\lambda_{std})$ and whose monodromy is $\tau^{-k}$ is not exactly symplectically fillable.
\end{thm}

The above generalizes the well-known case $n=1$, whence the contact manifold described is an overtwisted lens space, and the case $k=1$ originally due to Bourgeois and van Koert \cite{BK:Stabilize}.

A theorem of Seidel asserts that the square of a generalized Dehn twist along $S^{2}$ is smoothly isotopic to the identity mapping.  In Section \ref{Sec:DehnTwist} we make use of his proof \cite[Lemma 6.3]{Seidel2} as well as some classic homotopy theoretic results \cite{Adams, James} to enhance this theorem.

\begin{thm}\label{Thm:SquareClassification}
Considered as an element of $\Diff^{+}(\disk^{*}S^{n},\partial\disk^{*}S^{n})$, $\tau_{n}^{2}$ is isotopic to the identity mapping if and only if $n$ is either $2$ or $6$.
\end{thm}

With the help of Theorem \ref{Thm:SquareClassification}, we construct exotic contact structures on $S^{5}$ and $S^{13}$ by performing contact $\frac{1}{2k}$-surgeries along Legendrian spheres in $(S^{5},\xi_{std})$ and $(S^{13},\xi_{std})$.  The following theorem then immediately follows from this construction.

\begin{thm}\label{Thm:TauSquared}
Considered as elements of $\Symp((\disk^{*}S^{n},-\lambda_{std}),\partial\disk^{*}S^{n})$ $(n=2,6)$, $\tau_{2}^{2}$ and $\tau_{6}^{2}$ are not isotopic to the identity.\footnote{A considerably stronger version of this result which is specific to the case $n=2$ can be found in \cite[Proposition 2.6]{Seidel1}.}
\end{thm}

The case $n=2$ of Theorem \ref{Thm:TauSquared} was originally proved in \cite{Seidel1} using Floer homology.  Our proof is a consequence of the exoticity of the contact spheres described above, which we establish using the Eliashberg-Floer-Gromov-McDuff theorem \cite{Eliashberg:Plumbing, McDuff} asserting that an exact symplectic filling $(W,\lambda)$ of $(S^{2n+1},\xi_{std})$ must be such that $W$ is diffeomorphic to $\disk^{2n+2}$.

More examples of exotic contact structures on spheres of dimension greater than three can be found in \cite{DG:Exotic}, \cite{Eliashberg:Plumbing}, and \cite{U99}.  Non-standard contact structures on $S^{3}$ are completely understood by \cite{Bennequin, Eliashberg:OT, Eliashberg:Ineq, Huang:OT}.  Another contact-geometric proof of the symplectic non-triviality of squares of Dehn twists along $S^{2}$ -- obtained by analyzing the contact manifolds described in \cite{U99} -- can be found in \cite{KN05}.

\subsection{Outline}

The remainder of this paper is organized as follows:

\textbf{Section \ref{Sec:Definitions}.}  This section consists mostly of the establishment of notation and includes a brief overview of Weinstein handle attachment.

\textbf{Section \ref{Sec:StandardNeighborhood}.}  We carry out the technical details concerning neighborhood theorems for Liouville hypersurfaces required to rigorously define the Liouville connect sum.

\textbf{Section \ref{Sec:Cobordism}.}  This section contains the proof of Theorems \ref{Thm:Cobordism}.

\textbf{Section \ref{Sec:HighCodimension}.}  Here we analyze neighborhoods of Liouville submanifolds of codimension greater than one.

\textbf{Section \ref{Sec:LHSEx}.}  In this section we give various examples of Liouville hypersurfaces in contact manifolds.

\textbf{Section \ref{Sec:Liouville}.}  We describe further basic consequences of Definition \ref{Def:Hypersurface}.

\textbf{Section \ref{Sec:Applications}.}  Here we prove some of the corollaries of Theorem \ref{Thm:Cobordism} stated in Section \ref{Sec:TopologicalApplications}.  We also provide an example which shows how the proof of Theorem \ref{Thm:Cobordism} can be used to draw a Kirby diagrams for the cobordism described in Theorem \ref{Thm:Cobordism} in the event that the contact manifold $\Mxi$ is $3$-dimensional.

\textbf{Section \ref{Sec:ContactSurgery}.}  This section defines and outlines some of the basic properties of contact $(1/k)$-surgery.  There we also briefly review known facts about contact $(1/k)$-surgery on contact 3-manifolds and generalized Dehn twists.

\textbf{Section \ref{Sec:TwistApplications}.}  In this section we use open books and contact $(1/k)$-surgery to study the symplectic topology of generalized Dehn twists, proving Theorems \ref{Thm:NegativeMonodromy}, \ref{Thm:SquareClassification}, and \ref{Thm:TauSquared}.

\subsection{Updates}

We briefly address some research developments which have occurred since the first version of this article was posted to the ArXiv in 2012. The following comments are ordered in accordance with the appearances of relevant content in the text below.

Regarding foundational aspects of Liouville domains and their role in contact topology, the reader is referred to Giroux's \cite{Giroux:IdealLiouville} which establishes new notion of completed Liouville domain useful for studying convex surfaces. See also \cite{HH:Bypass, HH:Convex}, in which Honda-Huang explore foundations of convex surface theory in contact manifolds of dimension greater than $3$.

Regarding the applications of Theorem \ref{Thm:Cobordism} appearing in Section \ref{Sec:Applications}: Foundations of contact homology (Corollary \ref{Cor:HCMonoids}) have since been rigorously established by Bao-Honda \cite{BaoHonda:ContactHomology} and Pardon \cite{Pardon:ContactHomology}. In \cite{Gironella:Branches}, Gironella provides an updated and more natural definition of contact-branched cover as considered in Theorem \ref{Thm:Branched}

For content pertaining to contact surgery in Section \ref{Sec:ContactSurgery}: 
Overtwisted contact structures have been defined and classified in all dimensions by Borman-Eliashberg-Murphy in \cite{BEM14}. The equivalences of various notions of overtwistedness -- such as those appearing Conjecture \ref{Conj:Loose} -- have been established by Casals-Murphy-Presas in \cite{CMP19}.

\subsection{Acknowledgments}

I would like to thank my advisor Ko Honda for his invaluable guidance and for helping me to edit earlier versions of this paper.  I also thank the organizers of the 2011 Trimester on Contact and Symplectic Topology at the Universit\'{e} de Nantes as well as the organizers of the Caltech Geometry and Topology Seminar for providing me the opportunity to present the results described here.  Thanks to Patrick Massot, Klaus Niederkr\"{u}ger, and Chris Wendl for interesting discussions especially in regards to weak symplectic fillings and cobordisms. Finally I am indebted to an anonymous referee for suggesting many improvements to this paper including the addition of Theorem \ref{Thm:NegativeMonodromy}.

\section{Notation and definitions}\label{Sec:Definitions}

This section begins by establishing some notation which will be used throughout the paper.  In Section \ref{Sec:Symplectic} we recall some definitions and basic examples from symplectic geometry.  In Section \ref{Sec:WHandle} we provide a brief overview handlebody constructions and decompositions of Weinstein domains and cobordisms.

\subsection{Notation}

Suppose that $L$ is a smooth $n$-dimensional manifold.

A vector bundle $E\rightarrow L$ over $L$ will always assumed to be smooth and have finite rank.  Such a vector bundle will always be considered as a real vector bundle, even if it is equipped with a complex structure.

The cotangent bundle will be written $T^{*}L$.  After equipping $L$ with a Riemannian metric $\langle \ast,\ast\rangle$ we may consider the unit disk and sphere bundles in $T^{*}L$.  These will be denoted by $\disk^{*}L$ and $S^{*}L$, respectively.  In this paper, we will not be interested in the geometry of any particular Riemannian metrics, and so -- with the exception of Section \ref{Sec:DehnTwist} -- we will refer to $\disk^{*}L$ and $S^{*}X$ without explicitly specifying a metric.

For a vector field $V$ on $L$, the diffeomorphism of $L$ determined by the time-$t$ flow of $V$ will be written $\Flow^{t}_{V}$. Lie derivatives with respect to $V$ will be written $\Lie_{V}$.

If $L$ is closed and oriented, then the fundamental class of $L$ in $H_{n}(L, \mathbb{Z})$ will be written $[L]$.

For a contact manifold $\Mxi$, with contact form $\alpha$, the associated \emph{Reeb vector field} will be denoted by $R_{\alpha}$. Recall that $R_{\alpha}$ is uniquely determined by the equations
\begin{equation*}
\alpha(R_{\alpha})=1\quad\text{and}\quad d\alpha(R_{\alpha},\ast)=0.
\end{equation*}

For a smooth function $f\in\Cinfty(W)$ on a symplectic manifold $(W,\omega)$, the \emph{Hamiltonian vector field} $X_{f}$ is defined by the convention that
\begin{equation*}
df(\ast)=\omega(X_{f},\ast).
\end{equation*}

\subsection{Definitions from symplectic geometry}\label{Sec:Symplectic}

We continue with the discussion started in Section \ref{Sec:Preliminaries}.

\begin{ex}\label{Ex:CotangentBundle}
Let $L$ be a closed, smooth $n$-manifold.  The cotangent bundle $T^{*}L$ of $L$ admits a Liouville 1-form $-\lambda_{can}$.  If $(q_{1},\dots,q_{n})$ is a coordinate chart on $L$ then, in the associated coordinate $(q_{i},p_{i})$ on $T^{*}L$ -- the $p_{j}$ being the coefficients of the $dq_{j}$ -- we have $\lambda_{can}=\sum p_{j}dq_{j}$.  The associated Liouville vector field is given by the radial vector field, written $X_{can}=\sum p_{j}\partial_{p_{j}}$ in local coordinates.  Then $(\disk^{*}L, -\lambda_{can})$ is a Liouville domain.  The induced contact structure $\xi_{can}=\ker(-\lambda_{can})$ on $S^{*}L$ is called the \emph{canonical contact structure}.  Similarly, $\lambda_{can}$ is called the \emph{canonical 1-form} and $-d\lambda_{can}$ is called the \emph{canonical symplectic form} on $T^{*}L$.  Note that $\xi_{can}$ is independent of the metric used to define $S^{*}L$ by Gray's stability theorem.
\end{ex}

\begin{ex}\label{Ex:Products}
Let $\Sdom$ and $(\Sigma',\beta')$ be two Liouville domains.  Then $(\Sigma\times\Sigma',\beta+\beta')$ admits the structure of a Liouville domain after rounding the corners $(\partial\Sigma)\times(\partial\Sigma')$ of the product.
\end{ex}

\begin{defn}\label{Def:WeinsteinDom}
A \emph{Weinstein domain} $(\Sigma, \omega, X, f)$ is a Liouville domain $(\Sigma, \omega(X, \ast))$ whose Liouville vector field $X$ is gradient-like for a Morse function $f:\Sigma \rightarrow \mathbb{R}$ for which $\partial \Sigma$ is a regular value.
\end{defn}

The above conditions imply that the Liouville vector field $X$ satisfies $df(X) \geq 0$ with strict inequality along $\partial\Sigma$. Note that $(\disk^{2n+2},\lambda_{std})$, as described in Example \ref{Ex:Sstd} has a Weinstein structure with $f(z)=\| z \|^{2}$.

\begin{defn}\label{Def:Fillable}
Let $\Mxi$ be a contact manifold and suppose that $(M',\xi')$ is another contact manifold with $\dim(M')=\dim(M)$.  A symplectic manifold $(W,\omega)$ with $\partial W = M\sqcup (-M')$ is
\be
\item a \emph{symplectic cobordism with concave boundary $(M',\xi')$ and convex boundary $\Mxi$} if both $M$ and $M'$ are contact-type hypersurfaces in $(W,\omega)$ and the induced contact structures on $M$ and $M'$ are $\xi$ and $\xi'$ respectively.
\item an \emph{exact symplectic cobordism with concave boundary $(M',\xi')$ and convex boundary $\Mxi$} if it is a symplectic cobordism and the $1$-form $\lambda$ used to identify $M$ and $M'$ as contact-type hypersurface (as described in Section \ref{Sec:Preliminaries}) is defined on all of $W$.
\item a \emph{Weinstein cobordism with concave boundary $(M',\xi')$ and convex boundary $\Mxi$} if it is an exact symplectic cobordism for which there exists a Morse function $f:W \rightarrow \mathbb{R}$ such that $M$ and $M'$ are inverse images of regular values of $f$ and $X$ is gradient-like for $f$.
\ee
Similarly, we say that $\Mxi$ is \emph{symplectically (resp. exactly, Weinstein) fillable} if there is a symplectic (resp. exact, Weinstein) cobordism with concave boundary $\Mxi$ and empty concave boundary.
\end{defn}

When a symplectic cobordism $(W,\omega)$ is exact (or Weinstein) with $\lambda\in\Omega^{1}(W)$ satisfying $d\lambda=\omega$ as in item $(2)$ of the above definition, it shall be specified by the pair $(W,\lambda)$ to emphasize exactness.

Two exact symplectic cobordisms $\Ldom$ and $(W,\lambda')$ will be called \emph{homotopic} if there is a smooth $[0,1]$-family $\lambda_{t}$ of $1$-forms on $W$ such that $\lambda_{0}=\lambda$, $\lambda_{1}=\lambda'$, and $(W,\lambda_{t})$ is an exact symplectic cobordism for all $t\in[0,1]$.  Note that if $\Ldom$ and $(W,\lambda')$ are homotopic, then the concave and convex boundaries of $\Ldom$ and $(W,\lambda')$ are pairwise contact-diffeomorphic.

\begin{defn}\label{Def:FiniteSymplectization}
Let $\Mxi$ be a closed contact manifold with contact 1-form $\alpha$ and let $C_{1}, C_{2}$ be a pair of real constants with $C_{1} < C_{2}$.  A \emph{finite symplectization}
\begin{equation*}
([C_{1}, C_{2}]\times M, e^{t}\alpha)
\end{equation*}
is an exact symplectic cobordism whose concave and convex boundaries are both $\Mxi$. Writing $t$ for a coordinate on $[C_{1}, C_{2}]$, it is clear that 
\begin{equation*}
([-C_{1}, C_{2}]\times M, d(e^{t}\cdot\alpha), \partial_t, t)
\end{equation*}
is a Weinstein cobordism without critical points.
\end{defn}

There is one last type of cobordism we will consider in this paper:

\begin{defn}\label{Def:Weak}
Let $\Mxi$ and $(M',\xi')$ be $(2n+1)$-dimensional contact manifolds.  A compact symplectic manifold $(W,\omega)$ is a \emph{weak symplectic cobordism} with convex boundary $\Mxi$ and concave boundary $(M',\xi')$ if
\be
\item $\partial W=M\sqcup (-M')$,
\item both $\alpha\wedge(d\alpha +\omega|_{\xi})^{n}$ and $\alpha\wedge \omega|_{\xi}^{n}$ define positive volume forms on $M$ for every choice of contact 1-form $\alpha$ for $\Mxi$, and
\item both $\alpha'\wedge(d\alpha' +\omega|_{\xi'})^{n}$ and $\alpha'\wedge \omega|_{\xi'}^{n}$ define positive volume forms on $M'$ for every choice of contact 1-form $\alpha'$ for $(M',\xi')$.
\ee
In the event that $M'=\emptyset$, we say that $(W,\omega)$ is a \emph{weak symplectic filling} of $\Mxi$.
\end{defn}

The above definition -- first stated in \cite{MNW12} -- is related to the notion of \emph{$\omega$-dominating cobordism} -- first defined in \cite[\S 3]{EG:Convex}.  We say that a symplectic manifold $(W,\omega)$ with non-empty boundary \emph{dominates} a contact structure $\xi$ on its boundary if the conformal class of $\omega|_{\xi}$ coincides with the conformal class of symplectic structure on $\xi$ determined by a contact form for $(\partial W,\xi)$.  This definition is the same as Definition \ref{Def:Weak} for 3-dimensional contact manifolds. See \cite{MNW12} for further discussion.

As pointed out by McDuff in \cite[Lemma 2.1]{McDuff}, $\omega$-dominating cobordisms between (or fillings of) contact manifolds of dimension greater than or equal to 5 are symplectic cobordisms (fillings) in the sense of Definition \ref{Def:Fillable}(1).  However, there are 4-dimensional weak symplectic -- but not symplectic -- cobordisms and fillings.  See, for example, \cite[Theorem 1]{DG:Surgery}, \cite[\S 3]{Eliashberg:Torus}, \cite[\S 2.D]{Giroux:T3}, and \cite[\S 1.2]{Wendl:Cobordism}.  In \cite{MNW12} a strategy is described for constructing weakly fillable -- but not symplectically fillable -- contact manifolds of all dimensions greater than three, with examples provided in dimension five \cite[Theorem E]{MNW12}.

\subsection{Weinstein handles}\label{Sec:WHandle}

Now we define Weinstein handle attachments and outline their role in the construction of Weinstein domains and cobordisms.  We have included this material as the proof of the second statement of Theorem \ref{Thm:Cobordism}, which is contained in Section \ref{Sec:WHandleDecomposition}, will require an explicit description of the differential forms involved. It should be noted that we could alternatively have chosen to state many of the results of this paper in the language of \emph{Stein manifolds}.  See \cite{SteinToWeinstein, Eliashberg:Stein, EG:Convex, Gompf:Handles}.

\subsubsection{Definition and construction of the handle}

Consider $\mathbb{R}^{2n}$ with its standard Liouville form $\lambda_{std}$ and Liouville vector field $X_{\lambda_{std}}=\frac{1}{2}\sum_{1}^{n}(x_{j}\partial_{x_{j}}+y_{j}\partial_{y_{j}})$ as described in Example \ref{Ex:Sstd}. Let $\disk^{k}\subset\mathbb{R}^{2n}$ be the unit disk in the plane Span$(x_{1},\dots,x_{k})$.  Then $\disk^{k}$ is an isotropic submanifold of $(\mathbb{R}^{2n},d\lambda_{std})$.  Consider a tubular neighborhood $H_{n,k}:=\disk^{k}\times\disk^{2n-k}$ of $\disk^{k}$.  Then
\begin{equation*}
\partial H_{n,k}=\bigr{(}(\partial \disk^{k})\times\disk^{2n-k}\bigr{)}\cup\bigr{(} \disk^{k}\times(\partial \disk^{2n-k})\bigr{)}=\bigr{(} S^{k-1}\times\disk^{2n-k}\bigr{)}\cup\bigr{(}\disk^{k}\times S^{2n-k-1}\bigr{)}.
\end{equation*}

Now consider the function $f_{k}(x,y)=\sum_{1}^{k} x_{j} y_{j}$.  The Hamiltonian vector field of $f_{k}$ with respect to $d\lambda_{std}$ is $X_{f_{k}}=\sum_{1}^{k}(x_{j}\partial_{x_{j}}-y_{j}\partial_{y_{j}})$.  Then $X_{\lambda_{std}} + X_{f_{k}}$ is a symplectic dilation of $(\mathbb{R}^{2n},d\lambda_{std})$ which points into $H_{n,k}$ along $S^{k-1}\times \disk^{2n-k}$ and out of $H_{n,k}$ along $\disk^{k}\times S^{2n-k-1}$.  In other words, the 1-form
\begin{equation} \label{Eq:WForm}
\lambda_{n,k}(\ast):=d\lambda_{std}(X_{\lambda}+X_{f_{k}},\ast) = \sum_{1}^{k}\big{(}\frac{3}{2}x_{j}dy_{j}+\frac{1}{2}y_{j}dx_{j}\big{)}+\frac{1}{2}\sum_{k+1}^{n}\big{(}x_{j}dy_{j}-y_{j}dx_{j}\big{)}
\end{equation}
determines a contact structure on each of the smooth pieces of $\partial H_{n,k}$, such that $S^{k-1}\times \disk^{2n-k}$ is concave and $\disk^{k}\times S^{2n-k-1}$ is convex.

\begin{defn}
$(H_{n,k},\lambda_{n,k})$ is called the \emph{$2n$-dimensional Weinstein $k$-handle}.
\end{defn}

Now suppose that $(W,\omega)$ is a $(2n+2)$-dimensional symplectic with concave boundary $(M',\xi')$ and convex boundary $\Mxi$ and that there is a contact embedding of $(S^{k-1}\times \disk^{2n+2-k},\ker(\lambda_{n+1,k}))$ into $\Mxi$.  By Equation \eqref{Eq:WForm} $L:=S^{k-1}\times\{0\}$ is an isotropic submanifold of $\Mxi$.  On a collar neighborhood $(\frac{1}{2},1]\times M$ of $M$ in $W$, we can write $\omega=d(t\cdot \alpha)$ where $\alpha$ is a contact form for $\Mxi$.  As $\Mxi$ is a convex component of $(W,\omega)$ and $S^{k-1}\times \disk^{2n+2-k}$ is a concave component of $(H_{n+1,k},\lambda_{n+1,k})$ then we can patch together the Liouville forms $t\cdot \alpha$ and $\lambda_{k}$ on the manifold
\begin{equation*}
W\cup_{S^{k-1}\times\disk^{2n+2-k}} H_{n+1,k}
\end{equation*}
to get a new symplectic cobordism whose concave boundary is $(M',\xi')$ and whose convex boundary is a contact manifold $(M'',\xi'')$ obtained from $\Mxi$ via surgery.  This procedure is called \emph{Weinstein handle attachment} along $L$ and is due to Weinstein \cite{Weinstein:Handles}.  When $k=n+1$, the submanifold $L$ is Legendrian and the handle attachment is often called \emph{Legendrian surgery}.

\begin{rmk}  Note that the above discussion excludes the edge-rounding required to make the boundary of the manifold obtained by handle attachment smooth. A more careful description of the gluing map for Weinstein handle attachment can be found in \cite{Weinstein:Handles} or by following the handle attachment construction of Section \ref{Sec:Handle}.
	
The manifold produced by performing Legendrian surgery along an isotropic sphere will in general depend on the chosen parameterization of $L$ as well as a choice of framing on its normal bundle. A description of normal bundles of isotropic submanifolds are described in Example \ref{Ex:IsotropicNeighborhood}. The results of Section \ref{Sec:HighCodimensionNeighborhoods} may be viewed as a generalization. Parameterization choices for $(2n)$-dimensional $n$-handles in relation to contact surgery and symplectic Dehn twists are discussed in Section \ref{Sec:ParametrizationDependenceOfTwists}.
\end{rmk}

\begin{ex}\label{Ex:WeinsteinSphere}
Again consider the Lagrangian disk $L=\Span(x_{1},\dots,x_{n})\cap \disk^{2n}$ in $(\mathbb{R}^{2n},d\lambda_{std})$.  Then $\partial L$ is a Legendrian sphere in $(S^{2n-1},\xi_{std})=\partial (\disk^{2n},\lambda_{std})$.  Suppose that we attach a Weinstein handle $H_{n,n}$ to $\partial \disk^{2n}$ along $\partial L$ producing a new Liouville domain $\Ldom$.  If we write $L'$ for the core disk $\disk^{n}\times \{0\} \subset H_{n,n}$ of the Weinstein handle, we see that $L\cup L'$ is a closed Lagrangian submanifold of $(W,d\lambda)$ which is homeomorphic to the sphere $S^{n}$.  By applying the time $t$ flow of the vector field $-X_{\lambda}$ on $W$ (for $t\in (0,\infty)$ arbitrarily large) and appealing to the Weinstein neighborhood theorem for Lagrangian submanifolds, we see that $(W,\lambda)$ is Liouville homotopic to the cotangent disk bundle $(\disk^{*}(L\cup L'),-\lambda_{can})$.
\end{ex}

\subsubsection{Handle decompositions of Weinstein domains}

Passing through a critical level of a Morse function on a Weinstein cobordism corresponds to attaching a Weinstein handle. The following theorem allows us to use the following working definition: \emph{A Weinstein domain (cobordism) is a Liouville domain (resp. exact symplectic cobordism) built by a finite sequence of Weinstein handles.}

\begin{thm}\label{Thm:WHandle}
Let $\Mxi$ and $(M',\xi')$ be $(2n+1)$-dimensional contact manifolds.
\be
\item Let $\Ldom$ be a Weinstein cobordism with convex end $\Mxi$.  If we attach a Weinstein handle to $W$ along an isotropic sphere $L\subset\Mxi$, then the resulting symplectic cobordism is also Weinstein.
\item A $2n$-dimensional Liouville domain $\Ldom$ is Weinstein if and only if it admits a filtration
\begin{equation*}
\sqcup(\disk^{2n},\lambda_{std})=(W_{0},\lambda_{0})\subset\cdots\subset (W_{n},\lambda_{n})=\Ldom
\end{equation*}
where each $(W_{k},\lambda_{k})$ is a Weinstein domain obtained from $(W_{k-1},\lambda_{k-1})$ by the attachment of a finite number of $2n$-dimensional Weinstein $k$-handles.
\item Similarly, a $(2n+2)$-dimensional symplectic cobordism $(W,\omega)$ from $(M',\xi')$ to $\Mxi$ is Weinstein if and only if it can be obtained from the finite symplectization of $(M',\xi')$ by a finite sequence of Weinstein handle attachments.
\ee
\end{thm}

This observation will be particularly useful in Section \ref{Sec:Applications}.  For further discussion of the topology of Weinstein manifolds with an emphasis on the 4-dimensional case, see \cite{Gompf:Handles, Gompf:Stein, OzbSt:SteinSurgery}. For a more general treatment, see \cite{SteinToWeinstein}.

\subsection{Convex hypersurfaces}\label{Sec:Convex}

Throughout this section, $\Mxi$ will be a fixed $(2n+1)$-dimensional contact manifold.  For simplicity, we only consider closed convex hypersurfaces in this paper.

\begin{defn}\label{Def:Convex}
A \emph{convex hypersurface} in $\Mxi$ is a pair $(S,X)$ consisting of
\be
\item a closed, oriented $2n$-dimensional submanifold $S\subset M$ and
\item a vector field $X$, defined on a neighborhood of $S$, which is positively transverse to $S$ and whose flow preserves $\xi$.
\ee
\end{defn}

Let $(S,X)$ be a convex hypersurface in $\Mxi$.  The vector field $X$ and contact structure $\xi$ provide a decomposition of $S$ into three pieces:
\be\label{Con:ConvexData}
\item the \emph{positive region} $S^{+}$ consisting of all points in $S$ for which $X$ is positively transverse to $\xi$,
\item the \emph{dividing set} $\Gamma_{S}$ consisting of all points in $S$ for which $X\subset \xi$, and
\item the \emph{negative region} $S^{-}$ consisting of all points in $S$ for which $X$ is negatively transverse to $\xi$.
\ee

It is easy to see that if we define $(S',X')=(-S,-X)$, then as oriented manifolds
\begin{equation*}
(S')^{+}=-S^{-},\quad \Gamma_{S'}=\Gamma_{S}, \quad (S')^{-}=-S^{+}.
\end{equation*}
Condition (2) of Definition \ref{Def:Convex} is equivalent to saying that for each contact form $\alpha$ for $\xi$ we have $\Lie_{X}\alpha=G\alpha$ for some smooth function $G$ defined in a tubular neighborhood of $S$.  Now suppose that we identify a neighborhood of $S$ with $N(S):=[-1,1]\times S$, and $X=\partial_{\theta}$ where $\theta$ is a coordinate on $[-1,1]$.  Then we can write $\alpha=f\cdot d\theta+\beta$ for some function $f\in\Cinfty(N(S),\mathbb{R})$ and $\beta\in\Cinfty([-1,1],\Omega^{1}(S))$.  The following proposition allows us to normalize the Lie derivative of $\alpha$ with respect to the vector field $\partial_{\theta}$ on $N(S)$.

\begin{prop}\label{Prop:Normalize}
In the above notation, let $H\in\Cinfty(N(S),\mathbb{R})$ be a smooth function defined in a tubular neighborhood $N(S)$ of a convex hypersurface $(S,\partial_{\theta})$.  Then we can choose a contact form $\alpha$ for $\Mxi$ such that on a neighborhood of $S$,
\begin{equation*}
\Lie_{\partial_{\theta}}(\alpha)=H\alpha.
\end{equation*}
\end{prop}

\begin{proof}
Let $\alpha'$ be a contact form on $N(S)$ satisfying $\Lie_{\partial_{\theta}}\alpha'=G\alpha'$.  We will find a function $F$ so that $\alpha=e^{F}\alpha'$ is as desired.  We have that
\begin{equation*}
\Lie_{\partial_{\theta}}(e^{F}\alpha')=e^{F}\big{(}\frac{\partial F}{\partial \theta}\alpha' + \Lie_{\partial_{\theta}}\alpha'\big{)}=e^{F}\big{(}\frac{\partial F}{\partial \theta}+G\big{)}\alpha'.
\end{equation*}
Therefore, we can find the function $F$ solving the equation $H=\frac{\partial F}{\partial \theta}+G$ by defining
\begin{equation*}
F(z,x)=\int_{0}^{\theta}\bigr{(}H(t,x)-G(t,x)\bigr{)}dt
\end{equation*}
for $\theta\in[-1,1]$ and $x\in S$.
\end{proof}

Taking the function $H$ in Proposition \ref{Prop:Normalize} to be zero, we are guaranteed the existence of a contact form $\alpha$ for $\Mxi$ such that 
\begin{equation}\label{Eq:NSAlpha}
\alpha|_{N(S)}=f\cdot d\theta+\beta    
\end{equation}
where $\frac{\partial f}{\partial \theta} = 0$ and $\Lie_{\partial_{\theta}}\beta=0$.  With respect to this $\theta$-invariant contact form, we can compute
\begin{equation}\label{Eq:Alphadalpha}
\alpha\wedge (d\alpha)^{k} = (f d\theta \wedge d\beta + \beta \wedge d\beta + k\beta\wedge df \wedge d\theta)\wedge (d\beta)^{k-1}
\end{equation}
for $k \geq 1$. By analyzing this equation, we are led to the following proposition, the 3-dimensional case of which was first observed in \cite{Giroux:Convex}.

\begin{prop}
In the above notation, $\Gamma_{S}$ is a closed, non-empty, codimension-1 submanifold of $S$.  When oriented as the boundary of $S^{+}$, it is a non-empty codimension-2 contact submanifold of $\Mxi$. Possibly after multiplying $\alpha$ by a non-vanishing function, the restriction of $d\alpha$ to $S^{+}$ ($S^{-}$) with a collar neighborhood of its boundary removed is symplectic with respect to its positive (resp. negative) orientation inherited from the inclusion in $S$.
\end{prop}

\begin{proof}
For $k = n$, Equation \eqref{Eq:Alphadalpha} shows that the contact condition is
\begin{equation*}
    \alpha \wedge (d\alpha)^{n} = d\theta \wedge (f d\beta + n\beta\wedge df)\wedge (d\beta)^{n-1} > 0
\end{equation*}
with respect to the orientation on $[-1, 1]\times S$. This implies that $0$ must be a regular value of the function $f$ and we see that $\alpha \wedge (d\alpha)^{n-1}$ must then be a volume form for $\Gamma = \{ f = \theta = 0 \}$. Hence $\Gamma$ is a contact submanifold of $M$.

For the statement regarding Liouville forms, possible after multiplying $\alpha$ by a function $S \rightarrow (0, \infty)$, we may assume that the function $f$ in Equation \eqref{Eq:NSAlpha} is constant on the complement of a neighborhood of $\Gamma$ in $S$. Where $f$ is constant, the contact condition becomes $d\theta \wedge (d\beta)^{n}$ so that $(d\beta)^{n} \neq 0$ where $f$ is constant.

If $\Gamma = \emptyset$, then $S$ would equal its positive or negative region. In this case, with this special contact form $\alpha$, we would have $\int_{S}(d\alpha)^{n} \neq 0$ which is impossible by Stokes theorem and our assumption that $S$ is closed.
\end{proof}

For further discussion, see \cite{Giroux:IdealLiouville} in which Giroux defines \emph{ideal Liouville domains}. These objects elegantly encode the geometry of positive and negative regions of convex hypersurfaces.

\subsection{Convex gluing}\label{Sec:ConvexGluing}

Suppose that $(M_{1}, \xi_{1})$ and $(M_{2},\xi_{2})$ are two contact manifolds with convex boundary, where the contact vector fields $X_{1}$ and $X_{2}$ defined on collar neighborhoods of $\partial M_{1}$ and $\partial M_{2}$ point out of $M$ and $M'$, respectively.  If we can identify the convex boundary components of $\Mxi$ and $(M',\xi')$ using the collar neighborhoods of the $M_{j}$, these collar neighborhoods may be identified to obtain a larger contact manifold by \emph{convex gluing}.

\begin{ex}[Convex gluing of cotangent bundles]\label{Ex:SmoothGluing}
Let $L$ and $L'$ be two smooth, compact $(n+1)$-dimensional manifolds with non-empty boundaries which are identified via some orientation reversing diffeomorphism $\Phi:\partial L\rightarrow \partial L'$.  Consider the associated contact manifolds $\Mxi=(S^{*}L,\xi_{can})$ and $(M',\xi')=(S^{*}L',\xi_{can})$. Suppose that $\partial L$ and $\partial L'$ each have collar neighborhoods $[-1,1]\times\partial L$ and $[-1,1]\times\partial L'$ on which we identify $\partial L=\{1\}\times \partial L$ and $\partial L'=\{1\}\times\partial L'$.  Assume that we have fixed Riemannian metrics on $L$ and $L'$ which restrict to product metrics on each collar neighborhood and are such that the map $\Phi$ is an isometry with respect to the induced metrics on $\partial L$ and $\partial L'$.

Let $t$ be a coordinate on $[-1,1]$.  Then $\partial_{t}$ lifts to a vector field on $S^{*}L$ which points out of the boundary of $S^{*}L$. By our choice of metrics, the vector field $\widetilde{\partial_{t}}$ is tangent to $S^{*}L$, and when restricted to $M$, is convex where it is defined.  The dividing set may be identified with $S^{*}(\partial L)$, and the positive and negative regions can each be shown to be diffeomorphic to a tubular neighborhood of the zero-section of the cotangent bundle of $\partial M$.  Define a similar convex vector field on a neighborhood of $\partial S^{*}L'$ which points out of the boundary of $M'$.

By the restrictions imposed upon metrics used to define $S^{*}L$ and $S^{*}L'$, the mapping $\Phi$ provides a diffeomorphism $\widehat{\Phi}:\partial(S^{*}X)\rightarrow \partial(S^{*}L')$.  Then we can use $\widehat{\Phi}$ to perform a convex gluing using the map $(t,x)\mapsto (-t,\Phi(x))$ from $[-1,1]\times \partial M$ to $[-1,1]\times \partial M'$. Under this identification
\begin{equation*}
\Mxi \cup (M',\xi')=(S^{*}(L\cup_{\Phi} L'),\xi_{can}).
\end{equation*}
\end{ex}

\section{Neighborhood constructions and the Liouville connect sum}\label{Sec:StandardNeighborhood}

In this section we provide a rigorous account of the construction of neighborhoods of Liouville hypersurfaces and the Liouville connect sum. We also discuss convex gluings which modify the Liouville connect sum which will later be useful for the constructions of contact manifolds in Section \ref{Sec:Applications}. We begin by collecting some prerequisite results regarding neighborhoods of Liouville hypersurfaces.

We note that our constructions and the analyses of contact forms involved are inspired by and consequently quite similar to the construction of a contact manifolds from open book decompositions with symplectic monodromy. See Remark \ref{Rmk:OBConvention}.

\subsection{Neighborhood theorems and deformations of contact forms}\label{Sec:Neighborhood}

\begin{lemma}\label{Lemma:ReebTransversality}
Suppose that $\Sigma$ is a submanifold of a contact manifold $\Mxi$ and that $\Mxi$ is equipped with a contact form $\alpha$ for which $d\alpha|_{T\Sigma}$ is a symplectic form on $\Sigma$. Then the Reeb vector field $R$ for $\alpha$ is transverse to $\Sigma$.
\end{lemma}

\begin{proof}
If $R$ was tangent to $\Sigma$ at some point $x \in \Sigma$, then $d\alpha(R, \ast)$ would be zero on $T_{x}\Sigma$ in violation of our hypothesis that $d\alpha|_{T\Sigma}$ is non-degenerate.
\end{proof}

\begin{lemma}\label{Lemma:SquareNeighborhood}
Suppose that $\Sdom$ is a Liouville hypersurface contained in the interior of a contact manifold $\Mxi$ and that $\alpha$ is a contact form for $\Mxi$ for which $\alpha|_{T\Sigma}=\beta$.  Then for a sufficiently small positive constant $\epsilon$, there is a neighborhood of $\Sigma$ of the form
\begin{equation*}
N(\Sigma)=[-\epsilon,\epsilon]\times\Sigma\quad\text{satisfying}\quad \alpha|_{N(\Sigma)}=dz+\beta.
\end{equation*}
Here $z$ is a coordinate on $[-\epsilon,\epsilon]$, and $\Sigma=\{0\}\times\Sigma$.
\end{lemma}

\begin{proof}
Define a map $[-\epsilon,\epsilon]\times\Sigma\rightarrow M$ by
\begin{equation*}
(z,x)\mapsto \Flow^{z}_{R_{\alpha}}(x).
\end{equation*}
For $\epsilon>0$ sufficiently small, this will be an embedding by Lemma \ref{Lemma:ReebTransversality} and our presumed compactness for $\Sigma$. As $\alpha$ is $R_{\alpha}$-invariant, it pulls back to $[-\epsilon,\epsilon]\times\Sigma$ as desired.
\end{proof}

The remainder of Section \ref{Sec:Neighborhood} describes how contact forms can be modified on neighborhoods of Liouville hypersurfaces as described in the previous lemma.

\begin{lemma}\label{Lemma:df}
Suppose that $\beta$ and $\beta'$ are two Liouville forms on a compact manifold $\Sigma$ which agree on a collar neighborhood of $\partial \Sigma$ and satisfy $d\beta=d\beta'$.  Then there is an isotopy $\phi_{t}$, $t\in[0,1]$, of $\Sigma$ such that
\be
\item $\phi_{0}=\Id_{\Sigma}$ and $\phi_{1}^{*}\beta-\beta'=df$ for some smooth function $f$ on $\Sigma$ which vanishes on a collar neighborhood of $\partial\Sigma$,
\item there is a collar neighborhood of $\partial\Sigma$ on which $\phi_{t}$ is the identity mapping for all $t\in[0,1]$, and
\item $\phi^{*}_{t}d\beta=d\beta$ for all $t\in[0,1]$.
\ee
\end{lemma}

See \cite[Corollary 5]{Giroux:IdealLiouville} for a similar result.

\begin{proof}
Define a vector field $V$ on $\Sigma$ as the unique solution to the equation $d\beta(V,\ast)=\beta'-\beta$.  Then $\MCL_{V}(d\beta)=0$ and $\MCL_{V}(\beta'-\beta)=0$
so that $\Flow^{t}_{V}$ preserves $d\beta$ and $\beta'-\beta$.  Moreover, $\Flow^{t}_{V}$ is equal to the identity on a collar neighborhood of $\partial\Sigma$.  Now we calculate
\begin{equation*}
\begin{aligned}
\frac{\partial}{\partial t}\big( (\Flow^{t}_{V})^{*}(\beta)\big) &= (\Flow^{t}_{V})^{*}(\MCL_{V}\beta)\\
&= (\Flow^{t}_{V})^{*}\big{(}d\beta(V,\ast)+d(\beta(V))\big{)}\\
&= (\Flow^{t}_{V})^{*}\big{(}\beta'-\beta+d(\beta(V))\big{)}\\
&=\beta'-\beta + dg_{t}
\end{aligned}
\end{equation*}
where $g_{t} = (\Flow^{t}_{V})^{*}(\beta(V))$. By our hypotheses that $\beta - \beta'$ is supported on the complement of a collar neighborhood of $\partial \Sigma$, there is a collar neighborhood of $\partial \Sigma$ on which $V$, and hence $\beta(V)$ and $g_t$ vanish for all $t$.

Then $(\Flow^{1}_{V})^{*}\beta=\beta' +df$ where $f=\int_{0}^{1}g_{t}dt$.  Defining $\phi_{t}=\Flow^{t}_{V}$ completes the proof.
\end{proof}

\begin{lemma}\label{Lemma:PosNegRegionPerturbation}
Let $I$ be a connected $1$ manifold parameterized by a variable $\theta$ and let $\beta_{\theta}$ be an $I$-family of $1$ forms on a $(2n)$-dimensional manifold $\Sigma$. Then the $1$ form
\begin{equation*}
\alpha = d\theta + \beta_{\theta}
\end{equation*}
is contact on $I \times \Sigma$ if and only if
\begin{equation*}
d\beta_{\theta} + \frac{\partial \beta_{\theta}}{\partial \theta}\wedge \beta_{\theta}
\end{equation*}
is a non-degenerate $2$-form on each $\{ \theta \}\times \Sigma$.
\end{lemma}

\begin{proof}
We compute
\begin{equation*}
\begin{gathered}
d\alpha = d\theta \wedge \frac{\partial \beta_{\theta}}{\partial \theta} + d_{\Sigma}\beta_{\theta},\quad (d\alpha)^{n} = (d_{\Sigma}\beta_{\theta})^{n-1}\wedge ( n d\theta \wedge \frac{\partial \beta_{\theta}}{\partial \theta} + d_{\Sigma}\beta_{\theta})
\end{gathered}
\end{equation*}
where $d_{\Sigma}\beta_{\theta}$ is the restriction of $d\beta_{\theta}$ to tangent spaces of the $\{ \theta \}\times \Sigma$. Then
\begin{equation*}
\begin{aligned}
\alpha \wedge (d\alpha)^{n} &= (d\theta + \beta_{\theta})\wedge (d_{\Sigma}\beta_{\theta})^{n-1}\wedge ( n d\theta \wedge \frac{\partial \beta_{\theta}}{\partial_{\theta}} + d_{\Sigma}\beta_{\theta})\\
&= d\theta \wedge (d_{\Sigma}\beta_{\theta})^{n-1}\wedge (d_{\Sigma}\beta + n\frac{\partial \beta_{\theta}}{\partial \theta}\wedge \beta_{\theta}).
\end{aligned}
\end{equation*}
If follows that $\alpha$ is contact if and only if 
\begin{equation*}
(d_{\Sigma}\beta_{\theta})^{n-1}\wedge (d_{\Sigma}\beta + n\frac{\partial \beta_{\theta}}{\partial \theta}\wedge \beta_{\theta}) = (d_{\Sigma}\beta_{\theta} + \frac{\partial \beta_{\theta}}{\partial \theta}\wedge \beta_{\theta})^{n}
\end{equation*}
is a volume form on each $\{ \theta \} \times \Sigma$. The above computations shows that satisfaction of the contact condition is equivalent to the non-degeneracy condition appearing in the statement of the lemma.
\end{proof}

\begin{lemma}\label{Lemma:GammaPerturbation}
Suppose that $\alpha'$ is a contact form on a $(2n-1)$-dimensional manifold $M'$. Let $I_{\theta}$ and $I_{s}$ be connected $1$-manifolds parameterized by variables $\theta$ and $s$ respectively. Then for functions
\begin{equation*}
f(s), g(\theta, s), h(s),
\end{equation*}
the $1$ form
\begin{equation*}
\alpha = f(s)d\theta + g(\theta, s)ds + h(s)\alpha'
\end{equation*}
is a contact form on the $(2n+1)$-dimensional manifold $M = I_{\theta} \times I_{s} \times M'$ if and only if
\begin{equation*}
h^{n-1}\Big(f\frac{\partial h}{\partial s} - \big(\frac{\partial f}{\partial s} - \frac{\partial g}{\partial \theta}\big) h \Big) > 0
\end{equation*}
for all $s \in I_{s}$.
\end{lemma}

\begin{proof}
We need to compute the contact condition for $\alpha$ as it is defined above:
\begin{equation*}
\begin{aligned}
d\alpha &= ds \wedge \Big(\big(\frac{\partial f}{\partial s} - \frac{\partial g}{\partial \theta}\big)d\theta + \frac{\partial h}{\partial s}\alpha'\Big) + hd\alpha' \\
(d\alpha)^{n} &= n h^{n-1} ds \wedge \Big(\big(\frac{\partial f}{\partial s} - \frac{\partial g}{\partial \theta}\big)d\theta + \frac{\partial h}{\partial s}\alpha'\Big)\wedge(d\alpha')^{n-1} \\
\alpha \wedge (d\alpha)^{n} &= n h^{n-1}\Big(f \frac{\partial h}{\partial s} - \big(\frac{\partial f}{\partial s} - \frac{\partial g}{\partial \theta}\big) h \Big)d\theta \wedge ds \wedge \alpha' \wedge (d\alpha')^{n-1}.
\end{aligned}
\end{equation*}
\end{proof}

\subsection{Construction of $\mathcal{N}(\Sigma)$}\label{Sec:Corners}

We now give a rigorous description of the edge-rounding on the neighborhood
\begin{equation*}
N(\Sigma)=[-\epsilon,\epsilon]\times\Sigma
\end{equation*}
provided by Lemma \ref{Lemma:SquareNeighborhood}, producing a model neighborhood $\mathcal{N}(\Sigma)$ of $\Sigma$ with smooth, convex boundary as described in Section \ref{Sec:ConnectSum}.

Due to the Weinstein neighborhood theorem for contact type hypersurfaces in symplectic manifolds , we can decompose $\Sigma$ into two parts $\Sigma=\widehat{\Sigma}\cup C$.  Here $\widehat{\Sigma}$ is diffeomorphic to $\Sigma$ and is disjoint from $\partial\Sigma$. The manifold $C$ is a collar neighborhood of $\partial\Sigma$ of the form 
\begin{equation*}
C = [\frac{1}{2},1]\times\partial\Sigma
\end{equation*}
where $\{1\}\times\partial\Sigma=\partial\Sigma$.  Taking a coordinate $t$ on $[\frac{1}{2},1]$ we can assume that
\begin{equation*}
\beta|_{C}=t\cdot\alpha'
\end{equation*}
for some contact form $\alpha'$ on $\partial\Sigma$.  This induces a decomposition of $[-\epsilon,\epsilon]\times\Sigma$ into two pieces
\begin{equation*}
[-\epsilon,\epsilon]\times\Sigma=\big([-\epsilon,\epsilon]\times \widehat{\Sigma}\big) \cup\big([-\epsilon,\epsilon]\times C \big).
\end{equation*}

\begin{figure}[h]
	\begin{overpic}[scale=.7]{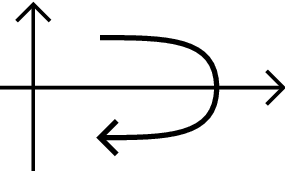}
        \put(-2,50){$z$}
        \put(103,25){$t$}
        \put(49,2){$\gamma$}
    \end{overpic}
    \vspace{1.5mm}
	\caption{The curve $(z,t)=(z(s),t(s))$ whose image is denoted by $\gamma$.}
    \label{Fig:RoundingCurve}
\end{figure}

To smooth the corners of $[-\epsilon,\epsilon]\times\Sigma$ we can then focus our attention on $[-\epsilon,\epsilon]\times[\half,1]\times\partial\Sigma$.  Let 
\begin{equation*}
(z,t)=(z,t)(s):[-1,1]\rightarrow[-\epsilon,\epsilon]\times[\half,1]
\end{equation*}
be a smooth curve satisfying the following conditions:
\be
\item $(z,t)(-1)=(\epsilon,\half)$, $\partial_{s}(z,t)(-1)=(0,1)$, and $(\partial_{s})^{k}(z,t)(-1)=(0,0)$ for all $k>1$.
\item $(z,t)(1)=(-\epsilon,\half)$, $\partial_{s}(z,t)(1)=(0,-1)$, and $(\partial_{s})^{k}(z,t)(1)=(0,0)$ for all $k>1$.
\item $(z,t)(-s)=(-z,t)(s)$ for all $s\in[-1,1]$.
\item The one form $zdt-tdz$ evaluated at $\partial_{s}(z,t)$ is always positive.
\ee
Write $\gamma$ for the curve $(z,t)$ in $[-\epsilon,\epsilon]\times[\half,1]$.  See Figure \ref{Fig:RoundingCurve}.

\begin{defn}
In the above notation, let $\mathcal{N}(\Sigma)$ be the region in $[-\epsilon,\epsilon]\times \Sigma$ containing $[-\epsilon, \epsilon]\times\widehat{\Sigma}$ and bounded by
\begin{equation*}
\big{(}\{-\epsilon,\epsilon\}\times \widehat{\Sigma}\big{)}\cup\big{(} \gamma\times\partial\Sigma\big{)}.
\end{equation*}
Here, $\gamma\times\partial\Sigma$ is considered as a subset of $[-\epsilon,\epsilon]\times C$.  We call a neighborhood $\mathcal{N}(\Sigma)$ of $\Sigma$ constructed in this fashion a \em{standard neighborhood} of the Liouville hypersurface $\Sdom\subset\Mxi$.
\end{defn}

Provided a Liouville vector field $X_{\beta}$ for $(\Sigma, \beta)$, the vector field $V_{\beta}$ of Equation \eqref{Eq:Vbeta} is defined on $\mathcal{N}(S)$ is positively transverse to $\partial \mathcal{N}(S)$ and satisfies $\Lie_{V_{\beta}}(\alpha) = \alpha$. Thus $(\partial \mathcal{N}(S), V_{\beta})$ is a convex surface in $\Mxi$ as described in Section \ref{Sec:Convex}.

\subsection{Contact forms on $[-\delta,\delta]\times\partial \mathcal{N}(\Sigma)$}\label{Sec:ThetaInvariant}

In order to perform a convex gluing along two copies of $\partial \mathcal{N}(\Sigma)$ using (a modification of) the map $\Upsilon$ described in Section \ref{Sec:SurgeryOutline}, we must analyze contact forms in a neighborhood of this hypersurface. The cost we pay for having such a simple description for the map $\Upsilon$ is that the modification of contact forms will necessarily be non-trivial.

The hypersurface $\partial \mathcal{N}(\Sigma)$ is smooth and transverse to the vector field $V_{\beta}=z\partial_{z}+X_{\beta}$ whose flow is defined for time $T \in (-\infty, \delta)$ for some $\delta > 0$. By the construction of the neighborhood $\mathcal{N}(\Sigma)$ of $\Sigma$ in the previous section, we can write
\begin{equation*}
\alpha|_{\partial \mathcal{N}(\Sigma)}=
\begin{cases}
\beta & \quad\text{on}\quad \{-\epsilon,\epsilon\}\times \widehat{\Sigma}\\
\frac{\partial z}{\partial s}ds+t(s)\cdot\alpha' & \quad\text{on}\quad \gamma\times\partial\Sigma.
\end{cases}
\end{equation*}
Identify a tubular neighborhood of $\partial \mathcal{N}(\Sigma)$ with $[-\delta,\Theta]\times \partial \mathcal{N}(\Sigma)$ where $\partial_{\theta}=-V_{\beta}$.  Here $\theta$ is a coordinate on $[-\delta,\Theta]$ with $\Theta$ an arbitrarily large positive constant.  Our choice of sign for $V_{\beta}$ is chosen so that $\partial_{\theta}$ provides an outward pointing normal vector if we cut $\Flow^{\Theta}_{-V_{\beta}}(\mathcal{N}(\Sigma))$ out of $M$. Then
\begin{equation*}
\alpha(\partial_{\theta})=-z, \quad \Lie_{\partial_{\theta}}\alpha=-\alpha.
\end{equation*}
Therefore
\begin{equation*}
e^{\theta}\alpha =
\begin{cases}
\mp \epsilon d\theta +\beta & \quad\text{on}\quad [-\delta,\Theta]\times\{\pm\epsilon\}\times \widehat{\Sigma}\\
-z(s)d\theta + \frac{\partial z(s)}{\partial s}ds + t(s)\alpha' & \quad\text{on}\quad [-\delta,\Theta]\times\gamma\times\partial\Sigma.
\end{cases}
\end{equation*}
where $\partial \mathcal{N}(\Sigma)$ is identified with the level set $\{0\}\times\partial \mathcal{N}(\Sigma)\subset[-\delta,\Theta]\times\partial \mathcal{N}(\Sigma)$ of the function $\theta$.  The contact form $\alpha$ can then be normalized over $[-\delta, \Theta] \times \partial \mathcal{N}(\Sigma)$ to obtain a $\theta$-invariant contact form 
\begin{equation*}
\alpha_{0} = e^{\theta}\alpha.
\end{equation*}

The $\frac{\partial z(s)}{\partial s}ds$ term in the above equation will be inconvenient for convex gluing (using the simple expression $\widehat{\Upsilon}$ appearing below) so we describe how it can be removed on some \begin{equation*}
I \times \partial \mathcal{N}(\Sigma) \subset [-\delta, 4\Theta] \times \partial \mathcal{N}(\Sigma),\quad \Theta > 0
\end{equation*} 
using a Moser argument. Pick a small $\delta > 0$ and consider functions $G(\theta): [-\delta, 4\Theta] \rightarrow \mathbb{R}$ satisfying
\be
\item $G = 1$ on $[-\delta, 0] \cup [3\Theta, 4\Theta]$,
\item $G = 0$ on $[\Theta-\delta, 2\Theta]$.
\ee
For any positive constant $C > 0$, we can choose $\Theta$ to be large enough so that $\frac{\partial G}{\partial \theta}$ is bounded in absolute value point-wise by $C$. See Figure \ref{Fig:GluingCutoff}.

\begin{figure}[h]
\begin{overpic}[scale=.6]{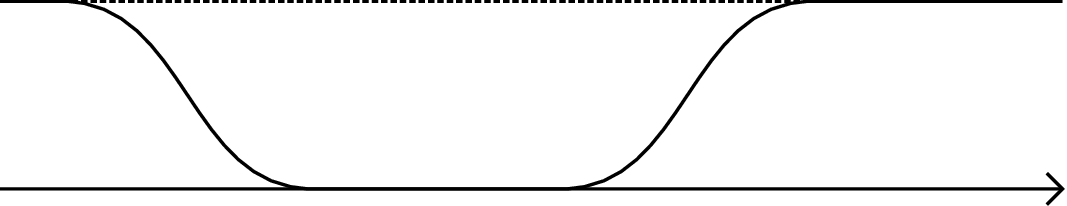}
\put(7, -4){$0$}
\put(28, -4){$\Theta$}
\put(52, -4){$2\Theta$}
\put(75, -4){$3\Theta$}
\put(97, -4){$4\Theta$}
\end{overpic}
\vspace{2.5mm}
\caption{A schematic for the function $G$. The value $1$ is depicted as a dashed arc.}
\label{Fig:GluingCutoff}
\end{figure}

Using the function $G$, we can define a $1$ form $\widehat{\alpha}$ on $[-\delta, 4\Theta]\times \partial \mathcal{N}(\Sigma)$ by deforming $\alpha_{0}$ as follows
\begin{equation}\label{Eq:GPerturbedAlpha}
\alpha_{G} =
\begin{cases}
\mp \epsilon d\theta +\beta & \quad\text{on}\quad [-\delta,4\Theta]\times\{\pm\epsilon\}\times \widehat{\Sigma}\\
-z(s)d\theta + G(\theta)\frac{\partial z(s)}{\partial s}ds + t(s)\alpha' & \quad\text{on}\quad [-\delta,4\Theta]\times\gamma\times\partial\Sigma.
\end{cases}
\end{equation}

Denote by $\widehat{\alpha}$ a choice of $\alpha_{G}$ for which $\widehat{\alpha}$ is contact. As described above, such a choice is possible by
\be
\item making $\Theta$ large,
\item bounding $|\frac{\partial G}{\partial \theta}|$ point-wise,
\item and then applying Lemma \ref{Lemma:GammaPerturbation}.
\ee
For such a choice of $G$, we have that
\begin{equation*}
T \alpha_{0} + (1 - T)\widehat{\alpha}
\end{equation*}
is constant in $T$ outside of a compact set and is contact for all $T \in [0, 1]$. Hence the contact structures determined by $\alpha$ and $\widehat{\alpha}$ is isotopic by Moser's argument. 

Simplifying notation, we henceforth write $\theta$ for what we have previously expressed as $\theta + 2\Theta - \delta$. In our new notation,
\begin{equation}\label{Eq:GluingAlpha}
\widehat{\alpha} =
\begin{cases}
\mp \epsilon d\theta +\beta & \quad\text{on}\quad [-\Theta, \delta]\times\{\pm\epsilon\}\times \widehat{\Sigma}\\
-z(s)d\theta + t(s)\alpha' & \quad\text{on}\quad [-\Theta, \delta]\times\gamma\times\partial\Sigma.
\end{cases}
\end{equation}
It will be of importance in Section \ref{Sec:GluingInstructions} below that $\Theta$ may be chosen taken to be arbitrarily large.

\subsection{Convex gluing for the Liouville connect sum}

Now we show that the map $\Upsilon$ from Equation \eqref{Eq:ConnectSum} determines a convex gluing.  Again consider two disjoint Liouville embeddings $i_{1},i_{2}:\Sdom\rightarrow\Mxi$.  The construction of Section \ref{Sec:Corners} above provides two disjoint neighborhoods $\mathcal{N}(i_{1}(\Sigma))$ and $\mathcal{N}(i_{2}(\Sigma))$ of $i_{1}(\Sigma)$ and $i_{2}(\Sigma)$, respectively.  The construction in that section also provides us with collar neighborhoods 
\begin{equation*}
[-\Theta,\delta]\times\partial \mathcal{N}(i_{1}(\Sigma)), \quad [-\Theta,\delta]\times \partial \mathcal{N}(i_{2}(\Sigma))
\end{equation*}
and a contact form $\widehat{\alpha}$ which is, in each of the collar neighborhoods, $\theta$-invariant.  Then by the conditions defining the curve $\gamma$ (used to smooth the corners of $[-\epsilon,\epsilon]\times\Sigma$) and the explicit formula for $\hat{\alpha}$ in Equation \eqref{Eq:GluingAlpha}, the map
\begin{equation}\label{Eq:Upsilon}
\begin{gathered}
\widehat{\Upsilon}: [-\delta,\delta]\times\partial \mathcal{N}(i_{1}(\Sigma))\rightarrow[-\delta,\delta]\times \partial \mathcal{N}(i_{2}(\Sigma)),\\
\widehat{\Upsilon}(\theta,x)=(-\theta,\Upsilon(x))
\end{gathered}
\end{equation}
satisfies
\begin{equation*}
\widehat{\Upsilon}^{*}\hat{\alpha}|_{[-\delta,\delta]\times \partial \mathcal{N}(i_{2}(\Sigma))}=\widehat{\alpha}|_{[-\delta,\delta]\times\partial \mathcal{N}(i_{1}(\Sigma))}.
\end{equation*}
Hence this map can be used to perform the desired convex gluing which defines the Liouville connect sum as described in Section \ref{Sec:ConnectSum}.

\subsection{Modifications of $\widehat{\Upsilon}$}\label{Sec:GluingInstructions}

Now we discuss how the gluing map $\widehat{\Upsilon}$ defined in the previous section can be modified using a pair of symplectomorphisms
\begin{equation*}
\Phi, \Psi \in \Symp(\Sdom, \partial \Sigma)
\end{equation*}
This will be useful for the constructions of contact manifolds in Section \ref{Sec:Applications}. The construction of this section is summarized in Figure \ref{Fig:GluingInstructions}. See Remark \ref{Rmk:OBConvention} regarding our choice of gluing convention. Our goals are to show that this schematic description determines a convex gluing -- after an appropriate perturbation -- and that the contact manifold determined by such a gluing depends only on the isotopy classes of $\Phi$ and $\Psi$ within $\Symp(\Sdom, \partial \Sigma)$.

\begin{figure}[h]
\begin{overpic}[scale=.55]{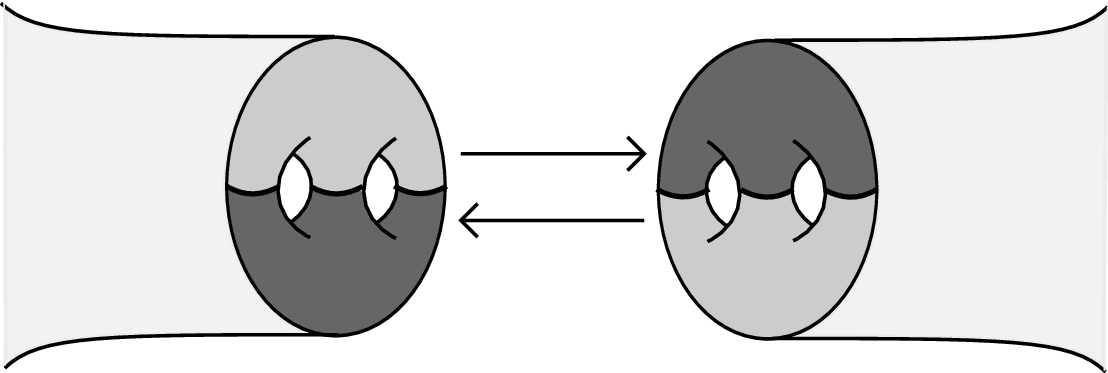}
	\put(30,-3.5){$S_{1}$}
	\put(68,-3.5){$S_{2}$}
	\put(48,22){$\Phi$}
	\put(48,8){$\Psi$}
\end{overpic}
\vspace{4.5mm}
\caption{This figure gives a schematic description of a convex gluing performed using convex gluing instructions $(\Phi,\Psi)$ (without the correction isotopies $\phi_{t}$ and $\psi_{t}$).  Our convention is such that the maps $\Phi$ and $\Psi$ each send the positive region of one convex boundary component of $\Mxi$ to the negative region of the other convex boundary component.}
\label{Fig:GluingInstructions}
\end{figure}

Suppose, as in the previous section, that we have a contact manifold $\Mxi$, disjoint collar neighborhoods $[-\Theta,\delta]\times \partial \mathcal{N}(i_{j}(\Sigma)), j=1,2$ of two convex boundary components of $M$, and a fixed identification $\widehat{\Upsilon}$ between the $[-\delta, \delta]\times \mathcal{N}(i_{j}(\Sigma))$ as described in Equation \eqref{Eq:Upsilon}.  For notational simplicity, we write
\begin{equation*}
S_{1}=\partial\mathcal{N}(i_{1}(\Sigma)),\quad S_{2}= \partial\mathcal{N}(i_{2}(\Sigma))
\end{equation*}
throughout the remainder of this section. We assume that we have a contact form $\widehat{\alpha}$ as described in Equation \eqref{Eq:GluingAlpha} with the parameter $\Theta > $ arbitrarily large and $\delta > 0$ arbitrarily small.\footnote{In the setup of the previous subsection, we are allowed to have $\Theta$ as large as we like using the fact that the flow of the vector field $-V_{\beta}$ is defined for all positive time. In an abstract setup -- a contact manifold with this type of convex boundary -- we may extend a neighborhood of the boundary by attaching an arbitrarily large collar.} Note that using the outward pointing vector field $\partial_{\theta}$ as in the previous subsection that $\{ -\epsilon \}\times \widehat{\Sigma}$ is contained in the positive region of $\partial M$ and that $\{ \epsilon \} \times \widehat{\Sigma}$ is contained in the negative region of $\partial M$.

Let $\Phi$ and $\Psi$ be symplectomorphisms as described above which we assume to coincide with the identity outside of $\Int(\widehat{\Sigma})$. Writing $x$ for points in $S_{1}$, we define a diffeomorphism 
\begin{equation}\label{Eq:TwistedUpsilonDef}
\begin{gathered}
\widehat{\Upsilon}_{(\Phi, \Psi)}: [-\delta, \delta] \times S_{1} \rightarrow [-\delta, \delta] \times S_{2}\\
\widehat{\Upsilon}_{(\Phi, \Psi)}(\theta, x) = \begin{cases}
(-\theta, \Phi(x)) \in \{ \epsilon \}\times \widehat{\Sigma} & \quad x \in \{ -\epsilon \} \times \Sigma \\
(-\theta, \Psi^{-1}(x)) \in \{ -\epsilon \}\times \widehat{\Sigma} & \quad x \in \{ \epsilon \} \times \Sigma \\
\widehat{\Upsilon}(\theta, x) & \quad \text{on}\quad [-\delta, \delta]\times \gamma \times \partial\Sigma.
\end{cases}
\end{gathered}
\end{equation}

Along $[-\delta, \delta]\times \gamma \times \partial\Sigma$, we have $\widehat{\Upsilon}_{(\Phi, \Psi)}^{*}\widehat{\alpha} = \widehat{\alpha}$. However along $[-\delta, \delta]\times \{ \pm \epsilon \}\times \widehat{\Sigma}$ we are only guaranteed to have
\begin{equation}\label{Eq:BetaPMDef}
\widehat{\Upsilon}_{(\Phi, \Psi)}^{\ast}\widehat{\alpha} = \mp \epsilon d\theta + \beta_{\pm}
\end{equation}
for $\beta_{\pm}\in \Omega^{1}(\{\pm \epsilon\} \times \widehat{\Sigma})$ for which $d\beta_{\pm} = d\beta$ and $(\beta_{\pm} - \beta)$ is supported on $\Int(\widehat{\Sigma})$. In order for this to be a contact gluing, we must connect the $\beta_{\pm}$ to $\beta$. We choose a function 
\begin{equation*}
H(\theta): [-\Theta, \delta] \rightarrow [0, 1]
\end{equation*}
for which $H = 1$ along $[0, \delta]$ and $H = 0$ along $[-\Theta, -\Theta + \delta]$. With a fixed choice of $H$, define
\begin{equation*}
\beta_{\theta, \pm} = H(\theta)\beta_{\pm} + (1 - H(\theta))\beta.
\end{equation*}

Using the fact that $\Theta$ can be chosen arbitrarily large, we can ensure that $\frac{\partial H}{\partial \theta}$ -- and so $\frac{\partial \beta_{\theta, \pm}}{\partial \theta}$ -- is small enough so that the $1$-form
\begin{equation}\label{Eq:AlphaGluingInstructions}
\widehat{\alpha}_{\Phi, \Psi} =
\begin{cases}
\mp \epsilon d\theta +\beta_{\theta, \pm} & \quad\text{on}\quad [-\Theta, \delta]\times\{\pm\epsilon\}\times \widehat{\Sigma}\\
-z(s)d\theta + t(s)\alpha' & \quad\text{on}\quad [-\Theta, \delta]\times\gamma\times\partial\Sigma.
\end{cases}
\end{equation}
is contact over the collar neighborhood of $S_{1}$ and extends $\widehat{\Upsilon}_{(\Phi, \Psi)}^{*}\widehat{\alpha}$. Here we are applying Lemma \ref{Lemma:PosNegRegionPerturbation}. We can therefore use this contact form to define our gluing operation.

Note that a $1$-parameter family $H_{T}$ of such cut-off functions will produce a $1$-parameter family of contact forms on the surgered contact manifold. Because the support of $H_{0} - H_{T}$ is compact for all $T \in [0, 1]$, the $1$-parameter family of contact forms produced will determine isotopic contact forms by Moser's argument. Hence our definition is independent of the choice of $H$. 

\begin{defn}
We say that the contact manifold obtained from $\Mxi$ by the above construction is \emph{described by convex gluing instructions} $(\Phi,\Psi)$.
\end{defn}

The following lemma -- which will be important in applications -- shows that the construction only depends on the $\Phi$ and $\Psi$ up to deformation in $\SympGroup$.

\begin{lemma}\label{Lemma:GluingInvariance}
Suppose that $\Phi_{T}$ and $\Psi_{T}$ are paths in $\Symp(\Sdom, \partial\Sigma)$ parameterized by $T \in [0, 1]$. Then the contact manifolds determined by the gluing instructions $(\Phi_{0}, \Psi_{0})$ and $(\Phi_{1}, \Psi_{1})$ are contactomorphic.
\end{lemma}

\begin{proof}
We simply update the above construction where required to incorporate $1$-parameter families and will use similar notation: Consider the $T \in [0, 1]$ family of diffeomorphisms 
\begin{equation*}
\widehat{\Upsilon}_{(\Phi_{T}, \Psi_{T})}: [-\delta, \delta] \times S_{1} \rightarrow [-\delta, \delta] \times S_{2}
\end{equation*}
as described in Equation \eqref{Eq:TwistedUpsilonDef}. Define a manifold $W = [0, 1] \times M / \sim$ using the relation
\begin{equation*}
\big( T, (\theta, x)\big) \sim \big(T,\widehat{\Upsilon}_{(\Phi_{T}, \Psi_{T})}(\theta, x)\big)\ \text{for}\ (\theta, x) \in [-\delta, \delta] \times S_{1}.
\end{equation*}
By the definition of $\sim$, each $M_{T_{0}} = \{ T = T_{0} \} \subset W$ is the smooth manifold determined by the convex gluing instructions $\widehat{\Upsilon}_{(\Phi_{T_{0}}, \Psi_{T_{0}})}$.

Following Equation \eqref{Eq:BetaPMDef}, we define $1$-forms
\begin{equation*}
\beta_{-, T} = \Phi_{T}^{\ast}\beta,\quad  \beta_{+, T} = \Psi_{T}^{\ast}\beta
\end{equation*}
on $\Sigma$. Along the subsets
\begin{equation*}
[0, 1] \times [-\Theta, \delta] \times \{ \pm \epsilon \} \times \Sigma \subset [0, 1] \times [-\Theta, \delta] \times S_{1}
\end{equation*}
of $W$ we define $1$-forms $\beta_{\theta, \pm, T}$ by the formula
\begin{equation*}
\beta_{\theta, \pm, T} = H(\theta)\beta_{\pm, T} + (1 - H(\theta))\beta
\end{equation*}
where $H$ is described as above. The $1$-form $\widehat{\alpha}_{\Phi_{T}, \Psi_{T}} \in \Omega^{1}(W)$ determined by the formula
\begin{equation*}
\widehat{\alpha}_{\Phi_{T}, \Psi_{T}} =
\begin{cases}
\mp \epsilon d\theta +\beta_{\theta, \pm, T} & \quad\text{on}\quad \{ T\} \times [-\Theta, \delta]\times\{\pm\epsilon\}\times \widehat{\Sigma}\\
-z(s)d\theta + t(s)\alpha' & \quad\text{on}\quad \{ T\}\times [-\Theta, \delta]\times\gamma\times\partial\Sigma.
\end{cases}
\end{equation*}
is then well defined and restricts to $\alpha_{\Phi_{T_{0}}, \Psi_{T_{0}}}$ on each $M_{T_{0}}$ as defined in Equation \eqref{Eq:AlphaGluingInstructions}. By the compactness of $[0, 1]$, we may choose $H$ so that $\frac{\partial \beta_{\theta, \pm, T}}{\partial \theta}$ is small enough to ensure that the restriction of $\alpha_{\Phi_{T}, \Psi_{T}}$ to each $M_{T_{0}}$ is a contact form. Moreover, our construction is such that each
\begin{equation*}
\xi_{T_{0}} = \ker(\alpha_{\Phi_{T}, \Psi_{T}}|_{M_{T_{0}}}) \subset TM_{T_{0}}
\end{equation*}
is the contact structure determined by the gluing instructions $(\Phi_{T_{0}}, \Psi_{T_{0}})$ on $M_{T_{0}}$.

Noting that $dt$ is well defined on $W$, let $Z$ be a vector field on $W$ such that $dt(Z) = 1$ and $Z = \partial_{t}$ away from the gluing region where $W$ is obviously a product manifold. Then the restriction of the time $T$ flow, $\Flow^{T}_{Z}$, of $Z$ to $M_{0}$ is well defined for $T\in [0, 1]$ and maps $M_{0}$ diffeomorphically to $M_{T}$. Moreover
\begin{equation*}
(\Flow^{T}_{Z})^{\ast} \alpha_{\Phi_{T}, \Psi_{T}} \in \Omega^{1}(M_{0})
\end{equation*}
is a $1$-parameter family of contact forms. As the $(\Flow^{T}_{Z})^{\ast} \alpha_{\Phi_{T}, \Psi_{T}}$ all coincide away from the gluing region, we may apply Moser's argument to conclude that each 
\begin{equation*}
(M_{0}, \ker\big((\Flow^{T}_{Z})^{\ast} \alpha_{\Phi_{T}, \Psi_{T}}\big),\quad T \in [0, 1]
\end{equation*}
is a contact structure homotopic to the one determined by the convex gluing instructions $(\Phi_{0}, \Psi_{0})$. Observing that $(M_{0}, \ker((\Flow^{1}_{Z})^{\ast} \alpha_{\Phi_{1}, \Psi_{1}}))$ equivalent to $(M_{1}, \ker(\alpha_{\Phi_{1}, \Psi_{1}}))$, the proof is complete.
\end{proof}

\section{Symplectic cobordisms associated to Liouville connect sums}\label{Sec:Cobordism}

This section is devoted to the proof of Theorem \ref{Thm:Cobordism}.  In Section \ref{Sec:Handle} we prove the first part of the theorem, establishing the existence of the cobordism $\Ldom$.  Then in Section \ref{Sec:WHandleDecomposition}, we prove the second statement of Theorem \ref{Thm:Cobordism} by showing that when the Liouville hypersurface $\Sdom$ is Weinstein, the cobordism $\Ldom$ admits a Weinstein handle decomposition.  Contact manifolds appearing as the convex boundaries of weak symplectic cobordisms are dealt with in Section \ref{Sec:WeakHandle}.

Throughout this section we will be building upon the notation and analysis appearing in Section \ref{Sec:StandardNeighborhood}.

\subsection{Symplectic handle attachment}\label{Sec:Handle}

We construct a handle $\mathcal{H}_{\Sigma}$ from a standard neighborhood $\mathcal{N}(\Sigma)$ of a Liouville hypersurface $\Sdom\subset\Mxi$ as described in Section \ref{Sec:Neighborhood}, and then attach $\mathcal{H}_{\Sigma}$ to the convex boundary of the finite symplectization $([-\epsilon, 0]\times M, e^{t}\alpha)$ of $\Mxi$.\\

\noindent\textbf{Step 1: Setup.}  Let $\Mxi$ be a $(2n+1)$-dimensional contact manifold with a fixed contact form $\alpha$ and let $\Sdom$ be a $2n$-dimensional Liouville domain.  Let $i_{1}$ and $i_{2}$ be embeddings of $\Sigma$ into $M$ whose images are disjoint and satisfy $i_{1}^{*}\alpha=i_{2}^{*}\alpha=\beta$.  By the results of Section \ref{Sec:Neighborhood} there exist tubular neighborhoods $\mathcal{N}(i_{j}(\Sigma))$ of $i_{j}(\Sigma)$, $j=1,2$, with smooth convex boundary.

We fix an additional copy of a standard neighborhood $\mathcal{N}(\Sigma)$ from which $\mathcal{H}_{\Sigma}$ will be constructed.  The contact form $\alpha$ is assumed to take the form $\alpha=dz+\beta$ in each of these neighborhoods.\\

\noindent\textbf{Step 2: Model geometry on the handle.}  Our construction of $\mathcal{H}_{\Sigma}$ begins with the description of a simplified handle $H_{\Sigma}$. Consider the symplectic manifold
\begin{equation*}
(H_{\Sigma},\omega_{\beta})=([-1,1]\times \mathcal{N}(\Sigma),d\theta\wedge dz+d\beta),
\end{equation*}
where $\theta$ is the coordinate on $[-1,1]$.  The vector field $V_{\beta}=z\partial_{z}+X_{\beta}$ -- defined in Equation \eqref{Eq:Vbeta} and studied in Section \ref{Sec:Dilations} -- on $\mathcal{N}(\Sigma)$ satisfying $\Lie_{V_{\beta}}\alpha=\alpha$, where $X_{\beta}$ is the Liouville vector field for $\Sdom$. Viewing $V_{\beta}$ as a $\theta$-invariant vector field on $H_{\Sigma}$, 
\begin{equation*}
\omega_{\beta}(V_{\beta},\ast)=-z d\theta+\beta, \quad \MCL_{V_{\beta}}\omega_{\beta}=\omega_{\beta}.
\end{equation*}

\begin{figure}[h]
\begin{overpic}[scale=.7]{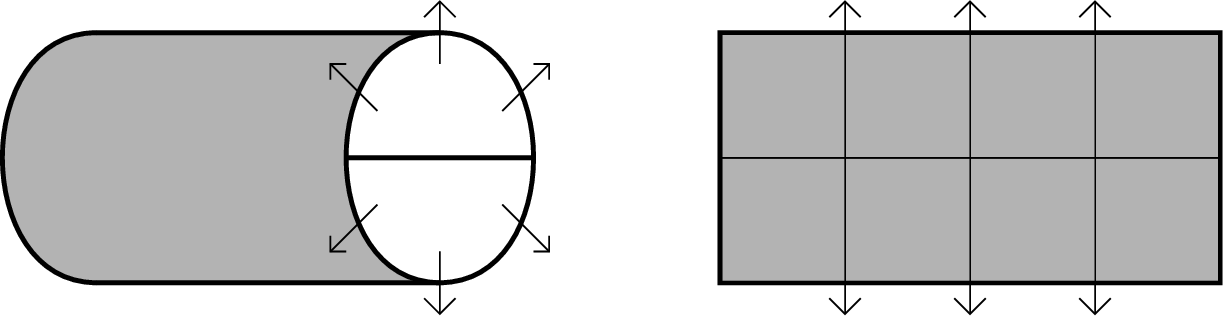}
\put(50, 0){\vector(1, 0){10}}
\put(62, -.5){$\theta$}
\put(50, 0){\vector(0, 1){10}}
\put(47, 12){$\mathcal{N}(\Sigma)$}
\put(35, 15){$\Sigma$}
\end{overpic}
\caption{The handle $H_{\Sigma}$ and the vector field $V_{\beta}$. On the left we see $\Sigma$ sitting inside of the boundary of the handle. On the right is a simplified schematic picture which we'll adopt for the remainder of the proof.}
\label{Fig:flat_handle}
\end{figure}

It follows that $V_{\beta}$ is a symplectic dilation of $(H_{\Sigma},\omega_{\beta})$.  This vector field points transversely out of $\partial H_{\Sigma}$ along $[-1,1]\times\partial \mathcal{N}(\Sigma)$ and is tangent to $\{\pm 1\}\times \mathcal{N}(\Sigma)$.  Therefore, $-z d\theta+\beta$ is a $\theta$-invariant contact 1-form on $[-1,1]\times\partial \mathcal{N}(\Sigma)$ inducing the $\theta$-invariant contact structure on $[-1,1]\times\partial \mathcal{N}(\Sigma)$ described in Proposition \ref{Prop:Normalize}. The handle $H_{\Sigma}$ along with the vector field $V_{\beta}$ is depicted in Figure \ref{Fig:flat_handle}.\\

\noindent\textbf{Step 3: Concavity of the negative end of the handle.} In order to be able to attach $\{\pm 1\}\times \mathcal{N}(\Sigma)$ to the ``top'' $\{ 0 \}\times M$ of the finite symplectization of $\Mxi$ in a way so that $e^{t}\alpha$ on $[-\epsilon, 0]\times M$ extends over $H_{\Sigma}$, we must modify $V_{\beta}$ so that the $\{\pm 1 \}\times \mathcal{N}(\Sigma)$ is concave in the sense of Definition \ref{Def:Fillable}.

\begin{figure}[h]
\begin{overpic}[scale=.7]{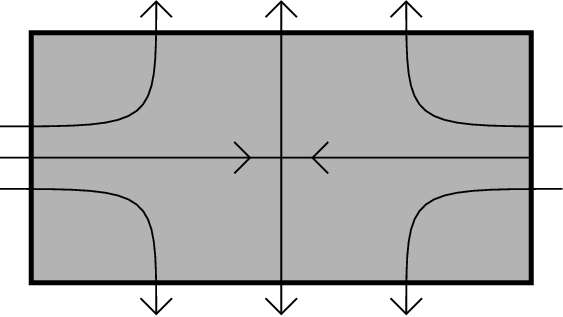}
\end{overpic}
\caption{The handle $H_{\Sigma}$ and the vector field $Z$.}
\label{Fig:flat_handle_perturbed_vf}
\end{figure}

To achieve concavity, write $X_{z\theta}= \theta \partial_{\theta} - z\partial_{z}$ for the Hamiltonian vector field of the function $z\theta$ on $H_{\Sigma}$ with respect to the symplectic form $\omega_{\beta}$. Then 
\begin{equation*}
Z =V_{\beta} - X_{z\theta} = - \theta\partial_{\theta} + 2z\partial_{z} + X_{\beta}
\end{equation*}
is a symplectic dilation of $\omega_{\beta}$.  Moreover, $Z$ points into $H_{\Sigma}$ along $\{\pm 1\}\times \mathcal{N}(\Sigma)$ and out of $H_{\Sigma}$ along $[-1,1]\times\partial \mathcal{N}(\Sigma)$. See Figure \ref{Fig:flat_handle_perturbed_vf}. 
The time $T$ flow of $Z$ is easily computed  as
\begin{equation}\label{Eq:ZFlow}
\Flow^{T}_{Z}(\theta, z, x) = (e^{-T}\theta, e^{2T}z, \Flow^{T}_{X_{\beta}}(x))
\end{equation}
where $x \in \Sigma$. Define
\begin{equation}\label{Eq:Lambda}
\lambda=\omega_{\beta}(Z,\ast)=-\theta dz-2zd\theta+\beta.
\end{equation}
Then $d\lambda=\omega_{\beta}$ so that $Z$ is the Liouville vector field of $\lambda$ and $\lambda|_{\{\pm 1\}\times \mathcal{N}(\Sigma)}=\mp dz+\beta$.\\

\noindent\textbf{Step 4: Flattening the negative ends of the handle.} Another modification of the handle $H_{\Sigma}$ is required so that after the negative end of our handle is attached to the top of a finite symplectization, the convex boundary of the space obtained will be smooth. Here we modify $H_{\Sigma}$ to that the concave and convex boundaries share an overlap. The result will be our finished handle $\mathcal{H}_{\Sigma}$.

\begin{figure}[h]
\begin{overpic}[scale=.8]{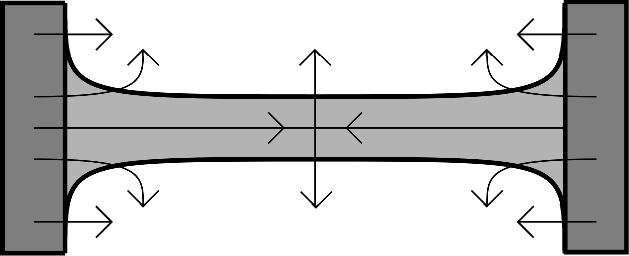}
\end{overpic}
\caption{The handle $\mathcal{H}_{\Sigma}$ and the vector field $Z$. The collar regions $(-\frac{\epsilon}{2}, 0]\times \mathcal{N}(\Sigma)$ are shaded dark gray.}
\label{Fig:smooth_handle}
\end{figure}

Noting that $Z$ points into $H_{\Sigma}$ along the set $\{ \pm 1\} \times \mathcal{N}(\Sigma) \subset \partial H_{\Sigma}$, we extend $H_{\Sigma}$ by collar regions $(-\frac{\epsilon}{2}, 0] \times \mathcal{N}(\Sigma)$ whose $(-\frac{\epsilon}{2}, 0]$-factor parameterizes the flow of the vector field $Z$. One such collar region is appended to each end $\{ \pm 1\} \times \mathcal{N}(\Sigma)$ of the concave boundary of $H_{\Sigma}$, resulting in an extended handle $H_{\Sigma}'$.

We now modify the subset $[-1, 1]\times \partial \mathcal{N}(\Sigma) \subset \partial H_{\Sigma} \cap \partial H_{\Sigma}'$ along which $Z$ points out of $H_{\Sigma}'$. Identify a collar neighborhood of $[-1, 1] \times \partial \mathcal{N}(\Sigma)$ in $H_{\Sigma}'$ with $(-\epsilon, 0] \times [-1, 1] \times \partial \mathcal{N}(\Sigma)$, so that the $(-\epsilon, 0]$-factor parameterized by a variable $\sigma$ for which
\begin{equation*}
\partial_{\sigma} = Z,\quad \{ 0 \} \times [-1, 1] \times \partial \mathcal{N}(\Sigma) = [-1, 1] \times \partial \mathcal{N}(\Sigma).
\end{equation*}
Let $h = h(\theta): [-1, 1] \rightarrow (-\epsilon, 0]$ be a function satisfying the following conditions:
\be
\item There is a small, positive constant $\delta < \epsilon$ such that $h(\theta) = \frac{\epsilon}{2}$ for $\theta \in [-1 + 2\delta, 1 - 2\delta]$.
\item Along $[-1, -1 + \delta]$, $h(\theta) = \log(-\theta)$ and along $[1 - \delta, 1]$, $h(\theta) = \log(\theta)$.
\ee
We define the handle $\mathcal{H}_{\Sigma}$ to be the complement of the set
\begin{equation*}
\{ \sigma > h(\theta) \} \subset (-\epsilon, 0] \times [-1, 1] \times \partial \mathcal{N}(\Sigma)
\end{equation*}
in $H_{\Sigma}'$. Clearly $Z$ points out of $\partial \mathcal{H}_{\Sigma}$ along the newly modified subset $\{ \sigma = h(\theta) \}$ of $\partial \mathcal{H}_{\Sigma}$.

To complete this step of the construction, we must establish that a neighborhood of the boundary of $\{ \sigma = h(\theta) \}$ is contained in $\{ \pm 1 \} \times \mathcal{N}(\Sigma)$. View the set $\{ \sigma = h(\theta) \}$ as the image of $[-1, 1]\times \partial \mathcal{N}(\Sigma)$ in $[-1, 1] \times \mathcal{N}(\Sigma)$ under the mapping
\begin{equation*}
\Xi(\theta, z, x) = \Flow^{h(\theta)}(\theta, z, x) = (e^{-h(\theta)}\theta, e^{2h(\theta)}z, \Flow^{h(\sigma)}_{X_{\beta}}(x)).
\end{equation*}
The second equality above follows from Equation \eqref{Eq:ZFlow}. By our restrictions on the function $h$, near $\theta = \pm 1$ we have the more explicit formula
\begin{equation*}
\Xi(\theta, z, x) = (e^{-\log(\pm \theta)}\theta, e^{2\log(\pm \theta)}z, \Flow^{\log(\pm\theta)}_{X_{\beta}}(x)) = (\pm 1, \theta^{2}z, \Flow^{\log(\pm\theta)}_{X_{\beta}}(x)),
\end{equation*}
establishing the desired result. The handle $\mathcal{H}_{\Sigma}$ is shown in Figure \ref{Fig:smooth_handle}.\\

\noindent\textbf{Step 5: Attaching the handle.} Consider a finite symplectization $([-\epsilon, 0] \times M, e^{t}\alpha)$ of the contact manifold $M$. According to our setup, this space contains subsets of the form 
\begin{equation*}
[-\epsilon, 0]\times \mathcal{N}(i_{j}(\Sigma)),\quad e^{t}\alpha = e^{t}(dz + \beta)
\end{equation*}
which are finite symplectizations of the neighborhoods of our Liouville hypersurfaces.

We may identify each connected component of the collar of $\mathcal{H}_{\Sigma}$ with the $(-\frac{\epsilon}{2}, 0]\times \mathcal{N}(i_{j}(\Sigma))$, determining an attachment of the handle to $[-\epsilon, 0] \times M$. By the construction of $\lambda$ on $\mathcal{H}_{\Sigma}$, the Liouville forms $\lambda$ and $e^{t}\alpha$ agree according to this identification.

\begin{figure}[h]
	\begin{overpic}[scale=.7]{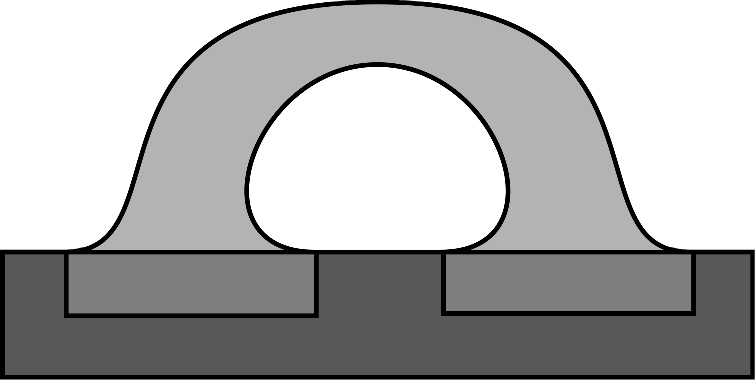}
		\put(14, 40){$\mathcal{H}_{\Sigma}$}
		\put(-25, 5){$[-\epsilon, 0]\times M$}
		\put(44, 22){\vector(-1, -1){10}}
		\put(56, 22){\vector(1, -1){10}}
		\put(36, 24){$(-\frac{\epsilon}{2}, 0]\times \mathcal{N}(\Sigma)$}
	\end{overpic}
	\caption{Attaching the smoothed handle $\mathcal{H}_{\Sigma}$ to a finite symplectization of $(M, \alpha)$.}
	\label{Fig:smooth_handle_attachment}
\end{figure}

By the fact that the convex boundary $\{ \sigma = h(\theta) \} \subset \partial \mathcal{H}_{\Sigma}$ is tangent to $\{ \pm 1\} \times \mathcal{N}(\Sigma)$ near its boundary, the convex boundary of the manifold obtained by the handle attachment is smooth. The handle attachment is depicted in Figure \ref{Fig:smooth_handle_attachment}.\\

\noindent\textbf{Step 6: Identifying the boundary as the connect sum.} To complete our construction, we must show that the symplectic cobordism obtained by the gluing construction described above recovers the Liouville connect sum along its convex boundary.

We study the contact form $\lambda$ restricted to the convex boundary of the flat handle $H_{\Sigma}$, obtaining a change of coordinates between 
\be
\item a collar neighborhood of the boundary $\partial \mathcal{N}(i_{j}(\Sigma))$ of the surgery locus $\mathcal{N}(i_{j}(\Sigma))$ in $M$ and
\item $[\pm 1, \pm 1 \mp \delta) \times \partial \mathcal{N}(\Sigma) \subset \partial H_{\Sigma}$
\ee
by following flow lines of the vector field $Z$.

Following Section \ref{Sec:ThetaInvariant}, we write 
\begin{equation*}
[-1, 1]\times \partial \mathcal{N}(\Sigma) = \bigg( [-1, 1]\times \{ \pm \epsilon \}\times \widehat{\Sigma} \bigg) \cup \bigg( [-1, 1]\times \gamma \times \partial\Sigma \bigg).
\end{equation*}
Using Equation \eqref{Eq:Lambda} and our description of the curve $\gamma$ in Section \ref{Sec:Corners},
\begin{equation*}
\lambda = \begin{cases}
-2\epsilon d\theta + \beta & \text{along}\quad [-1, 1]\times \{ \pm \epsilon \}\times \widehat{\Sigma} \\
-2z(s)d\theta - \theta\frac{\partial z}{\partial s}ds + t(s)\alpha' & \text{along}\quad [-1, 1]\times \gamma \times \partial\Sigma.
\end{cases}
\end{equation*}
After applying a transformation which divides $\theta$ by $2$, we obtain a case of the perturbed contact form $\alpha_{G}$ described in Equation \eqref{Eq:GPerturbedAlpha} used to describe the Liouville connect sum in Section \ref{Sec:ThetaInvariant}.

This completes the construction of the symplectic handle attachment described in Theorem \ref{Thm:Cobordism}.

\subsection{Weinstein handle decomposition of the connect-sum cobordism}\label{Sec:WHandleDecomposition}

In this subsection, we show that when $\Sdom$ is Weinstein, then so is the symplectic cobordism constructed by the attachment of the handle $\mathcal{H}_{\Sigma}$ described above.

Again, let $\Mxi$ be a $(2n+1)$-dimensional contact manifold with contact form $\alpha$.  In this section, we continue use of the notation appearing in the previous section.  Suppose that the Liouville domain $\Sdom$ is Weinstein.  Then there is a decomposition
\begin{equation}\label{Eq:Filter}
\sqcup(\disk^{2n},\lambda_{std})=(\Sigma_{0},\beta_{0})\subset\cdots\subset (\Sigma_{n},\beta_{n})=\Sdom
\end{equation}
of $\Sdom$ as described in Theorem \ref{Thm:WHandle}(3).  Consider the symplectic cobordism $\Ldom$ from $\Mxi$ to $\#_{\Sdom}\Mxi$ described in the previous section.  We will use the filtration described in Equation \eqref{Eq:Filter} to filter $\Ldom$ as
\begin{equation}\label{Eq:CobordismFilter}
([0, -\epsilon]\times M, e^{t}\alpha)=(W_{-1},\lambda_{-1})\subset(W_{0},\lambda_{0})\subset\cdots\subset (W_{n},\lambda_{n})=\Ldom
\end{equation}
where each $(W_{k},\lambda_{k})$ is obtained from $(W_{k-1},\lambda_{k-1})$ by attaching some number of $(2n+2)$-dimension Weinstein $(k+1)$-handles.

Consider the restrictions of the embeddings $i_{1}$ and $i_{2}$ to $(\Sigma_{j},\beta_{j})$.  Let $N_{j,1}$ and $N_{j,2}$ denote neighborhoods of $i_{1}(\Sigma_{j})$ and $i_{2}(\Sigma_{j})$ given by
\begin{equation*}
[-\epsilon,\epsilon]\times \Sigma_{k}\subset[-\epsilon,\epsilon]\times\Sigma\quad\text{with}\quad \alpha =dz+\beta_{k}.
\end{equation*}
Define $(W_{k},\lambda_{k})$ to be the cobordism associated to the Liouville connect sum of $\Mxi$ along $N_{k,1}$ and $N_{k,2}$.  The filtration of Equation \eqref{Eq:Filter} induces a filtration of handles 
\begin{equation*}
(\mathcal{H}_{\Sigma_{k}},\lambda_{k}|_{\mathcal{H}_{\Sigma_{k}}}=-\theta dz-2zd \theta+\beta_{k}).
\end{equation*}
Therefore we have
\begin{equation}\label{Eq:HandleFilter}
(\mathcal{H}_{\Sigma_{0}},\lambda_{0})\subset\cdots\subset(\mathcal{H}_{\Sigma_{n}},\lambda_{n}).
\end{equation}
As each of the $(W_{k},\lambda_{k})$ is given by the attachment of $(\mathcal{H}_{\Sigma_{k}},\lambda_{k})$ to the finite symplectization of $\Mxi$, then Equation \eqref{Eq:HandleFilter} induces the filtration of Equation \eqref{Eq:CobordismFilter}.

Note that $(W_{0},\lambda_{0})$ is obtained from $([-\epsilon, 0]\times\Mxi,e^{t}\alpha)$ by attachment of some number of $(2n+2)$-dimensional Weinstein $1$-handles.  This may be seen by comparing the first part of the proof of Theorem \ref{Thm:Cobordism} and the exposition in Section \ref{Sec:WHandle}.

Now consider the cobordism $(W_{k-1},\lambda_{k-1})$.  Let $\Lambda_{j}$, $j=1,\dots,m$ be the core $k$-disks of the $2n$-dimensional Weinstein $k$ handles that are attached to the boundary of $(\Sigma_{k-1},\beta_{k-1})$ to obtain $(\Sigma_{k},\beta_{k})$.  Consider, for each $j$, two copies $\Lambda_{j,1}$ and $\Lambda_{j,2}$ of $\Lambda_{j}$ living in the boundary of $(W_{k-1},\lambda_{k-1})$ as
\begin{equation*}
\begin{aligned}
\Lambda_{j,i} &= \{ z=0\}\times\Lambda_{j}\\
& \subset  \big{(} N_{k,i}\setminus \Int(N_{k-1,i}) \big{)} \\
& \subset  \big{(}\{ t=1\}\times (M\setminus \Int(N_{k-1,1}\cup N_{k-1,2}))\big{)}\subset \partial W_{k-1}
\end{aligned}
\end{equation*}
for $i=1,2$.  Define the $(k+1)$-dimensional disks
\begin{equation*}
\widetilde{\Lambda}_{j}=[-1,1]\times \Lambda_{j}\subset \left( H_{\Sigma_{k}}\setminus\Int(H_{\Sigma_{k-1}}) \right)\subset\left( W_{k}\setminus\Int(W_{k-1})\right).
\end{equation*}
The 1-form $\beta$ vanishes on $\Lambda=\{0\}\times\Lambda\subset[-\epsilon,\epsilon]\times \Sigma_{k}$ by the explicit description of the Weinstein handle given in Equation \eqref{Eq:WForm}.  Consequently, the Liouville 1-form $\lambda=-\theta dz-2zd \theta+\beta$ vanishes on $\widetilde{\Lambda}_{j}$.  The boundary of $\widetilde{\Lambda_{j}}$ is the piecewise smooth $k$-sphere
\begin{equation*}
\partial \widetilde{\Lambda}_{j}=\Lambda_{j,1}\cup([-1,1]\times\partial\Lambda_{j})\cup\Lambda_{j,2}.
\end{equation*}

After smoothing the corners of $W_{k-1}$, $\widetilde{\Lambda}_{j}\cap W_{k-1}$ will be a smooth isotropic sphere $S^{k}_{j}$, by the vanishing of the form $\lambda$ along $\widetilde{\Lambda}$.  Noting that $(W_{k},\lambda_{k})$ is obtained from $(W_{k-1},\lambda_{k-1})$ by attaching Weinstein handles along each of the $S^{k}_{j}$, the proof of Theorem \ref{Thm:Cobordism} is complete.

\subsection{Attaching handles to weak symplectic cobordisms and fillings}\label{Sec:WeakHandle}

In this section we describe when it is possible to attach a modified version of the symplectic handle $(\mathcal{H}_{\Sigma},\omega_{\Sigma})$ to the positive boundary of a weak symplectic cobordism $(W,\omega)$.  The main result of this section is the following:

\begin{thm}\label{Thm:WeakHandle}
Let $\Mxi$ be a $(2n+1)$-dimensional contact manifold, which is the convex boundary of a weak symplectic cobordism $(W,\omega)$ with concave boundary $(M',\xi')$.  Let $\Sdom$ be a 2n-dimensional Liouville domain and let $i_{j}:\Sdom\rightarrow \Mxi$, $j=1,2$, be Liouville embeddings with disjoint images.  Assume that \begin{equation*}
[i_{1}^{*}\omega]=[i_{2}^{*}\omega] \in H^{2}(\Sigma, \mathbb{R}).
\end{equation*}
Then there is a weak symplectic cobordism whose negative boundary is $(M',\xi')$ and whose positive boundary is the manifold $\#_{\Sdom}\Mxi$ obtained from $\Mxi$ by a Liouville connect sum.
\end{thm}

To prove Theorem \ref{Thm:WeakHandle}, we need the following lemma, summarizing some results of \cite{MNW12}:

\begin{lemma}\label{Lemma:WeakDeformation}
Suppose that $\Mxi$ is the positive boundary of weak symplectic cobordism $(W,\omega)$ and $\alpha$ is a contact 1-form for $\Mxi$.  Then $W$ can be extended to a non-compact symplectic manifold $(W',\omega')$ for which
\be
\item $W'\setminus W$ is diffeomorphic to $(0,\infty)\times M$,
\item the symplectic form $\omega'$ coincides with $\omega$ on $W$,
\item $\omega'|_{(t_{0},\infty)\times M)}= \omega|_{TM}+d(e^{t}\alpha)$ for a sufficiently large constant $t_{0}>0$, where $t$ is a coordinate on $(0,\infty)$, and
\item each of the level sets $(\{ t\}\times M,\xi)$ is weakly filled for $t>t_{0}$.
\ee
Furthermore, if $\omega''$ is a $2$-form on $M$ with $[\omega''] = [\omega|_{TM}] \in H^{2}(M)$ then $\omega'$ may be chosen so that is coincides with $\omega'' + d(e^{t}\alpha)$ on $(t_{0}, \infty)\times M$ for $t_{0} > 0$ sufficiently large.
\end{lemma}

\begin{proof}[Proof of Theorem \ref{Thm:WeakHandle}]
We continue to make use of the notation described in Section \ref{Sec:Handle}, modifying the construction described there as needed.

Suppose that $\alpha$ is a contact form for $\Mxi$ such that $i_{j}^{*}\alpha=\beta$ for $j=1,2$.  Consider the embeddings 
\begin{equation*}
i_{-}|_{\{\theta=-1\}},i_{+}|_{\{\theta=1\}}:\mathcal{N}(\Sigma)\rightarrow M
\end{equation*}
defined in the handle attachment of Section \ref{Sec:Handle}.  By our hypothesis, 
\begin{equation*}
[i_{-}|_{\{\theta=-1\}}^{*}\omega] = [i_{+}|_{\{\theta=1\}}^{\ast}\omega] \in H^{2}(\mathcal{N}(\Sigma)).
\end{equation*}
By deforming $\omega$ within its class $[\omega]\in H^{2}$ and applying Lemma \ref{Lemma:WeakDeformation} we may choose a collar $((0, \infty)\times M, \omega')$ over the convex boundary of $(W, \omega)$ for which
\begin{equation}\label{Eq:WeakCollar}
\begin{gathered}
\omega' = \omega'' + d(e^{t}\alpha),\\
i_{-}|_{\{\theta=-1\}}^{*}\omega'' = i_{+}|_{\{\theta=1\}}^{\ast}\omega'' \in \Omega^{2}(\mathcal{N}(\Sigma)).
\end{gathered}
\end{equation}
along some $(t_{0}, \infty)\times M$. 

Let $(W_{0}, \omega')$ be the compact symplectic manifold obtained by removing the open collar $(2t_{0},\infty)\times M$ from $W'$. Attach the handle $\mathcal{H}_{\Sigma}$ to $\partial W_{0}$ as described in the symplectic handle attachment of Section \ref{Sec:Handle}. By construction, $d(t\alpha)$ extends over the handle. To complete our proof, we must show that $\omega''$ -- and hence $\omega'$ -- can be smoothly extended over the handle so that the extension determines a weak filling of the contact manifold determined by the union of $(W_{0}, \omega')$ with the handle.

We can extend the $2$-form $\omega''$ over the handle $\mathcal{H}_{\Sigma}$ as a $\theta$-invariant form determined by its value along the gluing locus. The two-form $\omega' = \omega'' + d(e^{2t_{0} +t}\lambda)$ with $t > 0$ obtained is conformally equivalent to 
\begin{equation*}
e^{-2t_{0}}\omega'' + d(e^{t}\alpha)
\end{equation*}
and so is guaranteed to be symplectic on the contact hyperplanes of the convex boundary of $\mathcal{H}_{\Sigma}$ for $t_{0}$ sufficiently large by the compactness of $\mathcal{H}_{\Sigma}$.
\end{proof}

\section{A neighborhood theorem for Liouville submanifolds of high codimension}\label{Sec:HighCodimension}

The purpose of this section is to show that every Liouville submanifold $\Sdom$ of a contact manifold $\Mxi$ whose codimension is greater than one embeds into a Liouville hypersurface $(\widehat{\Sigma},\widehat{\beta})\subset\Mxi$ which smoothly retracts onto $\Sigma$.

The proof consists of two parts.  The first part of this proof, established in Theorem \ref{Thm:LiouvilleDiskBundle}, is an existence result asserting that the total space of every symplectic disk bundle over a Liouville domain admits the structure of a Liouville domain in a natural way.  It is easy to construct an exact symplectic form on the total space of such a bundle using, for example, Thurston's technique \cite[Theorem 6.3]{MS:SymplecticIntro}.  That being said, the content of Theorem \ref{Thm:LiouvilleDiskBundle} is that the Liouville vector field can be made transverse to the boundary of the total space of this disk bundle and so our construction will rely on a different technique.  The second part follows the standard Gray-Moser-Weinstein argument used to establish neighborhood theorems in contact and symplectic geometry, showing that we can isotop a submanifold $\widehat{\Sigma}$ of $\Mxi$ containing $\Sigma$ so that $\alpha|_{T\widehat{\Sigma}}$ coincides with the model Liouville 1-form provided by Theorem \ref{Thm:LiouvilleDiskBundle} where $\alpha$ is a fixed contact form for $\Mxi$.

\subsection{Existence results for $1$-forms on disk bundles}

\begin{lemma}\label{Lemma:Unitary}
Let $\pi:E\rightarrow M$ be a rank $2m$ vector bundle over a compact manifold $M$ equipped with a smooth section $\omega$ of $E^{*}\wedge E^{*}\rightarrow X$ which is symplectic on each fiber of $E$.  Then there is a 1-form $\lambda\in\Omega^{1}(E)$ on the total space $E$ such that
\be
\item on each fiber $E_{x}$, $x\in M$, there is a linear coordinate system $p_{j},q_{j}, j=1,\dots, m$ for which \begin{equation*}
    \lambda|_{TE_{x}}=\frac{1}{2}\sum_{1}^{m} (p_{j}dq_{j}-q_{j}dp_{j}),
    \end{equation*}
\item $d\lambda(\partial_{t}(v_{x}+tw_{x}),\partial_{t}(v_{x}+tw'_{x}))=\omega(w_{x},w'_{x})$ for all $v_{x},w_{x},w'_{x}\in E_{x}$, $x\in M$,
\item $\lambda$ and $d\lambda$ are both annihilated by vectors tangent to the zero section of $E$, and
\item $\lambda=\frac{1}{2}d\lambda(R_{E},\ast)$ where $R_{E}$ is the radial vector field on $E$ when restricted to the tangent spaces of fibers of $E$.
\ee
\end{lemma}

\begin{defn}
A vector bundle $E$ equipped with a fiber-wise bilinear form $\omega$ as described in the statement of the above lemma will be referred to as a \emph{symplectic vector bundle} and will be denoted by the pair $(E,\omega)$.
\end{defn}

\begin{proof}[Proof of Lemma \ref{Lemma:Unitary}]
Fix a complex structure $J$ and bundle metric $\langle\ast,\ast\rangle$ on $E$ so that $\omega(\ast,J\ast)=\langle \ast,\ast\rangle$.  Such a complex structure and bundle metric always exist as can be seen in \cite[\S 2.6]{MS:SymplecticIntro}. Let $\{ U_i \}$ be a finite covering of $M$ where each $U_x$ consists of a ball centered about a point $x_i \in M$. We write $\{ h_i \}$ for a collection of functions on $M$ determining a partition of unity for the cover with each $h_i$ supported in $U_i$.

Pick a unitary trivialization $p_j, q_j$ of $E$ within each $U_i$, providing an identification
\begin{equation*}
\begin{gathered}
\phi_i: E|_{U_i} \rightarrow U_i \times \mathbb{R}^{2m}\\
\omega(p_j, q_k) = \delta_{j,i},\quad \omega(p_i, p_j) = \omega(q_i, q_j) = 0.
\end{gathered}
\end{equation*}
We define $\lambda = \sum_{i} \lambda_{i}$  where the $\lambda_i$ is defined over each $E|_{U_i}$ as
\begin{equation*}
    \lambda_i = \frac{h_i}{2} \sum_{j=1}^{m}(p_{j}dq_{j} - q_{j}dp_{j}) = \frac{h_i}{2}(p dq - q dp)
\end{equation*}
in short-hand notation. Then
\begin{equation*}
    d\lambda_i = h_{i} \big( \sum_{j=1}^{m}dp_{j}\wedge dq_{j} \big) + \frac{dh_{i}}{2}\wedge \big(\sum_{j=1}^{m}(p_{j}dq_{j} - q_{j}dp_{j})\big).
\end{equation*}
Let $i_1 \neq i_2$ be distinct indices of our covering. Then the unitary-matrix-valued transition function $A = A_{j, k}(x)$ for the local trivializations provided above allow us to write
\begin{equation*}
    \lambda_{i_2} = \frac{h_{i_2}}{2} ((A p) d(A q) - (A q) d(A p)) = \frac{h_{i_2}}{2} (p dq - q dp + ((A p)q - (A q) p) dA).
\end{equation*}
From this we see that within each $U_i$ we can write
\begin{equation}\label{Eqn:LambdaGlobal}
    \lambda = \frac{1}{2}(p dq - q dp) + f_i d\theta_i
\end{equation}
where $\theta_i$ is a $1$-form whose kernel contains the $\partial_{p_{j}}, \partial_{q_{j}}$ and $f_i$ is a function which vanishes up to second order on the $p_j, q_j$ at $p_j = q_j = 0$.
\end{proof}

\begin{thm}\label{Thm:LiouvilleDiskBundle}
A sufficiently small neighborhood $\disk$ of the zero section of a rank $2m$ symplectic vector bundle $\pi:(E,\omega)\rightarrow\Sigma$ over a Liouville domain $\Sdom$ carries the structure of a Liouville domain $(\widehat{\Sigma},\widehat{\beta})$, where $\widehat{\Sigma}$ is obtained by rounding the corners $\partial(\pi^{-1}(\partial\Sigma))$ of $\disk$.  The 1-form $\widehat{\beta}$ is such that
\be
\item on each fiber $\disk_{x}$, $x\in \Sigma$, there is a coordinate system $p_{j},q_{j}, j=1,\dots,m$ for which
    \begin{equation*}
    \widehat{\beta}|_{T\disk_{x}}=\frac{1}{2}\sum_{1}^{m}( p_{j}dq_{j}-q_{j}dp_{j}),
    \end{equation*}
\item $d\widehat{\beta}(\partial_{t}(v_{x}+tw_{x}),\partial_{t}(v_{x}+tw'_{x}))=\omega(w_{x},w'_{x})$ for all $v_{x},w_{x},w'_{x}\in \disk_{x}$, $x\in \Sigma$, and
\item $\widehat{\beta}$ coincides with $\beta$ when restricted to the tangent space of the zero-section of $\disk$.
\ee
\end{thm}

The proof of this theorem is a continuation of the construction described in the proof of Lemma \ref{Lemma:Unitary} and makes use of the unitary structure $(\omega,J,\langle \ast, \ast\rangle)$ and the system of local trivializations over the $U_i$ on the vector bundle $E$ described there.

\begin{proof}
We define $\widehat{\beta}=\lambda+\pi^{*}\beta$ where $\lambda$ is as described in Lemma \ref{Lemma:Unitary}.  By the properties of $\lambda$ listed in the statement of Lemma \ref{Lemma:Unitary} it follows that on a sufficiently small neighborhood 
\begin{equation*}
\disk_{\epsilon}:=\{v\in E\ :\ \langle v,v\rangle\leq\epsilon\}
\end{equation*}
of the zero-section of $E$ the 2-form $d\widehat{\beta}$ is symplectic.  As $\pi^{*}\beta$ and $d\pi^{*}\beta=\pi^{*}d\beta$ are annihilated by tangent vectors in the vertical subspaces of $TE$, $\widehat{\beta}$ satisfies the properties listed in the statement of the theorem.  Therefore, all that remains to be shown is that the vector field $X_{\widehat{\beta}}$ determined by the equation $d\widehat{\beta}(X_{\widehat{\beta}},\ast)=\widehat{\beta}$ is positively transverse to the boundary of $\disk_{\epsilon}$ for a sufficiently small constant $\epsilon>0$.

Let $X_{\beta}$ be the Liouville vector field of $\Sdom$. Over each coordinate patch $U_i$ in $\Sigma$, we can use the trivialization $\phi$ to write
\begin{equation*}
    X_{\widehat{\beta}} = \frac{R_{E}}{2} + X_{\beta} + Z_i
\end{equation*}
where $R_{E}$ is the radial vector field for the bundle $R_{E} = p\partial_{p} + q\partial_{q}$. The vector field $Z_i$ is our error term which will be analyzed to complete the proof. Following Equation \eqref{Eqn:LambdaGlobal}, we have
\begin{equation*}
    d\widehat{\beta} = d\beta + dp \wedge dq + df_{i}\wedge d\theta_{i}.
\end{equation*}
Using the above calculation and Equation \ref{Eqn:LambdaGlobal}, we see that $Z_i$ is the solution to the equation
\begin{equation*}
\begin{aligned}
d\widehat{\beta}(Z_i, \ast) &= \widehat{\beta} - d\widehat{\beta}(\frac{R_{E}}{2} + X_{\beta}, \ast)\\
&= (f_i - df_{i}(X_{\beta}) - \frac{df_i(R_{E})}{2})d\theta_{i} + d\theta_i(X_{\beta})df_i.
\end{aligned}
\end{equation*}

From this calculation, the fact that the $f_i$ vanishes up to second order along the zero-section of $E$, and compactness of $\Sigma$ we conclude that for some $\epsilon>0$ the vector field $X_{\widehat{\beta}}$ is positively transverse to the boundary of $\disk_{\epsilon}$ along the boundary of $\pi^{-1}(\Int(\Sigma))$.

To finish our proof we will show that for $\epsilon$ sufficiently small, $X_{\widehat{\beta}}$ is transverse to the boundary of $\disk_{\epsilon}$ along $\pi^{-1}(\partial\Sigma)$.  To see this, observe that for any $\epsilon>0$
\be
\item $\frac{1}{2}R_{E}$ is tangent to the boundary of $\disk_{\epsilon}$ along $\pi^{-1}(\partial\Sigma)$,
\item the vector field $\widetilde{X_{\beta}}$ is positively transverse to the boundary of $\disk_{\epsilon}$ along $\pi^{-1}(\partial\Sigma)$, and
\item the vector field $Z$ vanishes along the zero-section of $\disk_{\epsilon}$.
\ee

We conclude that for $\epsilon$ sufficiently small, $X_{\widehat{\beta}}$ is transverse to $\partial\disk_{\epsilon}$ along $\pi^{-1}(\partial\Sigma)$.  After having fixed such an $\epsilon$, we can round the corners of $\disk_{\epsilon}$ to obtain a manifold with smooth boundary $\widehat{\Sigma}\subset\disk_{\epsilon}$ for which $X_{\widehat{\beta}}$ is positively transverse to $\partial \widehat{\Sigma}$ so that $\widehat{\beta}$ is a Liouville 1-form on $\widehat{\Sigma}$.
\end{proof}

The same construction can be easily applied to symplectic vector bundles over contact manifolds.

\begin{thm}\label{Thm:ContactDiskBundle}
Let $\Mxi$ be a compact contact manifold with contact form $\alpha$.  Then a sufficiently small neighborhood $\disk$ of the zero section of a rank $2m$ symplectic vector bundle $\pi:(E,\omega)\rightarrow M$ naturally carries the structure of a contact manifold $(\disk,\xi_{E})$.  The contact structure $\xi_{E}$ can be described as $\ker(\widehat{\alpha})$ for a 1-form $\widehat{\alpha}$ such that
\be
\item on each fiber $\disk_{x}$, $x\in M$, there is a coordinate system $p_{j},q_{j}, j=1,\dots,m$ for which
    \begin{equation*}
    \widehat{\alpha}|_{T\disk_{x}}=\frac{1}{2}\sum_{1}^{m} (p_{j}dq_{j}-q_{j}dp_{j}),
    \end{equation*}
\item $d\widehat{\alpha}(\partial_{t}(v_{x}+tw_{x}),\partial_{t}(v_{x}+tw'_{x}))=\omega(w_{x},w'_{x})$ for all $v_{x},w_{x},w'_{x}\in \disk_{x}$, $x\in M$, and
\item $\widehat{\alpha}$ coincides with $\alpha$ when restricted to the tangent space of the zero-section of $\disk$.
\ee
\end{thm}

To prove this theorem, we may define $\widehat{\alpha}=\lambda+\pi^{*}\alpha$.  The 1-form $\widehat{\alpha}$ is contact on a sufficiently small tubular neighborhood of the zero-section of $E$ by Equation \eqref{Eqn:LambdaGlobal}.  It is well known that the contact structure on a tubular neighborhood of a contact submanifold is uniquely determined by its symplectic normal bundle.  However, the author is unsure as to whether or not a proof of the existence of a contact structure on the total space of a symplectic disk bundle over a given contact manifold has been written anywhere.

\subsection{Liouville submanifolds of codimension $>1$}\label{Sec:HighCodimensionNeighborhoods}

\begin{thm} \label{Thm:HighCodimension}
Let $\Mxi$ be a $(2n+1)$-dimensional contact manifold with contact form $\alpha$ and suppose that $\Sigma$ is a compact $2k$-dimensional submanifold of the interior of $M$ such that $k<n$ and $\alpha|_{T\Sigma}$ is a Liouville 1-form on $\Sigma$.  Then $\Sigma$ admits a neighborhood of the form $[-\epsilon,\epsilon]\times\widehat{\Sigma}$ on which $\alpha=dz+\widehat{\beta}$ for a Liouville 1-form $\widehat{\beta}$ on $\widehat{\Sigma}$.  The manifold $\widehat{\Sigma}$ is obtained by rounding the corners $\pi^{-1}(\partial\Sigma)$ of a disk bundle $\pi:\disk\rightarrow\Sigma$.  Moreover, there is a coordinate system $(p_{j},q_{j})$ on each fiber of the disk bundle on which
\begin{equation*}
\alpha|_{\pi^{-1}(x)}=\frac{1}{2}\sum_{1}^{n-k} (p_{j}dq_{j}-q_{j}dp_{j}).
\end{equation*}
\end{thm}

\begin{proof}
By Lemma \ref{Lemma:ReebTransversality}, the Reeb vector field $R_{\alpha}$ for $\alpha$ is nowhere tangent to $\Sigma$. Thus we may decompose the vector bundle $\xi|_{\Sigma}$ into a direct sum $\xi|_{\Sigma}=\xi_{\Sigma}\oplus\xi_{\Sigma}^{\bot}$ where
\begin{equation*}
\begin{aligned}
\xi_{\Sigma}&:=\{v-\alpha(v) R_{\alpha}:v\in T\Sigma\}\ \text{and} \\
\xi_{\Sigma}^{\bot}&:=\{v\in\xi|_{\Sigma}:d\alpha(v,w)=0 \ \forall w \in \xi_{\Sigma}\}.
\end{aligned}
\end{equation*}
Observe that $\xi_{\Sigma}$ is isomorphic to $T\Sigma$ and that $d\alpha$ is fiber-wise symplectic on each of $\xi_{\Sigma}$ and $\xi_{\Sigma}^{\bot}$.  This follows from the computation
\begin{equation*}
d\alpha(v-\alpha(v) R_{\alpha},w-\alpha(w) R_{\alpha})=d\beta(v,w)
\end{equation*}
for each pair of vectors $v,w\in T\Sigma$, coupled with the fact that $\xi_{\Sigma}^{\bot}$ is by definition the symplectic complement of $\xi_{\Sigma}$ with respect to $d\alpha$.

Fix a Riemannian metric $\langle \ast,\ast\rangle$ on $M$ and denote by $\exp:TM\rightarrow M$ the associated exponential map, sending each tangent vector $v_{x}\in T_{x}M$ to $\gamma(1)$ where $\gamma(t)$ is the unique geodesic in $M$ satisfying $\gamma(0)=x$ and $\partial_{t}\gamma(0)=v_{x}$.  To be completely rigorous, we should either restrict the domain of $\exp$ or assume that the metric $\langle\ast,\ast\rangle$ is complete so that the mapping is defined.  However, this will not be an issue as we will be applying $\exp$ to vectors of arbitrarily small length along $TM|_{\Sigma}$ and have assumed that $\Sigma$ is contained in the interior of $M$.

Denote by $\disk_{\epsilon}$ the collection of vectors in $\xi_{\Sigma}^{\bot}$ of length less than or equal to $\epsilon$ for an arbitrarily small constant $\epsilon>0$.  As $\exp$ has the property that for each $v_{x}\in TM$, $\partial_{t}(\exp(tv_{x}))=v_{x}$, that we can choose the constant $\epsilon$ to be small enough so that $\exp(\disk_{\epsilon})$ is embedded and symplectic with respect to the 2-form $d\alpha$.  For simplicity, we shall henceforth use the symbol $\disk_{\epsilon}$ to denote the image $\exp(\disk_{\epsilon})\subset M$.  These assumptions guarantee that $R_{\alpha}$ is transverse to $\disk_{\epsilon}$.  By appealing to this transversality and using the assumption that $\epsilon$ is chosen to be small enough so that $\disk_{\epsilon}$ is contained in the interior of $M$, we obtain a neighborhood $[-\delta,\delta]\times\disk_{\epsilon}$ of $\disk_{\epsilon}$ via the mapping
\begin{equation*}
(z,x)\mapsto \Flow^{z}_{R_{\alpha}}(x)\quad\text{for}\quad x\in\disk_{\epsilon}.
\end{equation*}
Indeed, choosing the constant $\delta>0$ to be sufficiently small, we may assume that this mapping is an embedding.  Note that $\alpha_{[-\delta,\delta]\times\disk_{\epsilon}}=dz+\alpha|_{T\disk_{\epsilon}}$, where we consider $\alpha|_{T\disk_{\epsilon}}$ to be a $z$-invariant 1-form which evaluates to zero on $\partial_{z}$.

Possibly after further shrinking $\epsilon$ there is an isotopy $\Phi_{t}:\disk_{\epsilon}\rightarrow\disk_{\epsilon}$, $t\in[0,1]$, such that $\Phi_{0}=\Id_{\disk_{\epsilon}}$, $\Phi_{t}|_{\Sigma}=\Id_{\Sigma}$ for all $t$, and $d\alpha|_{T\disk_{\epsilon}}$ is equal to the symplectic 2-form $\Phi_{1}^{*}d\widehat{\beta}$ determined by the symplectic vector bundle $(\xi_{\Sigma}^{\bot},d\alpha|_{\xi_{\Sigma}^{\bot}})$ as described in Theorem \ref{Thm:LiouvilleDiskBundle}.  This is a consequence of the fact that the symplectic forms $d\alpha|_{T\disk_{\epsilon}}$ and $\widehat{\beta}$ agree on $T\disk|_{\Sigma}$.  See \cite[Lemma 3.14]{MS:SymplecticIntro}.  Therefore $\alpha-\Phi_{1}^{*}\widehat{\beta}$ is closed.  We also observe that $\widehat{\beta}$ and $\alpha$ agree on $T\disk_{\epsilon}|_{\Sigma}$, implying that $\alpha|_{T\disk_{\epsilon}}-\Phi^{*}_{1}\widehat{\beta}$ is exact.  This allows us to find a function $f\in\Cinfty(\disk_{\epsilon},\mathbb{R})$ satisfying $\alpha|_{T\disk_{\epsilon}}-\Phi_{1}^{*}\widehat{\beta}=df$ and $f|_{\Sigma}=0$.

If necessary, further shrink $\epsilon$ so that $|f|<\delta$ on $\disk_{\epsilon}$.  To complete the proof, isotop $\disk_{\epsilon}$ to the graph of the function $-f$ in $[-\delta,\delta]\times\disk_{\epsilon}$.  Then $\alpha|_{T\disk_{\epsilon}}=\widehat{\beta}$.  By the properties of $\widehat{\beta}$ listed in Theorem \ref{Thm:LiouvilleDiskBundle}, we see that after rounding the corners $\partial(\pi^{-1}(\partial\Sigma))$ of $\disk_{\epsilon}$ we obtain a Liouville hypersurface $(\widehat{\Sigma},\widehat{\beta})\subset\Mxi$ as desired.
\end{proof}

\section{Examples of Liouville hypersurfaces}\label{Sec:LHSEx}

In this section we give some simple examples of Liouville hypersurfaces in contact manifolds.

\subsection{Legendrian graphs}
Let $L\subset\Mxi$ be a Legendrian submanifold of the $(2n+1)$-dimensional contact manifold $\Mxi$.  Then $L$ admits a tubular neighborhood
\begin{equation*}
N(L)=[-\epsilon,\epsilon]\times \disk^{*}L\quad\text{where}\quad \xi|_{N(L)}=\ker(dz-\lambda_{can})
\end{equation*}
where $z$ is a coordinate on $[-\epsilon,\epsilon]$.  Then $\{0\}\times\disk^{*}L$ is a Liouville hypersurface in $\Mxi$.  More generally, we can construct interesting Liouville hypersurfaces by considering \emph{Legendrian graphs} in $\Mxi$.

\begin{defn}\label{Def:Ribbon}
Let $\Mxi$ be a $(2n+1)$-dimensional contact manifold.
\be
\item A \emph{Legendrian graph} in $\Mxi$ is a pair $(L,\phi)$ where $L$ is a compact $(n)$-manifold and a Legendrian immersion $\phi:L\rightarrow M$ with only double point singularities, possibly occurring along $\partial L$.  At each double point $x=\phi(p)=\phi(q)$, $p\neq q$, we require that $\xi_{x}=(T_{p}L)\oplus(T_{q}L)$.

\item A \emph{ribbon} of a Legendrian graph $(L,\phi)$ is a smooth, compact $2n$-dimensional submanifold $R(\phi(L))$ of $M$ such that
    \be
    \item $\phi(L)\subset R(\phi(L))$ ,
    \item $R(\phi(L))$ deformation retracts onto $\phi(L)$, and
    \item for $x\in R(\phi(L))$, $\xi_{x}=T_{x}R(\phi(L))$ if and only if $x\in\phi(L)\subset R(\phi(L))$.
    \ee
\ee
\end{defn}

If $(L,\phi)$ is a Legendrian graph in $\Mxi$ then $\phi(L)$ admits a neighborhood of the form
\begin{equation*}
N(\phi(L))=[\epsilon,\epsilon]\times R(\phi(L))\quad\text{such that}\quad \xi|_{N(\phi(L))}=\ker(dz+\beta)
\end{equation*}
where $\beta$ is a Liouville 1-form on the ribbon $R(\phi(L))$.  A ribbon of a Legendrian graph $(L,\phi)$ is Liouville diffeomorphic to a plumbing of the cotangent bundle of $L$ at the double points of the immersion $\phi$.  For examples of Liouville domains constructed from plumbings of cotangent bundles, see \cite[\S 7]{Eliashberg:Plumbing}.  The above definition can easily be extended to \emph{isotropic graphs}.

\subsection{Inclusions}
By the results of Section \ref{Sec:HighCodimension}, the Liouville hypersurface property is well behaved with respect to inclusion mappings.  Suppose that $\Sdom$ is a Liouville hypersurface in a contact manifold $(C,\zeta)$ of dimension $(2k+1)$ and we realize $(C,\zeta)$ as a contact submanifold of a $(2n+1)$-dimensional contact manifold $\Mxi$.  Then $\Sdom$ is a Liouville submanifold of $\Mxi$ and we can apply Theorem \ref{Thm:HighCodimension}.

\subsection{Liouville hypersurfaces in unit cotangent bundles}

Let $L$ be a smooth manifold. We consider the unit cotangent bundle $(S^{\ast}L, \xi_{can})$ and write $\pi$ for the projection map $S^{*}L\rightarrow L$.

Suppose that $L$ is oriented and that $M$ is an oriented, codimension-one submanifold of $L$.  As noted in Example \ref{Ex:SmoothGluing}, $S=\pi^{-1}(M)$ is a convex hypersurface in $(S^{\ast}L, \xi_{can})$ whose dividing set is $\Gamma=S^{*}M$.  The orientation on $M$ allows us to specify one of the components of $\pi^{-1}(M)\setminus\Gamma$ as the positive region $S^{+}$ by taking the vector field $\partial_{t}$ from Example \ref{Ex:SmoothGluing} to be oriented so that it is positively transverse to $M$.  In this situation $(S^{+},-\lambda_{can})$ is the cotangent bundle $(T^{*}M,-\lambda_{can})$ of $M$.

\subsection{Counterexamples: Cabling and overtwisted submanifolds}
Now we give examples of null-homologous contact embeddings which do not bound Liouville hypersurfaces.  Let $\Mxi$ be a 5-dimensional contact manifold and let $T^{2}\subset M$ be a Legendrian torus.  Identify $S^{\ast}T^{2}=T^{3}\subset M$ as the boundary of the ribbon of $T^{2}$.  After fixing a trivialization of the normal bundle of $T^{3}$, we can identify a tubular neighborhood $N$ of $T^{3}$ with
\begin{equation*}
N=T^{3}\times\disk^{2},\quad \xi|_{N}=\ker\big{(}\sin(2\pi z)dx+\cos(2\pi z)dy + r^{2}d\theta\big{)}
\end{equation*}
where we consider coordinates $(x,y,z)$ on $T^{3}=([0,1]/0\sim 1)^{3}$ and polar coordinates on $\disk^{2}$.

For $R\in(0,1)$, consider the map $\Psi_{n,R}:T^{3}\rightarrow M$ given by
\begin{equation*}
\begin{gathered}
\Psi_{n,R}(x,y,z)=\big{(}(x,y,nz),Re^{2\pi iz}\big{)}\in N=T^{3}\times\disk^{2},\\
\quad \Psi_{n,R}^{*}\alpha=\sin(2\pi nz)dx+\cos(2\pi nz)dy+2\pi R^{2}dz
\end{gathered}
\end{equation*}
where $\alpha$ is the contact form on $N$ used above to describe $\xi|_{N}$.  The map $\Psi_{n,R}$ is a contact embedding whose image is null-homologous.  By applying Moser's trick to the family of contact forms $\Psi^{*}_{n,R}\alpha$ as $R$ goes to zero, it can be seen that the contact structure on $T^{3}$ determined by $\Psi_{n,R}$ for any $R\in(0,1)$ is the well known $(T^{3},\xi_{n})$, where
\begin{equation*}
\xi_{n} = \ker(\sin(2\pi nz)dx+\cos(2\pi nz)dy).
\end{equation*}
According to \cite{Eliashberg:Torus}, for $n>1$ these 3-tori are not symplectically fillable and so cannot bound a Liouville hypersurface in $\Mxi$.  However, as noted in \cite{Giroux:T3}, these tori are weakly symplectically fillable.

This example can be generalized to find ``cables'' of arbitrary codimension-2 submanifolds in a manifold of any dimension.  Suppose that $C$ is a closed, codimension-2 submanifold of a manifold $M$ whose normal bundle is trivial.  Fix an identification of a tubular neighborhood $N(C)$ of $C$ with $C\times \disk^{2}$.

Suppose that we have a surjective representation $\rho:\pi_{1}(C)\rightarrow \mathbb{Z}/q\mathbb{Z}$ of the fundamental group of $C$ into a finite cyclic group.  Denote by $\pi:\widetilde{C} \rightarrow C$ the universal cover and define $E_{\rho}\rightarrow C$ to be the complex line bundle over $C$ determined by
\begin{equation*}
\begin{gathered}
E_{\rho}=(\widetilde{C} \times \mathbb{C})/\sim,\\
(x,v)\sim (\gamma\cdot x, e^{\frac{\rho(\gamma)}{q}2\pi i}\cdot v),\quad \gamma\in \pi_{1}(C).
\end{gathered}
\end{equation*}
Here $\gamma\cdot x$ denotes the action of $\gamma\in \pi_{1}(C)$ on $x\in \widetilde{C}$ by deck transformation.  Let $D_{\rho}$ be the unit disk bundle in $E_{\rho}$ and suppose that we have a trivialization $\Phi:D_{\rho}\rightarrow C\times \disk^{2}$. Such a trivialization exists if $c_{1}(E_{\rho}) = 0$, which is guaranteed when $H^{2}(C, \mathbb{Z})$ is torsion-free, as $E_{\rho}$ has a flat connection with $S^{1}=U(1)$ holonomy by its construction.  This gives rise to an immersion of $\widetilde{C}$ into $M$ by
\be
\item identifying $\widetilde{C}$ with the section $ \widetilde{C}\times \{1\}\subset \widetilde{C}\times\mathbb{C}$,
\item immersing $\widetilde{C} \times \{1\}$ into $D_{\rho}$ using the projection $\widetilde{C} \times\mathbb{C}\rightarrow E_{\rho}$, and finally
\item applying the embedding $\Phi$ into $M$ using the identification $C \times \disk^{2}= N(C)$.
\ee
The image of this immersion will be a $q$-fold covering $\widetilde{C}_{q}$ of $C$ corresponding to the representation $\rho$. The image of the fundamental class of $\widetilde{C}_{\rho}$ is $[\widetilde{C}_{\rho}] = q[C] \in H_{\dim(M)-2}(M,\mathbb{Z})$.

In the simplest case, of an oriented unknot $C$ in $S^{3}$ with the obvious surjective representation $\rho:\pi_{1}(S^{1})\rightarrow \mathbb{Z}/q\mathbb{Z}$, this construction yields a $(p,q)$-torus knot.  The integer $p$ is determined by the trivialization of $N(C)$. As in the case of $T^{3}$ above, we can see that if $C$ is a contact submanifold $(C,\zeta)$ of a contact manifold $\Mxi$ this embedding is a contact embedding of $(\widetilde{C}_{\rho},\pi^{*}\zeta)$ into $\Mxi$.

This procedure can be used to find null-homologous embeddings of overtwisted contact $3$-manifolds into contact $5$-manifolds.  For example if $(C,\zeta)$ is exactly symplectically fillable with filling $\Sdom$, by taking a Liouville embedding of $\Sdom$ into $\Mxi$, we have a null-homologous contact embedding of $(C,\zeta)$ into $\Mxi$ as the boundary of $\Sdom\subset\Mxi$.  If $(C,\zeta)$ has a finite-cyclic, overtwisted cover then apply the above construction. Examples of closed, Weinstein fillable contact 3-manifolds which have finite-cyclic, overtwisted covers can be found in the work of Gompf \cite{Gompf:Handles} and Honda \cite{Honda:TightClassI}.

The cabling construction described above can be used in other contexts to produce interesting codimension-2 submanifolds of a given manifold.  For example, let $\Sigma$ be a closed, connected, 2-dimensional symplectic submanifold of a 4-dimensional symplectic manifold $(W,\omega)$.  Assume that $\Sigma$ has genus $g(\Sigma)\geq 1$ and self-intersection number $[\Sigma]\cdot[\Sigma]=0$.  This implies that the symplectic normal bundle to $\Sigma$ is trivial so that we can identify a neighborhood of $\Sigma$ in $W$ with $N(\Sigma)=\Sigma\times \disk^{2}$ where $\omega|_{N(\Sigma)}=d\lambda_{std}+\sigma$ for some symplectic form $\sigma$ on $\Sigma$.  Applying the cabling construction to a surjective representation $\rho:\pi_{1}(\Sigma)\rightarrow\mathbb{Z}/q\mathbb{Z}$ produces a closed, connected, embedded, symplectic surface $\widetilde{\Sigma}_{\rho}\subset W$ with genus $g(\widetilde{\Sigma}_{\rho})=q(g-1)+1$ and fundamental homology class $[\widetilde{\Sigma}_{\rho}]=q\cdot[\Sigma]\in H_{2}(W,\mathbb{Z})$.

\section{More basic consequences of Definition \ref{Def:Hypersurface}}\label{Sec:Liouville}

In this section we outline more of the basics of Liouville hypersurfaces in contact manifolds.  Special results for contact 3-manifolds are listed in Section \ref{Sec:DimThree}.  We begin with a discussion concerning a special family of contact vector fields called \emph{contact dilations}.

\subsection{Contact dilations}\label{Sec:Dilations}

Every standard neighborhood $\mathcal{N}(\Sigma)$ of a Liouville hypersurface $\Sdom$ in a contact manifold $\Mxi$ admits a special contact vector field which points out of $\partial \mathcal{N}(\Sigma)$.  Let $X_{\beta}$ be the Liouville vector field for $\Sdom$.  Then the vector field $V_{\beta}=z\partial_{z}+X_{\beta}$ of Equation \eqref{Eq:Vbeta} satisfies
\begin{equation}\label{Eq:Dilation}
\Lie_{V_{\beta}}\alpha = \alpha,
\end{equation}
and points transversely out of $\mathcal{N}(\Sigma)$ along its boundary so that $(S=\partial \mathcal{N}(\Sigma),V_{\beta})$ is a convex surface in $\Mxi$.

\begin{defn}\label{Def:Dilation}
A vector field for which there exists a contact form $\alpha$ such that the Lie derivative condition of Equation \eqref{Eq:Dilation} is satisfied will be referred to as a \emph{contact dilation}.
\end{defn}

Note that if a compact contact manifold $\Mxi$ admits a contact dilation which is defined on all of $M$, then the boundary of $M$ is necessarily non-empty.

\begin{prop}\label{Prop:PositiveRegion}
A hypersurface $\Sigma\subset M$ is Liouville if and only if there is a closed convex hypersurface $(S,X)$ in $\Mxi$ for which $\Sigma$ is the complement of a collar neighborhood of $\partial S^{+}$ in $S^{+}$.
\end{prop}

\begin{proof}
The ``if'' statement is a consequence of our ability to normalize contact forms in tubular neighborhoods in convex hypersurfaces as mentioned in Section \ref{Sec:Convex}.  For the ``only if'' statement, let $\Sigma\subset M$ be a Liouville hypersurface.  Then, in the notation of Section \ref{Sec:Neighborhood}, $\Sigma$ is isotopic through a family of Liouville hypersurfaces to  $\{\epsilon\}\times \widehat{\Sigma}\subset(\partial \mathcal{N}(\Sigma))^{+}$ for a standard neighborhood $\mathcal{N}(\Sigma)$ of $\Sigma$.
\end{proof}

Liouville hypersurfaces provide a simple means of partially characterizing contact dilations.

\begin{prop}
Let $\Mxi$ be a compact $(2n+1)$-dimensional contact manifold with contact 1-form $\alpha$. Suppose that $\Mxi$ has a contact dilation vector field $V$ for the contact form $\alpha$ which points out of $\partial M$. Then $\Mxi$ is contact-diffeomorphic to a neighborhood of a Liouville hypersurface.
\end{prop}

\begin{proof}
Let $z:M\rightarrow \mathbb{R}$ be the function $z=\alpha(V)$ and let $R_{\alpha}$ be the Reeb field for $\alpha$.  Then
\begin{equation}\label{Eq:DilationProof}
\begin{gathered}
\alpha=\Lie_{V}\alpha =d\alpha(V,\ast) + d(\alpha(V)) = d\alpha(V,\ast) + dz,\quad\text{and}\\
1=\alpha(R_{\alpha}) = d\alpha(V,R_{\alpha}) + dz(R_{\alpha}) = dz(R_{\alpha}).
\end{gathered}
\end{equation}
Therefore $dz$ is never zero and the vector field $R_{\alpha}$ is transverse to every level set of the function $z$.

As $V$ is positively transverse to $\partial M$, $\Mxi$ has convex boundary.  By the definition of the function $z$, the associated dividing set on $\partial M$ is
\begin{equation*}
\Gamma_{\partial M}= \partial M \cap \{ z=0\}.
\end{equation*}
Moreover, by the preceding paragraph $\Sigma :=\{z=0\}$ is a Liouville hypersurface in M with boundary equal to $\Gamma_{\partial M}$.  As $\Sigma$ is transverse to $\partial M$ in $M$, $\Sigma$ admits a tubular neighborhood $[-\epsilon,\epsilon]\times\Sigma$ on which $\alpha=dz+\beta$ by the results in Section \ref{Sec:Neighborhood}.  Here $\beta=\alpha|_{T\Sigma}$.  Indeed, the transversality of $R_{\alpha}$ with $\Sigma$ implies that $d\alpha$ is symplectic on $\Sigma$.  Moreover, $V$ is tangent to $\Sigma$, and when considered as a vector field on $\Sigma$ is the Liouville vector field for $\Sdom$.

Again, from the definition of $z$ and Equation \eqref{Eq:DilationProof}, $V$ points transversely out of $\mathcal{N}(\Sigma)$ and is non-vanishing on $M\setminus \mathcal{N}(\Sigma)$.  Therefore we can identify $M\setminus \mathcal{N}(\Sigma)$ as being contained in $[0,\infty)\times\partial \mathcal{N}(\Sigma)$ by placing a coordinate $s$ on $[0,\infty)$ such that $\partial_{s}=V$.  Because of the transversality of $V$ with $\partial M$, we have that $\partial M$ is the graph of a function $\partial \mathcal{N}(\Sigma)\rightarrow [0,\infty)$ contained in $[0,\infty)\times\partial \mathcal{N}(\Sigma)$.
\end{proof}

The above proposition is false without the assumption that the vector field $V$ is positively transverse to $\partial M$.  For example, if $S$ is a closed convex hypersurface in a contact manifold $\Mxi$, then Proposition \ref{Prop:Normalize} indicates that $S$ admits a tubular neighborhood of the form $N(S)=[-1,1]\times S$ on which there is a contact form $\alpha$ admitting a contact dilation pointing into $N(S)$ along $\{-1\}\times S$ and out of $N(S)$ along $\{1\}\times S$.

\subsection{Special properties when $\dim(M)=3$}\label{Sec:DimThree}

In this section we state some results specific to Liouville surfaces in 3-dimensional contact manifolds. We assume that the reader is familiar with the basics of transverse knots, self-linking numbers, and characteristic foliations.  For more information, see \cite{Etnyre:KnotNotes}.

\begin{prop}
Suppose that $\Mxi$ is a 3-dimensional contact manifold.
\be
\item $\Sigma\subset M$ is a Liouville surface if and only if it is isotopic through a 1-parameter family of Liouville surfaces to a ribbon (Definition \ref{Def:Ribbon}) of some Legendrian graph.
\item $\Mxi$ is tight if and only if every one of its Liouville surfaces is genus minimizing among embedded surfaces with the same boundary and in the same boundary-relative homology class.
\item $\Mxi$ is tight if and only if the boundary of every Liouville surface in $\Mxi$ is transversely non-destabilizeable.
\ee
\end{prop}

Item (2) appears in \cite[Theorem 8]{BCV:Ribbons} where additional results regarding ribbons of Legendrian graphs in overtwisted contact manifolds can also be found. In the second item we use the generalized definition of Seifert genus for links.  A surface $\Sigma$ bounding a link $L$ is of \emph{minimal genus} if it realizes the maximal Euler characteristic among all surfaces bounding $L$ with no sphere components.  Note that (2) can be used to compute genera of certain topological knots and links as in \cite[Theorem 2]{Gabai:Murasugi}.

The necessity of taking into account relative homology classes in item (2) can be seen in the following simple example:  Let $\Sigma^{-}$ be a torus with a disk removed and let $\Sigma^{+}$ be a genus 2 surface with a disk removed.  Denote by $\Sigma^{+}\cup\Sigma^{-}$ the closed genus three surface obtained by identifying the boundary components of $\Sigma^{+}$ and $\Sigma^{-}$.  Let $\Mxi$ be the $S^{1}$-invariant contact structure on $S^{1}\times (\Sigma^{+}\cup\Sigma ^{-})$  whose (oriented) dividing set on each slice $\{\theta\}\times(\Sigma^{+}\cup\Sigma^{-})$ is $\partial \Sigma^{+}$.  Then each $\lbrace \theta \rbrace \times \Sigma^{+}$ is a Liouville surface which does not realize the Seifert genus of its boundary.  $\Mxi$ is universally tight by Giroux's criterion and so is tight.

\begin{proof}
(1) For the first item, we see that $\Sigma$ is isotopic (as in the statement above) to $(\partial \mathcal{N}(\Sigma))^{+}$.  Using the Legendrian realization principle \cite[\S 3.3.1]{Honda:TightClassI} it is easy to construct a Legendrian graph onto which $(\partial\mathcal{N}(\Sigma))^{+}$ deformation retracts.

(2) This is immediate from (1) and \cite[Theorem 8]{BCV:Ribbons}.

(3) Suppose that $\Mxi$ is overtwisted.  Then the standard transverse unknot given by the boundary of a Liouville disk $(\disk^{2},\lambda_{std})\subset\Mxi$ is transversely destabilizeable.  See \cite[Theorem 3.2]{Etnyre:KnotNotes}. Therefore overtwistedness of $\Mxi$ implies the existence of Liouville surfaces with transversely destabilizeable boundaries. To complete the proof, we will show that tightness of $\Mxi$ would imply that the boundaries of all Liouville surfaces are non-destabilizeable.

So suppose that $\Mxi$ is tight and that $\Sigma$ is a Liouville surface in $\Mxi$ bounding a transverse link $T=\sqcup T_{j}$.  Then the \emph{Thurston-Bennequin inequality} applies.  This asserts that for any transverse link $T$ bounding an embedded, oriented surface $\Sigma\subset M$, the inequality
\begin{equation}\label{Eq:BEBound}
s\ell(T,\Sigma)\leq -\chi(\Sigma)
\end{equation}
is satisfied.  Eliashberg's proof \cite{Eliashberg:Ineq} of this inequality relies on the inequality
\begin{equation}\label{Eq:Singularities}
s\ell(T,\Sigma)+\chi(\Sigma) = e^{\Sigma}_{-} - h^{\Sigma}_{-}\leq 0
\end{equation}
where $e^{\Sigma}_{-}$ and $h_{-}^{\Sigma}$ are the number of negative elliptic and negative hyperbolic singularities of a generic characteristic foliation on the surface $\Sigma$, respectively.  As the Liouville condition is an open condition, we may assume that the characteristic foliation of $\Sigma$ is generic -- i.e. its singularities are isolated and of Morse type -- and apply Equation \eqref{Eq:Singularities}.  In this case the numbers $e_{-}^{\Sigma}$ and $h_{-}^{\Sigma}$ are both zero as by definition there is a contact form $\alpha$ for $\Mxi$ for which $d\alpha|_{T\Sigma}$ is symplectic. Hence, $\Sigma$ satisfies the equality $s\ell(T,\Sigma)=-\chi(\Sigma)$.

Now suppose that some component $T_{j}$ of $T$ is a transverse stabilization of a transverse knot $T'_{j}$ in the complement of $M\setminus (T\setminus T_{j})$.  Then we would be able to find another embedded surface $S$, which is smoothly isotopic to $\Sigma$ bounding $T'=(T\setminus T_{j})\cup T'_{j}$.  Therefore, we would have
\begin{equation*}
s\ell(T',S)=s\ell(T,\Sigma)+1=-\chi(\Sigma)+1
\end{equation*}
contradicting Equation \eqref{Eq:BEBound}.
\end{proof}

It would be interesting to know what transverse knots bound Liouville surfaces.

\begin{q}
Is it possible to characterize the transverse links in $\Sthree$ which bound Liouville surfaces in terms of braid theory?
\end{q}

Implicit in the above question is the fact, due to Bennequin \cite{Bennequin}, that every transverse link in $\Sthree$ can be represented as a transverse braid.  See \cite[\S 2.4]{Etnyre:KnotNotes} and the references therein.

It would also be interesting to know whether or not a Liouville surface bounding a given transverse link is unique.

\begin{q}
Does there exist a transverse link $T$ in $\Sthree$ and two Liouville surfaces $\Sigma,\Sigma'\subset\Sthree$ for which $\partial\Sigma=\partial\Sigma'=T$ and such that $\Sigma$ is not isotopic to $\Sigma'$ through a family of Liouville surfaces?
\end{q}

An answer to the above question in the affirmative would analogous to the non-uniqueness of minimal genus Seifert surfaces of topological knots in $\mathbb{R}^{3}$.  See, for example, \cite[\S 5.A]{Rolfsen}.

\section{Applications of Theorem \ref{Thm:Cobordism}}\label{Sec:Applications}

In this section we provide proofs of most of the applications of Theorem \ref{Thm:Cobordism} stated in Section \ref{Sec:TopologicalApplications}.  The proof of each theorem will provide an example of a Liouville connect sum.

\subsection{Open books and mapping class monoids}\label{Sec:Monoids}

The purpose of this section is to prove Theorem \ref{Thm:Monoids}.

The relationship between symplectomorphism groups of Liouville domains and contact manifolds established in Theorem \ref{Thm:GirCor} has attracted a great deal of interest, especially in dimension three.  As an example, Baker-Etnyre-van Horn-Morris \cite[\S 1.2]{BEV:Monoids} and Baldwin \cite[Theorems 1.1 - 1.3]{Baldwin:Monoids} have shown that for a compact oriented surface $\Sigma$ with $\partial \Sigma\ne\emptyset$, the contact manifolds supported by open books with page $\Sigma$ which are fillable (in any of the senses of Definitions \ref{Def:Fillable} and \ref{Def:Weak}) constitute a monoid of $\SympGroup$.  The Liouville connect sum and Theorem \ref{Thm:Cobordism} provide a natural generalization of this result to open books whose pages are Liouville domains of any even dimension, as stated in Theorem \ref{Thm:Monoids}.

\begin{figure}[h]
	\begin{overpic}[scale=.7]{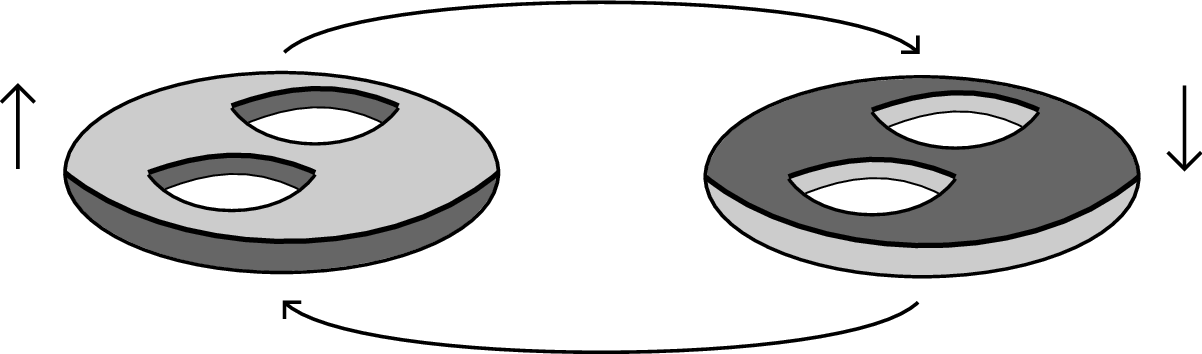}
        \put(0.5,12){$\partial_{z}$}
        \put(97.5,12){$\partial_{z}$}
        \put(49,25){$\Phi$}
        \put(49,2){$\Psi$}
    \end{overpic}
    \vspace{1.5mm}
	\caption{A Heegaard decomposition of a contact manifold $\Mxi_{(\Sdom,\Phi\circ\Psi)}$ determined by an open book decomposition.  The maps $\Phi$ and $\Psi$ provide instructions for performing a convex gluing as described in Section \ref{Sec:GluingInstructions}.}
    \label{Fig:Heegaard}
\end{figure}

The first step in the proof of Theorem \ref{Thm:Monoids} is the following lemma, which allows us to translate open book descriptions of contact manifolds into Heegaard decompositions. See Figure \ref{Fig:Heegaard}.\footnote{
Our use of the expression ``Heegaard decomposition'' is, of course, informal when speaking of contact manifolds whose dimensions are greater than three.}

\begin{lemma}\label{Lemma:Heegaard}
Let $(N_{j},\xi_{\Sdom})$ ($j=1,2$) be two standard neighborhoods of a Liouville domain $\Sdom$ and let $\Phi,\Psi\in\SympGroup$.  Define the contact manifold $\Mxi$ by the convex gluing instructions $(\Phi,\Psi):\partial N_{1}\rightarrow\partial N_{2}$ so that $(\Phi,\Psi)$ maps
\be
\item $(\partial N_{1})^{+}$ to $(\partial N_{2})^{-}$ via $\Phi$ and
\item $(\partial N_{2})^{+}$ to $(\partial N_{1})^{-}$ via $\Psi$.
\ee
See Figure \ref{Fig:Heegaard}.  Then $\Mxi$ is diffeomorphic to the contact manifold $\Mxi_{(\Sdom,\Phi\circ\Psi)}$ determined by the pair $(\Sdom,\Phi\circ\Psi)$.
\end{lemma}

\begin{proof}
This is a slight modification of the proof of Theorem \ref{Thm:GirCor}(1).  See \cite[\S 3]{Etnyre:OBIntro}.
\end{proof}

\begin{prop}\label{Prop:OBSum}
Let $\Sdom$ be a Liouville domain and let $\Phi,\Psi\in\Symp((\Sigma,d\beta),\partial\Sigma)$.  Then $\Mxi_{(\Sdom,\Phi\circ\Psi)}$ can be obtained from $\Mxi_{(\Sdom,\Phi)}\sqcup\Mxi_{(\Sdom,\Psi)}$ by a Liouville connect sum.
\end{prop}

\begin{proof}
Let $N_{1},N_{2},N_{1}'$ and $N_{2}'$ be copies of a standard neighborhood of $\Sigma$, each endowed with the contact structure determined by the contact form $dz+\beta$ as described in Section \ref{Sec:Neighborhood}.  By Lemma \ref{Lemma:Heegaard} we can construct $\Mxi_{(\Sdom,\Phi)}$ by identifying $\partial N_{1}$ and $\partial N_{2}$ using the convex gluing instructions $(\Phi,\Id_{\Sigma})$.  Similarly, we can construct $\Mxi_{(\Sdom,\Psi)}$ by identifying $\partial N_{1}'$ and $\partial N_{2}'$ using the convex gluing instructions $(\Psi,\Id_{\Sigma})$.

Now perform a Liouville connect sum on $\Mxi:=\Mxi_{(\Sdom,\Phi)}\sqcup\Mxi_{(\Sdom,\Psi)}$ by removing $N_{2}$ and $N_{2}'$ and then identifying $\partial N_{1}$ and $\partial N_{1}'$.  Then the resulting contact manifold $\#_{\Sdom}\Mxi$ can be described by identifying $\partial N_{1}$ and $\partial N_{1}'$ using the convex gluing instructions $(\Phi,\Psi)$.  By Lemma \ref{Lemma:Heegaard}, we have $\#_{\Sdom}\Mxi=\Mxi_{(\Sdom,\Phi\circ\Psi)}$.
\end{proof}

\begin{proof}[Proof of Theorems \ref{Thm:Monoids} and \ref{Thm:WeakMonoids}]
Items (1) and (3) of Theorem \ref{Thm:Monoids} are immediate from Proposition \ref{Prop:OBSum}, Theorem \ref{Thm:Cobordism}, and the fact that a composition of symplectic (exact, Weinstein) cobordisms is a symplectic (resp. exact, Weinstein) cobordism.  For Theorem \ref{Thm:Monoids}(2), the cobordism $\Ldom$ from Theorem \ref{Thm:Cobordism} can be obtained by a sequence of Weinstein handle attachments as both $\Mxi_{(\Sdom,\Phi)}$ and $\Mxi_{(\Sdom,\Psi)}$ are 3-dimensional. Regarding the weak symplectic fillings of Theorem \ref{Thm:WeakMonoids}, we note that the hypothesis of the theorem guarantees that Theorem \ref{Thm:WeakHandle} can be applied, providing a weak symplectic filling of $\Mxi_{(\Sdom,\Phi\circ\Psi)}$.
\end{proof}

\begin{rmk}
In the case $\dim(\Sigma)=2$, the proof of Theorem \ref{Thm:Monoids} coincides with the proofs of \cite[Theorem 1.3]{BEV:Monoids} and \cite[Theorem 1.1]{Baldwin:Monoids}.  This can be worked out by analyzing their proofs, the proof of Theorem \ref{Thm:Monoids}, and the proof of the second statement of Theorem \ref{Thm:Cobordism} appearing in Section \ref{Sec:WHandleDecomposition}.  Further intuition can be obtained by reading Section \ref{Sec:Kirby}.
\end{rmk}

\subsection{Mapping class monoids from contact homology}\label{Sec:HCMonoids}

The purpose of this section is to prove Corollary \ref{Cor:HCMonoids} as a consequence of Proposition \ref{Prop:OBSum}. The corollary follows from some basic properties of contact homology ($HC_{\ast}$), a symplectic field theory (SFT) invariant of closed contact manifolds \cite{SFTIntro}. We recommend the exposition \cite{Bourgeois:Notes} for a basic introduction. Rigorous proofs of its well definition and functoriality for exact symplect cobordisms are due to Bao-Honda \cite{BaoHonda:ContactHomology} and Pardon \cite{Pardon:ContactHomology}. The following theorem -- following the exposition of \cite[Section 1]{Pardon:ContactHomology} -- summarizes the basic properties of $HC_{\ast}$ which we will need for Corollary \ref{Cor:HCMonoids}.

\begin{thm}\label{Thm:HCSummary}
To each closed contact manifold $\Mxi$, we may associate a $\mathbb{Z}/2\mathbb{Z}$-graded $\mathbb{Q}$-algebra $HC_{\ast}\Mxi$ satisfying the following properties:
\be
\item For a pair of contact manifolds of the same dimension, $(M_{1}, \xi_{1})$ and $(M_{2}, \xi_{2})$,
\begin{equation*}
HC_{\ast}(M_{1} \sqcup M_{2}, \xi_{1}\sqcup \xi_{2}) \simeq HC_{\ast}(M_{1}, \xi_{1}) \otimes HC_{\ast}(M_{2},\xi_{2}).
\end{equation*}
\item Every symplectic cobordism $\Ldom$ with convex boundary $\Mxi$ and concave boundary $(M', \xi')$ induces a unital algebra homomorphism
\begin{equation*}
\Phi_{\Ldom}:HC_{*}(M,\xi)\rightarrow HC_{*}(M',\xi').
\end{equation*}
\ee
\end{thm}

In the last statement above, we note that if $HC_{*}(M',\xi') \neq 0$, then $HC_{*}(M,\xi)\neq 0$. For if $1 \neq 0$ in $HC_{*}(M',\xi')$ and $0 = 1 \in HC_{*}(M',\xi')$, then we have the contradictory statement $0 = \Phi_{\Ldom}(0) = 1 \in HC_{*}(M,\xi)$ from the fact that $\Phi_{\Ldom}$ preserves multaplicative units. According to a theorem of Eliashberg  and Yau \cite{Yau}, the contact homology of an overtwisted contact manifold is zero.

\begin{proof}[Proof of Corollary \ref{Cor:HCMonoids}]
Suppose as in the statement of the corollary that $\Sdom$ is a Liouville domain and $\Phi,\Psi\in\SympGroup$ are such that the contact manifolds $\Mxi_{(\Sdom,\Phi)}$ and $\Mxi_{(\Sdom,\Psi)}$ both have non-vanishing contact homology with coefficient ring $\mathbb{Q}$.  By the above theorem, the disjoint union $\Mxi_{(\Sdom,\Phi)}\sqcup\Mxi_{(\Sdom,\Psi)}$ also has non-vanishing contact homology.  Applying the exact symplectic cobordism provided by Proposition \ref{Prop:OBSum}, the contact homology of $\Mxi_{(\Sdom,\Phi\circ\Psi)}$ must also be non-zero.
\end{proof}

\subsection{Contact manifolds which fiber over the circle}\label{Sec:Fibration}

In this section we discuss symplectic fillability of the contact manifolds $\Mxi_{(\Sdom,\Phi,\Psi)}$ described in the discussion preceding the statement of Theorem \ref{Thm:Fibration}.

The contact manifolds $\Mxi_{(\Sdom,\Phi,\Psi)}$ can also be described using supporting open books and the contact fiber sum \cite[\S 3]{Geiges:Fiber}.  Perform a contact fiber sum of the contact manifolds $\Mxi_{(\Sdom,\Phi)}$ and $\Mxi_{(\Sdom,\Psi)}$ along the bindings of their associated open books, using the pages of the open books to frame the relevant normal bundles.  The resulting contact manifold will be $\Mxi_{(\Sdom,\Phi,\Psi)}$.  Using this description, if follows from Theorem \ref{Thm:GirCor} the the contact isotopy class of $\Mxi_{(\Sdom,\Phi,\Psi)}$ depends only on the isotopy classes of $\Phi$ and $\Psi$ in $\SympGroup$. In dimension three, this is an example of the blown-up, summed open book construction described in \cite{Wendl:Cobordism}.

To prove Theorem \ref{Thm:Fibration} we follow the same strategy as the proof of Theorem \ref{Thm:Monoids}, considering Heegaard splitting-type decompositions of contact manifolds determined by open books.  This time, instead of Liouville connect summing pages of \emph{distinct} open books, we apply the Liouville connect sum to the interiors of two pages of the \emph{same} open book. Again, we refer to Remark \ref{Rmk:OBConvention} regarding our choice of convention for gluing instructions.

\begin{figure}[h]
	\begin{overpic}[scale=.7]{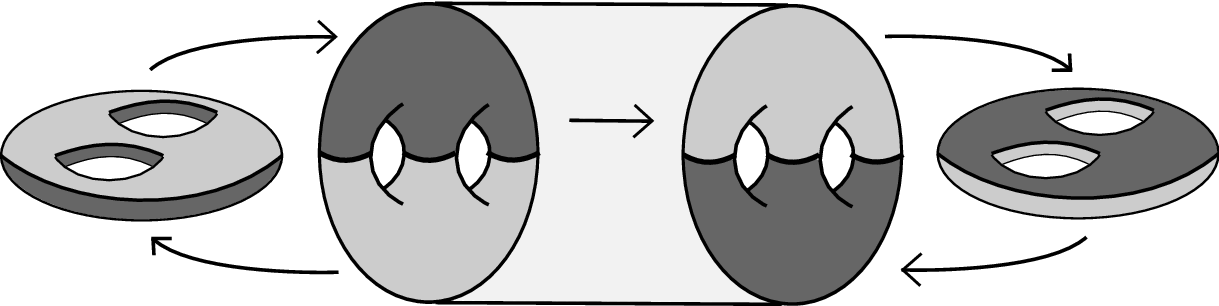}
        \put(14,22){$\Phi$}
        \put(49,12){$\partial_{\theta}$}
        \put(85,22){$\Id_{\Sigma}$}
        \put(85,0){$\Psi$}
        \put(14,0){$\Id_{\Sigma}$}
        \put(-4,10){$N_{1}$}
        \put(101,10){$N_{2}$}
        \put(41.5,-3){$[-1,1]\times \partial \mathcal{N}(\Sigma)$}
    \end{overpic}
    \vspace{4.5mm}
	\caption{A decomposition of the contact manifold $\Mxi_{(\Sdom,\Phi\circ\Psi)}$ into three pieces with convex gluing instructions.}
    \label{Fig:Heegaard2}
\end{figure}

\begin{lemma}\label{Lemma:Fibered}
Let $\Sdom$ be a Liouville domain and let $\Phi,\Psi\in \SympGroup$.  The contact manifold $\Mxi_{(\Sdom,\Phi,\Psi)}$ can be obtained from a contact manifold supported by an open book determined by the pair $(\Sdom,\Phi\circ\Psi)$ by a Liouville connect sum.
\end{lemma}

\begin{proof}
We decompose the contact manifold $\Mxi_{(\Sdom,\Phi\circ\Psi)}$ into three pieces $[-1,1]\times \partial \mathcal{N}(\Sigma),N_{1}$ and $N_{2}$. This decomposition will be a slight modification of the Heegaard splitting-type decomposition used in Lemma \ref{Lemma:Heegaard}.

Take $N_{1}$ and $N_{2}$ to be standard neighborhoods of the Liouville domain $\Sdom$.  Attach $\{ -1\}\times \partial \mathcal{N}(\Sigma)$ to $\partial N_{1}$ using the convex gluing instructions $(\Id_{\Sigma},\Phi)$.  Similarly, attach $\{1\}\times\partial \mathcal{N}(\Sigma)$ to $\partial N_{2}$ using the convex gluing instructions $(\Id_{\Sigma},\Psi)$.  The resulting contact manifold is $\Mxi_{(\Sdom,\Phi\circ\Psi)}$ as can be seen from Lemma \ref{Lemma:Heegaard}.  See Figure \ref{Fig:Heegaard2}.

Now perform a Liouville connect sum on $\Mxi_{(\Sdom,\Phi\circ\Psi)}$ along the standard neighborhoods of Liouville hypersurfaces $N_{1}$ and $N_{2}$.  This may be done by removing $N_{1}$ and $N_{2}$ from $\Mxi_{(\Sdom,\Phi\circ\Psi)}$ and gluing together the new convex boundary components.  By the identifications described in the previous paragraph, the resulting contact manifold is exactly $\Mxi_{(\Sdom,\Phi,\Psi)}$.
\end{proof}

Together with Theorem \ref{Thm:Cobordism}, the above lemma immediately proves Theorem \ref{Thm:Fibration}(1-2).  As for the statement regarding weak symplectic fillings, we must show that the cohomological condition described in the statement of Theorem \ref{Thm:Fibration} coincides with the one described in the statement of Theorem \ref{Thm:WeakHandle}.  We observe that the Liouville embeddings required to perform the necessary symplectic handle attachment must agree with the submanifolds $N_{1}$ and $N_{2}$ in the above proof.  Note that if we isotop $N_{2}$ through the region $[-1,1]\times\mathcal{N}(\Sigma)$ and into $N_{1}$ counterclockwise through the diagram shown in Figure \ref{Fig:Heegaard2}, we see that $\omega|_{N_{2}}=(\Phi^{-1})^{*}\omega|_{N_{1}}$ giving the cohomological obstruction described in Theorem \ref{Thm:Fibration} whose vanishing is required by Theorem \ref{Thm:WeakHandle}.

\subsection{Fillability of branched covers}\label{Sec:Branched}

In this section we apply Theorem \ref{Thm:Cobordism} to study branched covers. We begin by reviewing some known results on branched coverings of contact manifolds. The following theorem is a consequence of Gironella's \cite{Gironella:Branches}, refining results of Geiges \cite{Geiges:Fiber} and Gonzalo \cite{Gonzalo}.

\begin{thm}\label{Thm:BranchedCover}
Let $\Mxi$ be a $(2n+1)$-dimensional contact manifold and let $(C,\zeta)\subset\Mxi$ be a closed, connected, codimension two contact submanifold with trivial normal bundle.  Let $\pi:\widetilde{M}\rightarrow M$ be a branched cover of $M$ with branch locus $C \subset M$.  Then $\widetilde{M}$ naturally carries a contact structure $\xi_{\pi}$ for which the associated unbranched covering
\begin{equation*}
\pi:\big(\widetilde{M}\setminus\pi^{-1}(N(C)),\xi_{\pi}\big)\rightarrow \big(M\setminus N(C),\xi\big)
\end{equation*}
satisfies $T\pi(\xi_{\pi})=\xi$ where $N(C)$ is an arbitrarily small tubular neighborhood of $C$.
\end{thm}

For a precise statement regarding the naturality of the contact structure $\xi_{\pi}$ described above, we refer to \cite[Proposition A]{Gironella:Branches}. The following summarizes known results regarding branched coverings of contact 3-manifolds. Note that a $k$-fold cyclic branched cover $\Sthree_{C, k}$ over a transverse knot $C \subset \Sthree$ is uniquely determined by the branch locus and branch index.

\begin{thm}\label{Thm:BCEx}
Let $\Mxi$ be a 3-dimensional contact manifold containing the transverse link $C$.
\be
\item $\Mxi$ can be described as a (not necessarily cyclic) branched cover over a transverse link in $\Sthree$ \cite{Giroux:ContactOB,MM:FiberedBraids}.
\item If $C$ is a knot which realizes its Bennequin bound (Equation \eqref{Eq:BEBound}), then there is a Weinstein cobordism with concave boundary $\Mxi$ and convex boundary $\Mxi_{C,k}$ \cite{Baldwin:Monoids}.
\item If $C$ is a knot, for any $k > 0$ the contact distribution on $\Sthree_{C,k}$ satisfies $c_{1}(\xi_{C,k})=0$ and its homotopy class depends only on the self-linking number and topological type of $C$.  If $C$ is destabilizeable, then a cyclic branched cover over $C$ is overtwisted.  If $C$ can be represented as a quasipositive braid, then a cyclic branched cover over $C$ is Weinstein fillable \cite{HKP:Branched}.
\ee
\end{thm}

Now we describe the branched coverings appearing in Theorem \ref{Thm:Branched}. Our discussion follows the construction of cyclic branched covers, branched over null-homologous links in 3-manifolds described in \cite[Chapter 5, Section C]{Rolfsen} with slightly modified notation. As in the discussion preceding the statement of Theorem \ref{Thm:Branched}, we assume that $M$ is a manifold containing a compact, codimension $1$, properly embedded submanifold $\Sigma \subset M$ with non-empty boundary. 

The manifold $M_{\Sigma, k}$ is constructed as follows:
\be
\item Identify a neighborhood of $\Sigma$ with $[0, 1] \times\Sigma$.
\item Write $N_{j}$ for the set $[\frac{(j-1)}{k}, \frac{j}{k}]\times\Sigma\subset[0,1]\times\Sigma$ for $j=1,\dots,k-1$.
\item Consider $(k-1)$ additional copies of $M$, labeled $M_{j}$.  Define $N'_{j}=[0,\frac{1}{k}]\times\Sigma\subset M_{j}$.
\item Define $M^{0}=M$.  Inductively define $M^{j}=(M^{j-1}\setminus N_{j})\cup_{\Phi_{j}}(M_{j}\setminus N'_{j})$ where $\Phi_{j}:\partial N'_{j}\rightarrow \partial N_{j}$ is given by
\begin{equation*}
\Phi_{j}(z,x)=\big{(}\frac{j}{k}-z,x\big{)}\quad\text{for}\quad z\in[0,\frac{1}{k}],\ x\in\Sigma.
\end{equation*}
\ee
The map $\Phi_{j}$ above sends the bottom $\{ 0 \}\times \Sigma \subset M_{j}$ of each $N'_{j} \subset M_{j}$ to the top $\{\frac{j}{k} \}\times \Sigma \subset M$ of each $N_{j}$ and the top of each $N'_{j}$ to the bottom of each $N_{j}$.

The above construction provides a piece-wise linear description of $M_{\Sigma, k}$. To make this construction smooth, we can round the edges of each $\partial N_{j}$ to obtain some $\mathcal{N}_{j}$ and perform the gluing using maps 
\begin{equation*}
\widehat{\Upsilon}: (-\epsilon, \epsilon)\times \partial \mathcal{N}_{j}' \rightarrow (-\epsilon, \epsilon)\times \partial \mathcal{N}_{j}
\end{equation*}
described in Equation \eqref{Eq:Upsilon}. 

The proof of Theorem \ref{Thm:Branched} is then immediate from the construction described as each such gluing used to smoothly define $M_{\Sigma, k}$ may be realized as a Liouville connect sum provided that $\Sigma$ is a Liouville hypersurface $\Sdom$ and each $N_{k}$ is as described in Lemma \ref{Lemma:SquareNeighborhood}. To see that this construction recovers the branched cyclic coverings over transverse knots appearing in \cite{Baldwin:Monoids, HKP:Branched} we follow \cite{Rolfsen} in which case any -- not necessarily Liouville -- Seifert surface $\Sigma$ of a transverse knot may be used to construct the branched cover. The branched covering
\begin{equation*}
M_{\Sigma, k} \rightarrow M
\end{equation*}
as described in \cite[Chapter 5, Section C]{Rolfsen} has branch locus $\partial \Sigma$ -- a contact submanifold of $\Mxi$ -- and so determines a contact branched covering as described in Theorem \ref{Thm:BranchedCover}.

\subsection{Kirby diagrams from the proof of Theorem \ref{Thm:Cobordism}}\label{Sec:Kirby}

Now we will give an example of a Weinstein cobordism associated to a branched cover as described in Theorem \ref{Thm:Branched}.  By combining the proof of this theorem with the exposition in Section \ref{Sec:WHandleDecomposition}, we will be able to give a Kirby diagram description of the cobordism as in \cite{Gompf:Handles}.  This example should serve as a guide as to how to use the proof of Theorem \ref{Thm:Cobordism} -- in the Weinstein case -- to explicitly describe cobordisms associated to the Liouville connect sum in terms of Weinstein handle attachment.  For a similar construction, see \cite{HKP:Branched} where an algorithm is described which produces a contact surgery diagram of a cyclic branched cover, branched over a transverse braid in $\Sthree$.

Throughout this section figures will be drawn in the front projection $\mathbb{R}^{3}\rightarrow \{0\}\times \mathbb{R}^{2}$ of 
\begin{equation*}
(\mathbb{R}^{3},\xi_{std}=\ker(dz-ydx)).
\end{equation*}
Here $\Rthree$ is identified with the complement of a point in $\Sthree$.

\begin{figure}[h]
\begin{overpic}[scale=.7]{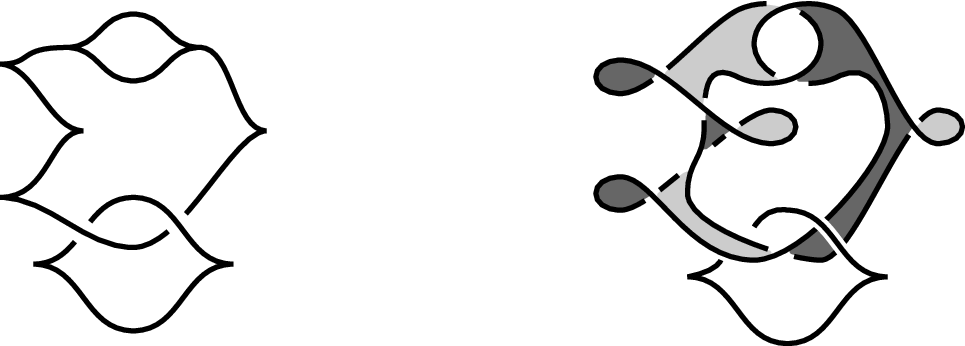}
    \put(25.5,8){$-1$}
    \put(94,6.5){$-1$}
\end{overpic}
\vspace{4.5mm}
\caption{On the left is a Legendrian graph in the contact manifold $(L(2,1),\xi_{std})$.  The ambient contact manifold is presented as the result of a Legendrian (or, equivalently, a contact $-1$) surgery on a Legendrian unknot with $\tb=-1$.  We will omit the surgery coefficient associated to this unknot in subsequent diagrams.  On the right hand side of the figure is the ribbon $\Sigma$ of the graph.  The boundary of $\Sigma$ is the transverse knot $T$.}
\label{Fig:WhiteheadDouble}
\end{figure}

Consider the Legendrian graph in Figure \ref{Fig:WhiteheadDouble}.  Using \cite[\S 4]{Avdek} we can draw its ribbon $\Sigma$ in the front projection.  The boundary $C$ of this ribbon is a Whitehead double of a homologically non-trivial knot in the contact lens space $(L(2,1),\xi_{std})=(S^{*}S^{2},\xi_{can})$ -- which can be described by a Legendrian surgery along an unknot with $\tb=-1$.

\begin{figure}[h]
	\begin{overpic}[scale=.7]{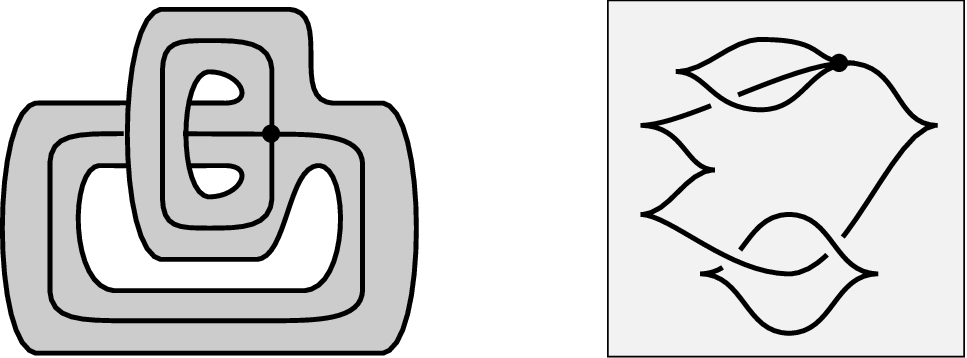}
    \put(50,29.5){$a$}
    \put(49,30){\vector(-1,0){18.5}}
    \put(52,30){\vector(1,0){17.5}}
    \put(50,14){$b$}
    \put(49,15){\vector(-1,0){9.5}}
    \put(52,15){\vector(1,0){14}}
    \put(93,7.5){$-1$}
    \end{overpic}
    \vspace{4.5mm}
	\caption{A Weinstein handle decomposition of the surface $\Sigma$ gives rise to an isotropic graph in the ambient contact manifold.}
    \label{Fig:SteinDecomp}
\end{figure}

By Theorem \ref{Thm:Branched}, there is a Weinstein cobordism with concave boundary $\sqcup^{k}(L(2,1),\xi_{std})$ and whose convex boundary is the $k$-fold cyclic branched cover $\Sthree_{C,k}$.  We will provide a Kirby diagram for this cobordism in the case $q=k$ and then describe a completed diagram for the case $q=k$.  According to the proof of Theorem \ref{Thm:Branched}, we can describe $\Sthree_{C,2}$ by taking two copies of $(L(2,1),\xi_{std})$ each containing a copy of $\Sigma$ and then perform a Liouville connect sum to the disjoint union of two copies of $L(2,1)$ by identifying the copies of $\Sigma$.

In our situation, the Liouville surface $\Sigma$ is a genus 1 surface with a single non-empty boundary component.  Therefore $\Sigma$ admits a Weinstein handle decomposition as a pair of 2-dimensional 1-handles attached to a single disk as depicted in the left-hand side of Figure \ref{Fig:SteinDecomp}.  There, the 0-handle is marked with a black dot which we will call $p$.  We label the curves of the core disks of the 2-dimensional 1-handles $a$ and $b$.  The left-hand side of Figure \ref{Fig:SteinDecomp} shows the curves $a$ and $b$ embedded in $(L(2,1),\xi_{std})$.  Consider $a$ and $b$ to be oriented counterclockwise in the figure.

\begin{figure}[h]
	\begin{overpic}[scale=.7]{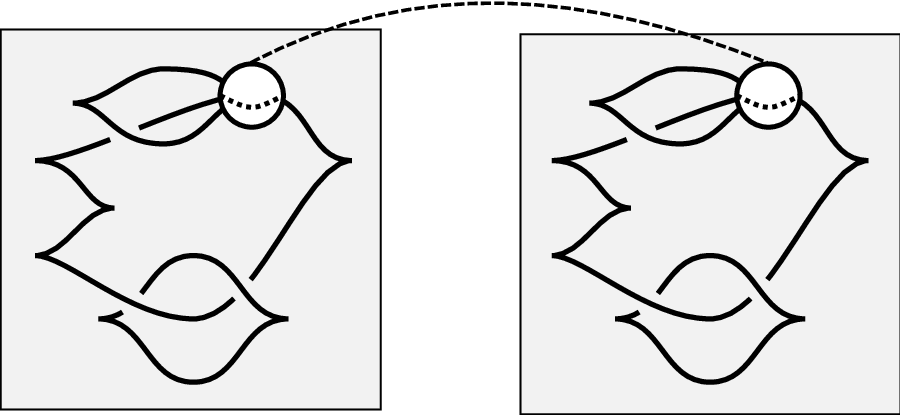}
    \put(33,10){$-1$}
    \put(91,10){$-1$}
    \end{overpic}
    \vspace{4.5mm}
	\caption{A Kirby diagram for the 2-fold cyclic branched cover of $(L(2,1),\xi_{std})$ over $T$.  The boxes represent the fact that the concave end of the cobordism is disconnected. Weinstein 1-handles are represented by spheres connected by dashed lines.  The Legendrian knots in the diagram are the attaching loci of the Weinstein 2-handles, and are drawn as thick, black lines.  A description of how Legendrian arcs are identified when passing through the 1-handles is described in the text.}
    \label{Fig:2CyclicCover}
\end{figure}

Now we consider two disjoint copies of $(L(2,1),\xi_{std})$, each containing the surface $\Sigma$, and so the graph $a\cup b\cup p$.  We will call one of the surfaces $\Sigma_{1}$ and the other $\Sigma_{2}$ and fix a diffeomorphism between them induced from the identification of the two copies of $(L(2,1),\xi_{std})$.  Similarly we will label the graphs $a\cup b\cup p$ in each of the copies of $L(2,1)$ by $a_{j}\cup b_{j}\cup p_{j}$, $j=1,2$.

The proof of Theorem \ref{Thm:Branched} tells us that we can describe the double branched cover of $(L(2,1),\xi_{std})$ over $T$ by Liouville connect summing the two copies of $L(2,1)$ along $\Sigma_{1}$ and $\Sigma_{2}$, using the identification $\Sigma_{1}\cong\Sigma_{2}$ described in the previous paragraph.  Theorem \ref{Thm:Cobordism} then tells us that this double branched cover can be realized as the convex boundary component of a symplectic cobordism $\Ldom$ whose concave boundary is $\sqcup^{2} (L(2,1),\xi_{std})$.  The proof of Theorem \ref{Thm:Cobordism}(2) described in Section \ref{Sec:WHandleDecomposition} provides a handle decomposition of this cobordism as follows:
\be
\item Each 2-dimensional 0-handle of the surface $\Sigma$ gives rise to a 4-dimensional 1-handle in $\Ldom$.  We attach a 4-dimensional 1-handle to the finite symplectization of $\sqcup^{2}(L(2,1),\xi_{std})$ along 3-dimensional disks centered about the points $p_{1}$ and $p_{2}$.
\item Each 2-dimensional 1-handle of $\Sigma$ gives rise to a 4-dimensional 1-handle in $\Ldom$.  The proof shows that we are to attach one of these two handles along the Legendrian knot $a_{1}\cup(-a_{2})$ and another along $b_{1}\cup(-b_{1})$.  These knots are indeed closed by identifying $\partial (a_{1})$ with $\partial(-a_{2})$ and $\partial (b_{1})$ with $\partial (-b_{2})$ using the 1-handle attachment along $p_{1}$ and $p_{2}$.
\ee

Figure \ref{Fig:2CyclicCover} shows the completed diagram.  Performing the Weinstein handle attachments described above provides a Weinstein cobordism with concave boundary $\sqcup^{2}(L(2,1),\xi_{std})$ and convex boundary $(L(2,1),\xi_{std})_{C,2}$.

\begin{figure}[h]
	\begin{overpic}[scale=.7]{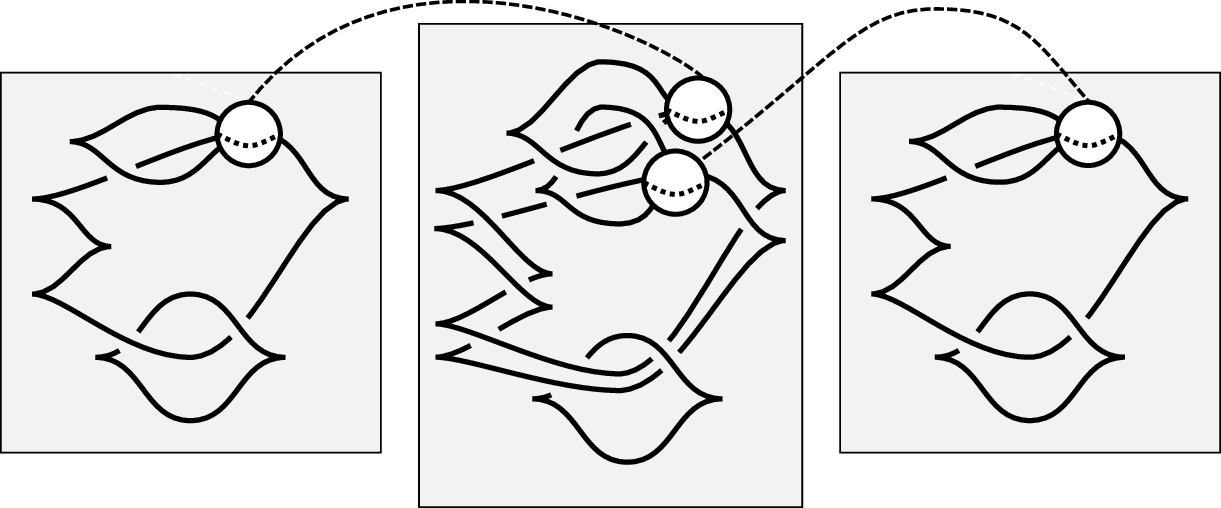}
    \put(24.2,11.5){$-1$}
    \put(93,11.5){$-1$}
    \put(60.5,8){$-1$}
    \end{overpic}
    \vspace{4.5mm}
	\caption{A Kirby diagram for the 3-fold cyclic branched cover of $(L(2,1),\xi_{std})$ branched over $T$.}
    \label{Fig:3CyclicCover}
\end{figure}

Now we will briefly describe the Weinstein cobordism with concave boundary $\sqcup^{3}(L(2,1),\xi_{std})$ and convex boundary the triple branched cover $(L(2,1),\xi_{std})_{C,3}$ over the transverse knot $T$.  This time we start with 3 copies of $(L(2,1),\xi_{std})$.  One of the copies contains one copy $\Sigma_{1}$ of $\Sigma$, another contains two copies $\Sigma_{2}$ and $\Sigma_{2}'= \Flow^{-\epsilon}_{\partial_{z}}(\Sigma_{2})$ of $\Sigma$, and the last contains a single copy $\Sigma_{3}$ of $\Sigma$.  Here, $\epsilon$ is an arbitrarily small positive constant.  The proof of Theorem \ref{Thm:Branched} indicates that we can describe the branched cover by performing two Liouville connect sums; the first identifying $\Sigma_{1}$ with $\Sigma_{2}$, while the second identifies $\Sigma_{2}'$ with $\Sigma_{3}$.  Again, by following the proof of Theorem \ref{Thm:Cobordism}(2) we obtain a Weinstein handle decomposition of the associated cobordism as in the case of the double-branched cover, described above.  The completed diagram is shown in Figure \ref{Fig:3CyclicCover}.

\subsection{Exact cobordisms which are not Weinstein}\label{Sec:NotWeinstein}

In this section we construct high-dimensional Liouville domains which do not admit Weinstein structures.  The examples below serve to illustrate that the cobordism described in Theorem \ref{Thm:Cobordism} is not always Weinstein.

By Theorem \ref{Thm:WHandle}, a connected Weinstein domain $\Sdom$ of dimension greater than 2 must have connected boundary.  The first examples of connected 4-dimensional Liouville domains whose boundaries are disconnected (and therefore, are not Weinstein) were discovered by McDuff \cite{McDuff}.  The examples were obtained by modifying the contact form $-\lambda_{can}$ on the unit cotangent disk bundle $\disk^{*}S_{g}$ of a closed, oriented, genus $g>1$ surface away from a neighborhood $N(S_{g})$ of its zero section, creating a Liouville 1-form on $[-1,1]\times S^{*}S_{g}\cong \disk^{*}S_{g}\setminus \Int(N(S_{g}))$.  In \cite{Geiges:Disconnected}, Geiges generalized this construction, listing a set of conditions associated to a fixed odd-dimensional manifold $M$ which guarantee the existence of a Liouville 1-form on the product $[-1,1]\times M$ and providing examples in the case $\dim (M)=5$.  Examples of Liouville 1-forms of manifolds of the form $[-1,1]\times M$ -- with $\dim (M)$ an arbitrary positive odd integer -- are described by Massot, Niederkr\"{u}ger, and Wendl in \cite[Theorem C]{MNW12}.

The cobordism associated to a Liouville connect sum, performed on the convex boundary of a symplectic cobordism $(W,\omega)$, either preserves or decreases the number of convex boundary components of $(W,\omega)$.  To establish that certain cobordisms constructed using Theorem \ref{Thm:Cobordism} are not Weinstein, we can use singular homology instead of numbers of boundary components.

\begin{lemma}\label{Lemma:Homology}
Suppose that $\Ldom$ is a connected $(2n+2)$-dimensional, Weinstein cobordism with concave boundary $\Mxi$.  Then the inclusion map of $M$ into $W$ induces isomorphisms
\begin{equation*}
H_{k}(W;\mathbb{Z})\cong H_{k}(M;\mathbb{Z}) \quad \forall\; k>n+1.
\end{equation*}
In particular, if $\Ldom$ is a $(2n+2)$-dimensional Weinstein domain, then $W$ has the homotopy type of an $(n+1)$-dimensional CW complex and so $H_{k}(W;\mathbb{Z})=0$ for all $k>n+1$.
\end{lemma}

\begin{proof}
This follows immediately from Theorem \ref{Thm:WHandle} and the application of a Mayer-Vietoris sequence associated to a Weinstein handle attachment.
\end{proof}

\begin{thm}
Let $\Mxi$ be a $(2n+1)$-dimensional contact manifold where $n>1$ and let $T$ be a closed $(2n-1)$-dimensional manifold.  Suppose that $[-1,1]\times T$ has a Liouville 1-form $\beta$ and that there are disjoint Liouville embeddings $i_{1},i_{2}:([-1,1]\times T,\beta)\rightarrow \Mxi$.  Suppose further that
\begin{equation*}
i_{1} [ T ]=i_{2} [ T ] \quad\text{in}\quad H_{2n-1}(M,\mathbb{Z}).
\end{equation*}
Then the exact cobordism $\Ldom$ of Theorem \ref{Thm:Cobordism} associated to the Liouville connect sum of $\Mxi$ along the $i_{j}([-1,1]\times T)$ ($j=1,2$) is not Weinstein.
\end{thm}

\begin{proof}
Again, we apply a Mayer-Vietoris sequence to the handle-attachment pair $([\half,1]\times M,H_{[-1,1]\times T})$.  It follows that $H_{2n}(W;\mathbb{Z})\cong H_{2n}(M;\mathbb{Z})\oplus\mathbb{Z}$, completing the proof by Lemma \ref{Lemma:Homology}.

The extra $\mathbb{Z}$ factor in $H_{2n}(W;\mathbb{Z})$ can be explicitly described as follows:  Let $X$ be an oriented $2n$-dimensional cobordism with boundary $T \cup -T$ such that there is a map $i_{X}:X\rightarrow M$ with $i_{X}(\partial X)=i_{1}(T)\cup i_{2}(T)$.  Let $Y=[-1,1]\times T$ and let $i_{Y}: Y\rightarrow H_{[-1,1]\times T}=[-1,1]\times N([-1,1]\times T)$ be the map $(\theta,x)\mapsto (\theta,0,0,x)$.  Here we are using the coordinates on a standard neighborhood as described in the proof of Theorem \ref{Thm:Cobordism}.  Then $i_{X}[X]+i_{Y}[Y]\in H_{2n}(W,\mathbb{Z})$ generates the desired homology class.
\end{proof}

Now we will provide some concrete examples.

\begin{ex}
Let $\Sdom$ be a $2n$-dimensional Liouville domain for which $\Sigma=[-1,1]\times T$ for some closed, smooth $(2n-1)$-dimensional manifold $T$.  Here we require that $n>1$.  As pointed out in the discussion following the statement of Theorem \ref{Thm:Monoids}, $(\disk^{2}\times\Sigma,\lambda_{std}+\beta)$ is a Liouville domain whose boundary is a contact manifold admitting an open book determined by the pair $((\Sigma,\beta),\Id_{\Sigma})$. Here the boundary manifold is diffeomorphic to $S^{2}\times T$. Under this identification the binding of the open book is given by the product of the north and south poles of the sphere with $T$ and each page is given by the product of a longitudinal line with $T$.

For each $\theta\in S^{1}=[0,2\pi]/\sim$ denote by $\Sigma_{\theta}$ the page of this open book corresponding to $\theta$, with a collar neighborhood of its boundary removed.  Then each $\Sigma_{\theta}$ is a Liouville hypersurface in the contact manifold $\partial(\disk^{2}\times\Sigma,\lambda_{std}+\beta)$.

Choose an even natural number $2g$ and let $\sigma$ be a permutation which fixes no element of the set $\{1,\dots,2g \}$ and squares to the identity, $\sigma^{2} = \Id$.  Perform $g$ Liouville connect sums along $\partial(\disk^{2}\times\Sigma,\lambda_{std}+\beta)$, identifying each $\Sigma_{\pi j/g}$ with $\Sigma_{\pi\sigma (j)/g}$. Using the handle attachment described in the proof of Theorem \ref{Thm:Cobordism} we can view this new contact manifold as the boundary of a Liouville domain whose underlying manifold is diffeomorphic to $H_{g}\times T$ for a genus-$g$ handlebody $H_{g}$.
\end{ex}

\section{Contact $(1/k)$-surgeries on contact manifolds of arbitrary dimension}\label{Sec:ContactSurgery}

The purpose of this section is to define contact $(1/k)$-surgery and describe some of its basic properties as stated in the introduction.  For the purpose of motivation, we begin by giving a brief overview of contact surgery on contact 3-manifolds as defined in \cite{DG:Surgery}.

\subsection{The case $\dim(M)=3$}

Suppose that $\Mxi$ is a contact 3-manifold containing a Legendrian knot $L$ and let $k$ be any integer.  Then $L$ admits a tubular neighborhood $N(L)$ in $\Mxi$ of the form $N(L)=[-\epsilon,\epsilon]\times \disk^{*}S^{1}=[-\epsilon,\epsilon]\times S^{1}\times[-1,1]$ on which $\xi=\ker(dz-\lambda_{can})$. To perform \emph{contact $(1/k)$-surgery} on $L\subset\Mxi$, remove $N(L)$ from $\Mxi$ and glue it back using a map which is boundary-relative isotopic to $-k$ Dehn twists along $\{\epsilon\}\times\disk^{*}S^{1}$ and isotopic to the identity on the remainder of the boundary of $N(L)$.  The boundary relative isotopy class of the Dehn twists may be chosen in such a way that the surgered manifold naturally admits a contact structure, which depends only on $\Mxi$, the Legendrian isotopy class of $L$ in $\Mxi$, and the integer $k$ \cite[Proposition 7]{DG:Surgery}.\footnote{Contact $(1/k)$-surgery along a Legendrian $L$ is so-called, as in dimension $3$ it is topologically a Dehn surgery with coefficient $(1/k)$ when computed using the longitudinal framing determined by the contact structure $\xi$.}

\begin{figure}[h]
	\begin{overpic}[scale=.7]{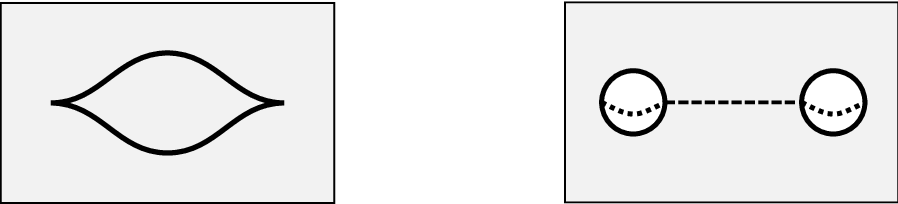}
        \put(48,10){$=$}
        \put(4,13){$+1$}
    \end{overpic}
	\caption{A Kirby diagram depicting the statement of Theorem \ref{Thm:SurgeryOverview}(3).  We emphasize that the diagram on the right only represents a 3-manifold, not a Weinstein 4-manifold.}
    \label{Fig:PlusOneSurgery}
\end{figure}

\begin{thm}\label{Thm:SurgeryOverview}
Let $\Mxi$ be a connected contact 3-manifold.
\be
\item Performing a contact $(-1)$-surgery along a Legendrian knot $L\subset\Mxi$ gives the same contact manifold as is obtained by performing a $4$-dimensional Weinstein $2$-handle attachment along $L$.
\item For any Legendrian knot $L\subset\Mxi$, the contact manifold described by performing a contact $(1/p)$-surgery on $L$, followed by a $(1/q)$-surgery on a push-off of $L$ is equivalent to the contact manifold described by performing a $(1/(p+q))$-surgery on $L\subset\Mxi$.
\item Performing a contact $(+1)$-surgery on a standard Legendrian unknot in $\Mxi$ produces the contact connect sum of $\Mxi$ with $(S^{1}\times S^{2} ,\xi_{std})=\partial(\disk^{*}S^{1}\times\disk^{2},-\lambda_{can}+\lambda_{std})$.  See Figure \ref{Fig:PlusOneSurgery}.
\item Performing a contact $(\half)$-surgery on a standard Legendrian unknot in $\Mxi$ yields $M$ equipped with an overtwisted contact structure.
\item Performing a contact $(+1)$-surgery on a stabilized Legendrian knot in $\Mxi$ produces an overtwisted contact manifold.
\item There is a Legendrian link $L=L^{+}\cup L^{-}$ in $\Sthree$ such that performing $(+1)$-surgery along the components of $L^{+}$ and $(-1)$-surgery along the components of $L^{-}$ yields $\Mxi$.
\ee
\end{thm}

For proofs of the above statements, we refer the reader to the exposition \cite[\S 11.2]{OzbSt:SteinSurgery} and the references therein.  An alternate proof of item (6) can be found in \cite{Avdek}.  Statements (1-4) in the above theorem can also be viewed as special cases of Theorem \ref{Thm:SurgeryProperties}, below.

\subsection{Generalized Dehn twists}\label{Sec:DehnTwist}

The essential ingredient in our definition of contact $(1/k)$-surgery is the \emph{generalized Dehn twist}, first discovered in the context of symplectic geometry by Arnol'd in \cite{Arnold} and further popularized in the work of Seidel \cite{Seidel1, Seidel2}.

Identify the cotangent bundle of the $n$-sphere with the set
\begin{equation*}
T^{*}S^{n}=\{(u,v)\in \mathbb{R}^{n+1}\times \mathbb{R}^{n+1}: \|u\|=1, \langle u,v\rangle=0\}.
\end{equation*}
Here $\langle\ast,\ast\rangle$ denotes the standard inner product on $\mathbb{R}^{n+1}$ and we may view $u, v$ as the real and imaginary parts of vectors in $\mathbb{C}^{n+1}$.  We consider $T^{*}S^{n}$ as a symplectic manifold with the canonical symplectic form $-d\lambda_{can}=\sum_{1}^{n+1} du_{i}\wedge dv_{i}$.  In this model situation we can write $-\lambda_{can}=-\sum_{1}^{n+1} v_{i}du_{i}$.  Fix an arbitrarily small positive constant $\epsilon<1$, and let $f:[0,\infty)\rightarrow\mathbb{R}$ be a function such that
\be
\item $f(0)=\pi$ and all derivatives vanish on a neighborhood of $0$,
\item $f$ is non-decreasing, and
\item $f(x)=2\pi$ for all $x\geq \epsilon$.
\ee

Now define the diffeomorphism on the complement of the zero-section of $T^{*}S^{n}$
\begin{equation*}
\widehat{\tau}_{n}:(T^{*}S^{n} \setminus S^{n}) \rightarrow (T^{*}S^{n} \setminus S^{n})
\end{equation*}
determined by the formula
\begin{equation*}
\widehat{\tau}_{n}(u,v)=\big(\cos\circ f(\| v \|)\cdot u + \sin\circ f(\| v \|)\cdot \frac{v}{\| v  \|}, -\|v\|\sin\circ f(\| v \|)\cdot u + \cos\circ f(\| v  \|)\cdot v\big)
\end{equation*}
using coordinates $u, v$ on $\mathbb{C}^{n+1}$.

By our conditions on the function $f$, $\widehat{\tau}_{n}$ extends to a diffeomorphism $\tau_{n}$ of $T^{*}S^{n}$ which extends smoothly over the zero-section. The restriction of $\widehat{\tau}_{n}$ to the zero-section $S^{n}$ is the antipodal map and the restriction to a collar neighborhood of $\partial\disk^{*}S^{n}\subset T^{*}S^{n}$ is the identity map by our assumption that $\epsilon<1$.  Hence we will view $\tau_{n}$ as an element of $\Diff^{+}(\disk^{*}S^{n},\partial\disk^{*}S^{n})$.\footnote{The reader may note that our choices of orientation on $T^{*}S^{n}$ and conditions defining the function $f$ above are the opposite of those often appearing in the literature.  However, the end result is the same.  See \cite[Remark 6.4]{Seidel2}.}

\begin{thm}
The diffeomorphism $\tau_{n}$ preserves $-d\lambda_{can}$ and its isotopy class in $\Symp((\disk^{*}S^{n}, -d\lambda_{can}),\partial\disk^{*}S^{n})$ is independent of the constant $\epsilon<1$ and the choice of function $f$.
\end{thm}

A proof may be found in \cite{Seidel1,Seidel2}.

\begin{defn}
We call any symplectomorphism which is boundary-relative symplectically isotopic to $\tau_{n}\in\Symp((\disk^{*}_{\epsilon},-d\lambda_{can}),\partial \disk^{*}_{\epsilon}S^{n})$ a \emph{generalized Dehn twist}.
\end{defn}

It is easy to see that the mapping $\tau_{1}$ coincides with the usual notion of a right-handed Dehn twist on an annulus when $n=1$.

\begin{ex}[\cite{Arnold}]\label{Ex:Arnold}
Consider the function $f:\mathbb{C}^{n+1}\rightarrow\mathbb{C}$ given by
\begin{equation*}
(z_{1},\dots,z_{n+1})\mapsto \sum_{1}^{n+1} z_{j}^{2}.
\end{equation*}
This function induces an open book decomposition of $S^{2n+1}=\partial\disk^{2n+2}$ whose binding is $f^{-1}(0)\cap S^{2n+1}$.  In this case, each page is diffeomorphic to $\disk^{*}S^{n}$ and the monodromy is given by a generalized Dehn twist.  This open book is compatible with the standard contact structure $\xi_{std}$ on $S^{2n+1}$.
\end{ex}

\subsection{Contact $(1/k)$-surgery}

We give two equivalent definitions of contact $(1/k)$-surgery.  The first will make it clear that our definition extends the 3-dimensional one described above, while the second will make it easier to prove some of its basic properties.  Afterwards we briefly discuss a subtlety in this definition, which is irrelevant when performing surgery on $3$-dimensional contact manifolds.

\subsubsection{First definition}

Our first definition of contact surgery uses convex gluing instructions as defined in Section \ref{Sec:GluingInstructions}.

Let $\Mxi$ be any $(2n+1)$-dimensional contact manifold containing a Legendrian sphere $L$.  The Weinstein neighborhood theorem for Legendrian submanifolds asserts that for any contact form $\alpha$ for $\Mxi$, we can find a ribbon $\Sigma=\disk^{*}L$ for which $\alpha|_{\disk^{*}L}=-\lambda_{can}$.  Consider a Liouville embedding $I:(\disk^{*}S^{n},-\lambda_{can})\rightarrow \Mxi$ whose image is the ribbon $\Sigma$ of $L$.

\begin{defn}\label{Def:Surgery}
To perform \emph{contact $(1/k)$-surgery} on $L$ with parameter $I$, remove a standard neighborhood $\mathcal{N}(\Sigma)$ of $\Sigma$ from $\Mxi$ and then reattach $\partial \mathcal{N}(\Sigma)$ to $\partial (M\setminus \Int(\mathcal{N}(\Sigma)))$ using the convex gluing instructions $(\tau^{-k},\Id_{\Sigma})$.
\end{defn}

\subsubsection{Second definition}

Our second definition of contact surgery uses Liouville connect sums between a given contact manifold and model contact manifolds provided by the following.

\begin{defn}\label{Def:TstarSnOB}
For each natural number $n>0$ define $\Mxink{n}{k}$ to be the $(2n+1)$-dimensional contact manifold determined by the open book for the pair $((\disk^{*}S^{n},-\lambda_{std}),\tau_{n}^{k})$.
\end{defn}

\begin{ex}\label{Ex:OpenBooks}
The smooth manifold underlying $\Mxink{n}{-1}$ is $S^{2n+1}$, although its contact structure is not $\xi_{std}$.  As shown in \cite{BK:Stabilize}, $\Mxink{n}{-1}$ is not symplectically fillable for any $n$.  This fact is well known in the case $n=1$, as $\Mxink{n}{-1}$ is overtwisted.  For a more explicit description of $\Mxink{n}{-1}$, see \cite[Example 5]{NP}. Theorem \ref{Thm:NegativeMonodromy} similarly addresses fillability of the $\Mxink{n}{k}$.

As pointed out in the discussion following the statement of Theorem \ref{Thm:WeakMonoids}, $\Mxink{n}{0}$ can be realized as the boundary of the Liouville domain $(\disk^{2}\times\disk^{*}S^{n},\lambda_{std}-\lambda_{can})$. By Example \ref{Ex:Arnold} $\Mxink{n}{1}$ is diffeomorphic to the standard contact sphere $(S^{2n+1},\xi_{std})$.

We claim that the contact manifold $\Mxink{n}{2}$ coincides with the canonical contact structure on the unit cotangent bundle $S^{*}S^{n+1}$.  By Example \ref{Ex:WeinsteinSphere}, we see that $(S^{*}S^{n+1},\xi_{can})$ can be realized by performing a Weinstein handle attachment along a Legendrian sphere -- the standard Legendrian unknot -- in $(S^{2n+1},\xi_{std})$. According to Example \ref{Ex:Arnold}, this Legendrian sphere can be realized as the zero section of $\disk^{*}S^{n}$ which we identify with one of the pages of the open book used to describe $\Mxink{n}{1}=\Sstd$.  According to Example \ref{Ex:LiouvilleWeinstein}, the surgered contact manifold $(S^{*}S^{n+1},\xi_{can})$ can be described by performing a Liouville connect sum on two disjoint copied of $\Mxink{n}{1}$, by identifying a page of one copy with a page of the other copy.  According to the proof of Theorem \ref{Thm:Monoids}, the resulting contact manifold is $\Mxink{n}{2}$.  Thus $\Mxink{n}{2}=(S^{*}S^{n+1},\xi_{can})$.
\end{ex}

Again, let $\Mxi$ be any $(2n+1)$-dimensional contact manifold containing a Legendrian sphere $L$ and consider a Liouville embedding $I:(\disk^{*}S^{n},-\lambda_{can})\rightarrow \Mxi$ whose image is the ribbon $\Sigma$ of $L$.  We have a fixed identification of $\disk^{*}S^{n}$ with a page of the open book $\Mxink{n}{k}$ for each $k$ as this manifold is defined constructively.

\begin{defn}\label{Def:SurgeryDef}
Define the contact manifold $\Mxi_{(L,I,k)}$ as the Liouville connect sum of 
\begin{equation*}
\Mxi\sqcup \Mxink{n}{-k}
\end{equation*}
using the Liouville embedding $I$.  We say that $\Mxi_{(L,I,k)}$ is obtained by \emph{contact (1/k)-surgery} on L with parameter $I$.
\end{defn}

This definition, together with Theorem \ref{Thm:Cobordism} allows us to associate a Weinstein cobordism to a contact $(1/k)$-surgery.  In Figure \ref{Fig:PlusOneCobordism} we provide a Kirby diagram for one such cobordism.

\begin{figure}[h]
	\begin{overpic}[scale=.7]{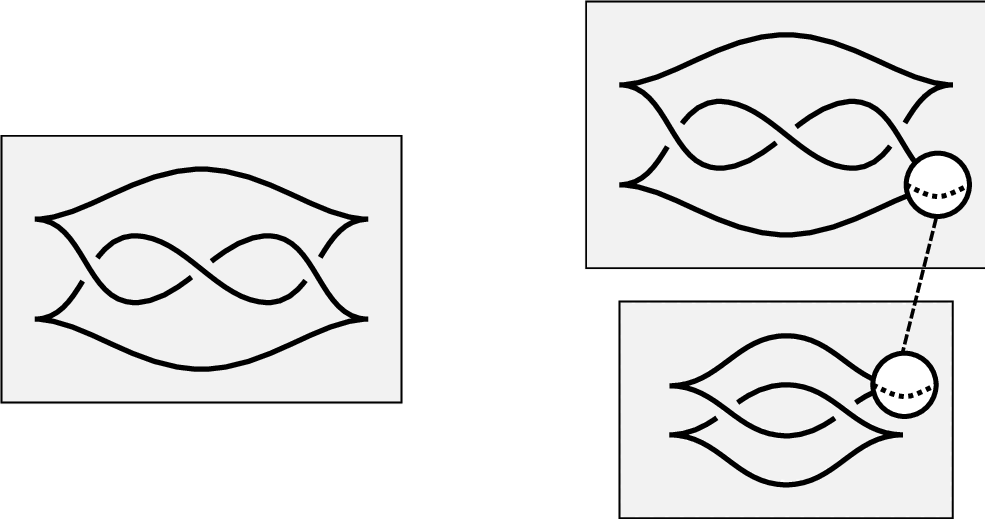}
        \put(48,25){$=$}
        \put(2,32){$1$}
        \put(62,46){}
        \put(67,4){$\frac{1}{2}$}
    \end{overpic}
	\caption{On the right we have a Weinstein cobordism whose concave boundary is a disjoint union of $\Sthree$ with the overtwisted 3-sphere $\Mxink{1}{-1}$.  The overtwisted sphere is given as the result of a $(\half)$-surgery on the standard Legendrian unknot in $\Sthree$.  The convex end of this cobordism is equivalent to the contact manifold on the left, described by a $(+1)$-surgery on a right-handed Legendrian trefoil in $\Sthree$.  The cobordism is decomposed into two Weinstein handle attachments; a $4$-dimensional $1$-handle attachment labeled by the 3-disks attached by a dotted line, and a $4$-dimensional $2$-handle determined by the unlabeled Legendrian knot which passes twice through the $1$-handle. This handle decomposition is given by applying the reasoning of Section \ref{Sec:Kirby} to the decomposition of $(\disk^{*}S^{1},-\lambda_{can})$ into a $2$-dimensional $0$-handle together with a single $2$-dimensional $1$-handle.}
    \label{Fig:PlusOneCobordism}
\end{figure}

\begin{prop}
Definitions \ref{Def:Surgery} and \ref{Def:SurgeryDef} are equivalent.
\end{prop}

\begin{proof}
As in Lemma \ref{Lemma:Heegaard}, we can present $\Mxink{n}{k}$ as two copies $N_{1}$ and $N_{2}$ of a standard neighborhood of $(\disk^{*}S^{n},-\lambda_{can})$ whose boundaries are identified via the convex gluing instructions $(\tau_{n}^{k},\Id_{\disk^{*}S^{n}})$.  Hence we can perform contact $(1/k)$-surgery as described in Definition \ref{Def:SurgeryDef} by removing $N_{2}$ from $\Mxink{n}{k}$, removing $\mathcal{N}(\Sigma)$ from $\Mxi$ and performing a convex gluing.  This is clearly equivalent to Definition \ref{Def:Surgery}.
\end{proof}

\subsubsection{Dependence on the parameter $I$}\label{Sec:ParametrizationDependenceOfTwists}

A priori the contact manifold $\Mxi_{(L,I,k)}$ depends on the parameterization $I$ -- not just the image $L$ of $I$.  For example, by combining Examples \ref{Ex:LiouvilleWeinstein} and \ref{Ex:WeinsteinSphere}, we see that different parameterization of a Legendrian sphere in $\Sstd$ can produce contact manifolds which may be inequivalent -- unit cotangent bundles of possibly exotic homotopy spheres. See \cite{Abouzaid:Exotic, EKS:Exotic} for related results and discussion pertaining to cotangent bundles of spheres. There are however certain cases in which we can guarantee that $\Mxi_{(L,I,k)}$ is independent of $I$.

\begin{prop}
Suppose that $H:\disk^{*}S^{n}\times[0,1]\rightarrow M$ is an isotopy of Legendrian spheres in $\Mxi$. That is, for each $t\in[0,1]$, $H(\ast,t):S^{n}\rightarrow M$ yields an embedded Legendrian sphere.  Writing $H(\ast,0)=I$, $H(\ast,1)=I'$, $I(S^{n})=L$, and $I'(S^{n})=L'$ we have that $\Mxi_{(L,I,k)}$ is contact-diffeomorphic to $\Mxi_{(L',I',k)}$.  Moreover, if $\Diff^{+}(S^{n})$ is path connected, then $\Mxi_{(L,I,k)}$ depends only on the submanifold $L$ and is independent of the chosen parameterization $I$.
\end{prop}

The first statement follows from the fact that we can write $H(\ast,t)=\phi_{t}\circ H(\ast,0)$ for an isotopy $\phi_{t}$ of $M$ which preserves $\xi$.  See \cite[Theorem 2.12]{Etnyre:KnotNotes}.  The existence of such an isotopy of $\Mxi$ then provides a homotopy of Liouville embeddings $(\disk^{*}S^{n}, -\lambda_{std})\rightarrow \Mxi$ by extending the original embedding using a Weinstein neighborhood theorem and applying the $\phi_{t}$. The second statement is essentially \cite[Lemma 6.2]{Seidel2} which states that isotopic parameterizations of the zero-section of $\disk^{\ast}S^{n}$ determine Dehn twists which are isotopic through a family of compactly supported symplectomorphisms.  The hypothesis of the second statement is known to hold true for $n=1,2,3,4,5,11,60$ and is known to not hold true for any other values of $n\leq 63$ \cite{Milnor:DiffTop}.

\subsection{Basic properties}

Now we outline some basic properties of contact $(1/k)$-surgery, showing that many of the results of Theorem \ref{Thm:SurgeryOverview} continue to hold in high dimensions.  In order to state our results, we must first define the \emph{standard Legendrian sphere} and \emph{Legendrian push-offs}.

\begin{defn}\label{Def:StandardLegendrian}
The \emph{standard Legendrian sphere in $(S^{2n+1},\xi_{std})$}, denoted $L_{std}$, is given by $S^{2n+1}\cap \Span(x_{1},\dots,x_{n+1})$ where we consider $\Sstd=\partial(\disk^{2n+2},\lambda_{std})$.  Let $\Mxi$ be a $(2n+1)$-dimensional contact manifold and identify $\Mxi$ with the contact connect sum of $\Mxi$ and $(S^{2n+1},\xi_{std})$, where the connect sum is performed outside of a tubular neighborhood of $L_{std}\subset S^{2n+1}$.  In this way, we view $L_{std}$ as a Legendrian sphere in $\Mxi=\Mxi\# (S^{2n+1},\xi_{std})$.  We say that a Legendrian sphere $L$ in $\Mxi$ is a \emph{standard Legendrian sphere} if it is Legendrian isotopic to $L_{std}\subset \Mxi$.
\end{defn}

According to the above definition, a standard Legendrian sphere in a contact 3-manifold is a Legendrian unknot with Thurston-Bennequin number equal to $-1$.  The sphere $L_{std}$ has a canonical parametrization given by its identification with the unit $n$-sphere in $\Span(x_{1},\dots,x_{n+1})$.

\begin{lemma}\label{Lemma:StandardConnectSum}
Performing contact $(1/k)$ surgery on a standard Legendrian sphere in some $(2n+1)$-dimensional contact manifold $\Mxi$ produces a contact connected sum of $\Mxi$ with $\Mxink{n}{1 - k}$.
\end{lemma}

\begin{proof}
Applying Definition \ref{Def:StandardLegendrian}, we may be view the surgery locus as living inside of $(S^{2n+1},\xi_{std})$. It therefore suffices to prove that contact $(1/k)$ surgery on the standard unknot in $(S^{2n+1}, \xi_{std})$ is contactomorphic to $\Mxink{n}{1 - k}$. The standard Legendrian sphere $L$ in $(S^{2n+1},\xi_{std})$ may be viewed as lying in a page of the open book decomposition which identifies $(S^{2n+1},\xi_{std})$ with $\Mxink{n}{1}$. As in this case, Definition \ref{Def:SurgeryDef} coincides exactly with the proof of Proposition \ref{Prop:OBSum}, and we see that the contact manifold obtained by performing $(1/k)$ surgery along $L$ will be supported by an open book decomposition with page $\disk^{\ast}S^{n}$ and monodromy is the product of the monodromies $\tau_{n}$ and $\tau_{n}^{-k}$.
\end{proof} 

\begin{defn}
Let $L\subset\Mxi$ be a Legendrian submanifold and identify a tubular neighborhood $N(L)$ of $L$ with $N(L)=[-\epsilon,\epsilon]\times \disk^{*}L$.  We say that a Legendrian submanifold of $(M\setminus N(L),\xi)$ is a \emph{push-off} of $L$ if it is Legendrian isotopic to $\{\epsilon\}\times L\subset M\setminus \Int(N(L))$.
\end{defn}

This notion of push-off clearly extends the usual definition of a push-off of a Legendrian knot in a contact 3-manifold.  A parametrization $I:S^{n}\rightarrow L$ of a Legendrian sphere in $\Mxi$ gives rise to a canonical parametrization of a push-off by $(\epsilon,I):S^{n}\rightarrow N(L)$.

\begin{thm}\label{Thm:SurgeryProperties}
Let $\Mxi$ be a connected $(2n+1)$-dimensional contact manifold.
\be
\item Performing a contact $(-1)$-surgery along any Legendrian sphere $L\subset\Mxi$ gives the same contact manifold as the one obtained by performing a $(2n+2)$-dimensional Weinstein $(n+1)$-handle attachment along $L$.
\item For any Legendrian sphere $L\subset\Mxi$, the contact manifold described by performing a contact $(1/p)$-surgery on $L$ with parameter, followed by a $(1/q)$-surgery on a push-off of $L$ (with its natural parameterization) is equivalent to the contact manifold described by performing a $(1/(p+q))$-surgery on $L\subset\Mxi$.
\item Performing a contact $(+1)$-surgery on a standard Legendrian sphere in $\Mxi$ with its natural parametrization produces the contact connect sum of $\Mxi$ with $(S^{n}\times S^{n+1},\xi_{std}):=\partial(\disk^{*}S^{n}\times\disk^{2},-\lambda_{can}+\lambda_{std})$.
\item Performing a contact $(\half)$-surgery on a standard Legendrian sphere in $\Mxi$ yields $M$ equipped with an algebraically overtwisted \cite{BN:Overtwisted} (and so not symplectically fillable) contact structure.
\ee
\end{thm}

\begin{proof}
The first assertion follows from a combination of Examples \ref{Ex:LiouvilleWeinstein} and \ref{Ex:Arnold}.

For the second: Suppose that we perform a $(1/p)$-surgery on $L\subset\Mxi$ using a Liouville embedding
\begin{equation*}
I:(\disk^{*}S^{n},-\lambda_{can})\rightarrow \Mxi.    
\end{equation*}
Write $\mathcal{N}(\Sigma)$ for a standard neighborhood of the image $\Sigma$ of $I$.  Legendrian isotop the push-off of $L$ into the interior of $\mathcal{N}(\Sigma)$ -- now considered as a subset of $\Mxi_{(L,I,p)}$ -- in the obvious fashion so that it is identified with the zero-section of $\disk^{*}S^{n}$ which we considered to be a page of $\Mxink{n}{-p}$ by alternately thinking of $\mathcal{N}(\Sigma)$ as a subset of $\Mxink{n}{-p}$.  Performing $(1/q)$-surgery on this sphere in $\Mxink{n}{p}$ amounts to adding $-k$ Dehn twists to the monodromy of the open book determining $\Mxink{n}{-p}$ and so results in $\Mxink{n}{-p-q}$.  Therefore the end result is a Liouville connect sum of $\Mxi$ and $\Mxink{n}{-p-q}$ along $\Sigma$ using the parametrization $I$.

The third and fourth statements follow immediately from Lemma \ref{Lemma:StandardConnectSum} together with the descriptions of $\Mxink{n}{0}$ and $\Mxink{n}{-1}$ appearing in Example \ref{Ex:OpenBooks}.
\end{proof}

\section{Applications to the study of Dehn twists}\label{Sec:TwistApplications}

For our final application of the Liouville connect sum we use contact $(1/k)$-surgery as a tool to study generalized Dehn twists.  We begin with a continuation of the previous section, proving Theorem \ref{Thm:NegativeMonodromy}.

Throughout, we will make use of the contact manifolds $\Mxink{n}{k}$ as described in Definition \ref{Def:TstarSnOB}.

\subsection{Negative monodromy and overtwistedness}

Now we provide an alternate proof and generalization of Theorem \ref{Thm:SurgeryProperties}(4). The following is a restatement of Theorem \ref{Thm:NegativeMonodromy} in the notation of the present section.

\begin{thm}
Let $k > 0$ be a natural number and let $\Mxink{n}{-k}$ be as in Definition \ref{Def:TstarSnOB}. Then $\Mxink{n}{-k}$ does not admit a symplectic filling with $(W, \omega)$ with $\omega|_{\pi_{2}(W)} = 0$. In particular, it does not admit an exact symplectic filling.
\end{thm}

Our proof of Theorem \ref{Thm:NegativeMonodromy} is a simple application of the following theorem of Eliashberg- Floer-Gromov-McDuff \cite{Eliashberg:Plumbing,McDuff}.

\begin{thm}\label{Thm:StandardFilling}
Let $n\geq 1$ be a positive integer and suppose that $(W,\omega)$ is a symplectic filling of $\Sstd$ for which $\omega$ integrates to zero over every embedded $2$-sphere in $W$.  Then $W$ is diffeomorphic to $\disk^{2n+2}$.
\end{thm}

Clearly exact symplectic manifolds contain no such symplectically embedded spheres.

\begin{proof}[Proof of Theorem \ref{Thm:NegativeMonodromy}]
Applying a Liouville connect sum to pages of the provided open books for 
\begin{equation*}
\Mxink{n}{k+1},\quad \Mxink{n}{-k}
\end{equation*}
produces an exact symplectic cobordism whose convex boundary is $\Mxink{n}{1} = \Sstd$ and whose concave boundary is $\Mxink{n}{k+1} \sqcup \Mxink{n}{-k}$. Denote the smooth manifold with boundary underlying this cobordism as $Y$.

We now describe a symplectic filling $(W, \omega_{W})$ of $\Mxink{n}{k+1}$. For $\epsilon \in S^{1} \subset \mathbb{C}$ let $\disk^{n+1}_{\epsilon}$ be the rotation of the Lagrangian disk $\disk^{n+1} \times \{0\} \subset \mathbb{R}^{n+1}\times\mathbb{R}^{n+1}\subset \mathbb{C}^{n+1}$ by $\epsilon$. Note that the boundary of each $\disk^{n+1}_{\epsilon}$ is a Legendrian sphere realized as the zero-section of a page $T^{*}S^{n}$ of the standard open book decomposition of $\Sstd$ described in Example \ref{Ex:Arnold} and that for $\epsilon_{1} \neq \epsilon_{2}$, $\disk^{n+1}_{\epsilon_{1}}$ and $\disk^{n+1}_{\epsilon_{1}}$ have a single intersection at $0\in\mathbb{C}^{n+1}$ which is transverse. We obtain $(W, \omega_{W})$ by performing a Weinstein handle attachment to $\partial \disk^{n+1}_{\epsilon_j}$ for $\epsilon_j=\zeta^{j}$, $j=0,\dots,k$ where $\zeta$ is a primitive $(k+1)$-th root of unity. Denote by $L_j$ the Lagrangian spheres in $(X, \omega_{X})$ given by the union of the $\disk^{n+1}_{\epsilon_{j}}$ with the core sphere of its associated Weinstein handle. Then each pair of distinct $L_j$ share a single transverse intersection. As $k+1 \geq 2$, we have at least two such Lagrangian spheres $L_1, L_2$.\footnote{We could just as easily described such a pair of such Lagrangian spheres via matching paths for a Lefschetz fibration description of $(W, \omega_{W})$ over the disk with generic fiber $T^{*}S^{n}$ and $k+2$ singular fibers, but want to avoid additional machinery and definitions. See for example \cite{AS04}. Likewise, we could see that $\Mxink{n}{k+1}$ is the boundary of a Brieskorn variety \cite{KK16} and appeal to classic homological computations to achieve the end result.}

To complete the proof, let $(W, \omega_{W})$ be an exact symplectic filling of $\Mxink{n}{-k}$ whose existence would contradict the statement of the theorem. Glue the cobordisms together to construct a manifold $Z = W \cup X \cup Y$ by identifying the concave boundary components of $Y$ with the convex boundaries of $X$ and $W$ in the obvious fashion. We thus have an exact symplectic filling of $\partial Z = \Sstd$. By considering $L_1, L_2$ as submanifolds of $Z$, we would see that its middle dimensional homology $H_{n+1}(Z)$ has non-zero intersection form which contradicts the fact that $Z$ must be diffeomorphic to a disk by Theorem \ref{Thm:StandardFilling}.
\end{proof}

\begin{rmk}
Here is an alternate proof of Theorem \ref{Thm:NegativeMonodromy} in the case $k > 1$, utilizing \cite{BK:Stabilize} which asserts that $\Mxink{n}{-1}$ has vanishing contact homology for all $n$. This implies that $\Mxink{n}{-1}$ is not exactly fillable by Theorem \ref{Thm:HCSummary}.

As in the proof above, we use the Liouville connect sum to construct a cobordism $W$ whose convex boundary is $\Mxink{n}{-1}$ and whose concave boundary is $\Mxink{n}{k+1} \sqcup \Mxink{n}{-k}$. The contact homology of the convex end of $W$ is $HC_{\ast}\Mxink{n}{k-1} \otimes HC_{\ast}\Mxink{n}{-k}$ by Theorem \ref{Thm:HCSummary}. By the Weinstein fillability of $\Mxink{n}{k-1}$, its contact homology is non-zero and we have that the contact homology of the concave boundary of $W$ is zero if and only if $HC_{\ast}\Mxink{n}{-k}$ is zero.

By the facts that $HC\Mxink{n}{-1} = 0$ and the cobordism $W$ induces an algebra homomorphism between contact homologies, $HC_{\ast}\Mxink{n}{-k}$ must also be zero. Hence $\Mxink{n}{-k}$ is not Liouville fillable.
\end{rmk}

Following classical results in $3$-dimensional contact topology and the above theorem, we ask how contact $(1/k)$-surgery could be used to produce more examples of non-symplectically fillable contact manifolds in arbitrary dimension or even be used to characterize overtwistedness.

In \cite[\S 4.1]{EES} a notion of ``stabilized Legendrian sphere'' is described, in analogy with the usual notion of a stabilized Legendrian knot in a contact 3-manifold.  Certain stabilized spheres, called \emph{loose Legendrian spheres} are classified (up to Legendrian isotopy) by homotopy theoretic data \cite{Murphy}.  Based on Theorem \ref{Thm:SurgeryOverview}(5) the fact that these spheres have trivial holomorphic curve invariants -- see \cite[Proposition 4.8]{EES} and \cite[\S 8]{Murphy} -- we present the following as a conjectured analogue of Theorem \ref{Thm:SurgeryOverview}(5).

\begin{conj}\label{Conj:Loose}
Let $L\subset \Mxi$ be a loose Legendrian sphere in a contact manifold of dimension greater than three.  A contact manifold obtained by performing a contact $(+1)$-surgery along $L$ is ``overtwisted''.
\end{conj}

Here ``overtwistedness'' of a contact manifold $\Mxi$ of dimension greater than three can be taken to mean -- or at least imply -- any of the following conditions:
\be
\item $\Mxi$ does not admit a weak symplectic filling.
\item Every contact form on $\Mxi$ has a contractible Reeb orbit.
\item The contact homology of $\Mxi$ is zero for any choice of coefficient system.
\item $\Mxi$ contains a plastikstufe as described in \cite{Niederkruger:PS}.
\item The contact structure $\xi$ is determined by qualitative and homotopical data as in \cite{Eliashberg:OT}.
\ee

\subsection{Squares of smooth Dehn twists}\label{Sec:TwistSquared}

In this section, we continue our study of symplectic Dehn twists, proving Theorem \ref{Thm:SquareClassification}.

\begin{lemma}\label{Lemma:NonZeroSquare}
Suppose that $n\neq 2,6$.  Then $\tau_{n}^{2}$ is not isotopic to the identity mapping in $\Diff^{+}(\disk^{*}S^{n},\partial\disk^{*}S^{n})$.
\end{lemma}

\begin{proof}
If $\tau^{2}_{n}$ is isotopic to the identity in $\Diff^{+}(\disk^{*}S^{n},\partial\disk^{*}S^{n})$, then the smooth manifolds underlying $\Mxink{n}{0}$ and $\Mxink{n}{2}$ are diffeomorphic.  Therefore, according to Example \ref{Ex:OpenBooks} it suffices to show that the unit cotangent bundle of $S^{n+1}$ is not diffeomorphic to $S^{n}\times S^{n+1}$ for $n\ne 0,2,6$.  We observe that these spaces are not homotopy equivalent.

If $n$ is odd, then $H^{n}(S^{*}S^{n+1};\mathbb{Z})=\mathbb{Z}/2\mathbb{Z}$ -- as can be computed using a cohomological Gysin sequence from the fact that $\chi(S^{n+1})=2$ -- while $H^{\ast}(S^{n}\times S^{n+1},\mathbb{Z})$ has no torsion.  In the event that $n$ is even, then $S^{*}S^{n+1}$ is homotopy equivalent to $S^{n}\times S^{n+1}$ if and only if $n+1=1,3,7$ as can be seen by combining results of Adams \cite[Theorem 1(b)]{Adams} and James-Whitehead \cite[Theorem 1.12]{James}.  Specifically, \cite[Theorem 1 (b)]{Adams} asserts that there is a map $S^{2n-1}\rightarrow S^{n}$ with Hopf invariant equal to one if and only if $n=1,2,4,8$ while \cite[Theorem 1.12]{James} asserts that $S^{*}S^{n+1}$ is homotopy equivalent to $S^{n}\times S^{n+1}$ if and only if there is an element in $\pi_{2n-1}(S^{n})$ with Hopf invariant equal to one.
\end{proof}

\begin{lemma}\label{Lemma:SquareToZero}
For $n=2,6$ the square of the generalized Dehn twist $\tau^{2}_{n}$ is isotopic to the identity mapping in $\Diff^{+}(\disk^{*}S^{n},\partial\disk^{*}S^{n})$.
\end{lemma}

\begin{proof}
Our proof is lifted from \cite[Lemma 6.3]{Seidel2} where the case $n=2$ is established.  The mechanism underlying the proof is the existence of an almost complex structure on $S^{2}$.  Our only contribution is the observation that $S^{6}$ also admits an almost complex structure which preserves the standard Riemannian metric as is determined by the cross product on the imaginary octonians. Thus we suppose that $S^{n}$ is a sphere equipped with an almost complex structure $J$ compatible with the standard round metric.

Let $u\in S^{n}$.  Then $J_{u}:T_{u}S^{n}\rightarrow T_{u}S^{n}$ determines an element $\mathfrak{j}_{u}$ of the Lie algebra $\mathfrak{so}(n+1)$ of $SO(n+1)$ as follows.  Using the standard metric on $TS^{n}$ provided by the natural inclusion of $S^{n}$ into $\mathbb{R}^{n+1}$, identify the $(n-1)$-sphere of unit length vectors in $T_{u}S^{n}$ with the $(n-1)$-sphere of points in $S^{n}$ which are orthogonal to $u$ when considered as vectors in $\mathbb{R}^{n+1}$.  In this way we can see that $J_{u}$ generates a circle subgroup of the subgroup of transformations in $SO(n+1)$ which fix the point $u$.  Indeed, for each $\theta\in S^{1}=[0,2\pi]/\sim$ we can consider that map $e^{\theta\cdot J_{u}}:T_{u}S^{n}\rightarrow T_{u}S^{n}$ as a map $S^{n}\rightarrow S^{n}$ fixing $u$.  Denote by $\mathfrak{j}_{u}\in\mathfrak{so}(n+1)$ the infinitesimal generator of this action.

Similarly, if $v\in T^{*}_{u}S^{n}$ is a non-zero cotangent vector then there is an associated vector $\mathfrak{v}_{u}\in \mathfrak{so}(n+1)$.  Denote by $v^{*}$ the associated dual vector in $T_{u}S^{n}$, which we will consider as a vector in $\mathbb{R}^{n+1}$.  Define $\mathfrak{v}_{u}$ to be the infinitesimal generator of the $SO(n+1)$-circle action on $S^{n}$ which rotates the oriented plane $\Span(u,\frac{1}{\| v^{*} \|}v^{*})$ counterclockwise and fixes the orthogonal complement of this plane.

In the notation of Section \ref{Sec:DehnTwist}, we can use the above definitions to express the Dehn twist $\tau_{n}$ as
\begin{equation*}
\tau_{n}(u,v)=\bigr{(}e^{f(\| v \|)\cdot \mathfrak{v}_{u}}u,e^{f(\| v \|)\cdot \mathfrak{v}_{u}}v\bigr{)}
\end{equation*}
for each pair $(u,v)$ satisfying $v\ne 0$.  Here $f$ is the function described in Section \ref{Sec:DehnTwist}.  For points of the form $(u,0)\in \disk^{*}S^{n}$, $\tau_{n}(u,0)=(-u,0)$.  Using the above formula we can write
\begin{equation*}
\tau_{n}^{2}(u,v)=\bigr{(}e^{2f(\| v \|)\cdot \mathfrak{v}_{u}}u,e^{2f(\| v \|)\cdot \mathfrak{v}_{u}}v\bigr{)}
\end{equation*}
for pairs $(u,v)\in\disk^{*}S^{n}$ satisfying $v\ne 0$, and $\tau_{n}^{2}(u,0)=(u,0)$ for each $(u,0)\in S^{n}\subset\disk^{*}S^{n}$.

Consider the $[0,1]$-family of diffeomorphisms of $\disk^{*}S^{n}$ given by the formula
\begin{equation*}
\Phi_{t}(u,v)=\bigr{(}e^{2f(\| v \|)\cdot \big{(}(1-t)\mathfrak{j}_{u}+t\mathfrak{v}_{u}\big{)}}u,e^{2f(\| v \|)\cdot \big{(}(1-t)\mathfrak{j}_{u}+t\mathfrak{v}_{u}\big{)}}v\bigr{)}
\end{equation*}
for $v\ne0$ and $\Phi_{t}(u,0)=(u,0)$ for all $t\in[0,1]$.  Note that $\Phi_{1}=\tau^{2}_{n}$, so that $\Phi_{t}$ provides an isotopy from $\tau_{n}^{2}$ to the diffeomorphism
\begin{equation*}
\Phi_{0}(u,v)=(u,e^{2f(\| v \|)\mathfrak{j}_{u}}v)
\end{equation*}
in such a way that $\Phi_{t}$ restricts to the identity mapping along the zero-section and boundary of $\disk^{*}S^{n}$ for all $t\in[0,1]$.  To complete the proof, consider the isotopy
\begin{equation*}
\Psi_{t}(u,v)=(u,e^{t2f(\| v \|)\mathfrak{j}_{u}}v)
\end{equation*}
which interpolates between $\Phi_{0}$ and the identity mapping in $\Diff^{+}(\disk^{*}S^{n},\partial\disk^{*}S^{n})$.
\end{proof}

The above lemmas combine to prove Theorem \ref{Thm:SquareClassification}.

\subsection{Squares of symplectic Dehn twists and exotic contact spheres}

With Theorem \ref{Thm:SquareClassification} established, we study the contact manifolds $\Mxink{n}{2k+1}$ to complete the proof of Theorem \ref{Thm:TauSquared}. Throughout this section, unless stated otherwise, we use $n$ to denote either $2$ or $6$.

\begin{thm}\label{Thm:FiveSphere}
For $k$ a non-zero integer and $n = 2, 6$, the smooth manifold underlying $\Mxink{n}{2k + 1}$ is $S^{2n+1}$.  However, this contact manifold is not contact-diffeomorphic to $(S^{2n+1},\xi_{std})$.
\end{thm}

As in Theorem \ref{Thm:NegativeMonodromy}, this is another simple application of Theorem \ref{Thm:StandardFilling}.

\begin{proof}
For the first statement, the monodromy of the open book decomposition underlying $\Mxink{n}{2k+1}$ is smoothly isotopic to $\tau_{n}$. Therefore $M_{n, 2k+1}$ is diffeomorphic to the manifold determined by the open book with page $\disk^{*}S^{n}$ and monodromy $\tau_{n}$, which is exactly $S^{2n+1}$.

For $k < 0$, the fact that $\Mxink{n}{2k+1}$ is not contact-diffeomorphic to $\Sstd$ follows from the fact that $\Sstd$ is exactly fillable while $\Mxink{n}{2k+1}$ is not by Theorem \ref{Thm:NegativeMonodromy}.

The case $k > 0$ follows a similar argument. By the proof of Theorem \ref{Thm:NegativeMonodromy}, we see that $\Mxink{n}{2k+1}$ admits a Weinstein filling of whose Euler characteristic is $1+2m\ne 1$. However, by Theorem \ref{Thm:StandardFilling}, any exact filling of $\Sstd$ must have Euler characteristic equal to $1$.
\end{proof}

\begin{proof}[Proof of Theorem \ref{Thm:TauSquared}]
If $\tau^{2}_{n}\in \Symp((\disk^{*}S^{n},-d\lambda_{can}),\partial \disk^{*}S^{n})$ was isotopic to the identity mapping, then for $k \neq 0$, the contact manifold $\Mxink{n}{2k+1}$ would be contact-diffeomorphic to $(S^{2n+1},\xi_{std})$ by Theorem \ref{Thm:GirCor} contradicting Theorem \ref{Thm:FiveSphere}. 
\end{proof}

\end{document}